\documentclass[A4,10pt]{article}
\usepackage[margin=1in,dvips]{geometry}

\pdfoutput=1 
\usepackage[dvipdfmx]{graphicx}

\usepackage{color, soul}
\usepackage{amsmath, amsthm, slashed}
\usepackage{amssymb}

\usepackage{marvosym}
\usepackage{rotating}
\usepackage{mathrsfs}
\usepackage{upgreek}

\usepackage{verbatim}
\usepackage[affil-it]{authblk}
\usepackage{rotating}
\usepackage{accents}
\usepackage{harmony, slashed}
\usepackage{tikz}
\usepackage{dsfont}
\usepackage{makeidx}
\usepackage{bbm}

\usepackage{tocloft}

\usepackage{thmtools}
\usepackage{thm-restate}

\usepackage{comment}

\usepackage{enumerate}

\usepackage[colorlinks=true,linkcolor=black,citecolor=blue]{hyperref}

\usepackage{cleveref}

\newtheorem{theorem}{Theorem}[section]

\newtheorem*{conjecture*}{Conjecture}
\newtheorem*{theorem*}{Theorem}

\newtheorem{proposition}{Proposition}[subsection]
\newtheorem{definition}[proposition]{Definition}
\newtheorem{corollary}[proposition]{Corollary}
\newtheorem*{corollary*}{Corollary}
\newtheorem{lemma}[proposition]{Lemma}
\newtheorem{remark}[proposition]{Remark}
\newtheorem{notation}[proposition]{Notation}

\numberwithin{equation}{section}

\newcommand{\Lie}{{\mathcal{L}}}

\newcommand{\der}{\nabla}

\newcommand{\les}{\lesssim}
\newcommand{\bea}{\begin{eqnarray}}
\newcommand{\eea}{\end{eqnarray}}

\newcommand{\derm}{ { \der^{(\bf{m})}} }
\newcommand{\rderm}{   {\mbox{$\nabla \mkern-13mu /$\,}^{(\bf{m})} }   }

\newcommand{\eps}{{\varepsilon}}\newcommand{\R}{{\mathbb R}}

\newcommand{\E}{{\cal E}}

\newcommand{\la}{\langle}\newcommand{\si}{\sigma}\renewcommand{\b}{\beta}

\def\a{\alpha}\def\ga{\gamma}\def\de{\delta}
\def\bm{\left( \begin{array}{cc}}
\def\endm{\end{array}\right)}\newcommand{\eq}{\end{equation}}
\def\a{\alpha}\def\b{\beta}
\def\ga{\gamma}\def\de{\delta}\def\pa{\partial}

\def \rectangle#1#2{\hbox{\vrule\vbox to #2 {\hrule\hbox to #1{\hfil}\vfil\hrule}\vrule}}

\def\a{\alpha}\def\b{\beta}\def\ga{\gamma}
\def\de{\delta}\def\pa{\partial}

\def\pa{\partial}

\def\beaa{\begin{eqnarray*}}
\def\eeaa{\end{eqnarray*}}
\def\pa{\partial}
\def\a{{\alpha}}
\def\b{{\beta}}
\def\ga{\gamma}

\def\de{\delta}

\def\eps{\epsilon}
\def\la{\lambda}

\def\si{\sigma}

\def\g{{\bf g}}

\def\SSS{{\Bbb S}}

\def\nn{{\mathbb N}}

\def\R{{\mathbb R}}

\def\12{\frac{1}{2}}

\def\N{{\mathcal N}}

\def\bep{\begin{proposition}}
\def\eep{\end{proposition}}

\def\4{\frac{1}{4}}

\def\12{\frac{1}{2}}

\def\N{\nn}

\def\bep{\begin{proposition}}
\def\eep{\end{proposition}}

\def\bm#1{\boldsymbol{#1}}

\def\build#1_#2^#3{\mathrel{\mathop{\kern 0pt#1}\limits_{#2}^{#3}}}
\def\4{\frac{1}{4}}

\def\<{\langle}
\def\>{\rangle}

\DeclareMathAlphabet\mathbfcal{OMS}{cmsy}{b}{n}

\makeatletter \@addtoreset{equation}{section}  \makeatother

\title{Dispersive estimates for a system of tensorial quasilinear wave equations satisfying the weak-null condition}

\author{Sari Ghanem}

\date{}

\begin{document}

\maketitle

\begin{abstract}
We establish both global existence and decay properties for solutions with small data for a general class of coupled system of tensorial quasilinear hyperbolic wave equations in three space dimensions, that covers the dynamical Einstein equations coupled to a class of non-linear matter sources that do \textit{not} satisfy the null condition of Christodoulou and Klainerman, and have \textit{new different non-linearities} than the one treated by Lindblad-Rodnianski, for which their celebrated seminal $L^\infty$-estimate does \textit{not} work, to the best of our knowledge. Global existence of solutions for a \textit{general} class of quasilinear wave equations satisfying the \textit{weak-null condition, with small initial data}, is largely an \textit{open problem} at present. There is no known theory to prove decay for the class of non-linear hyperbolic partial differential equations that we treat in this paper. We establish a technique based on novel \textit{decoupling} of the higher order energy estimates, at the level of the $L^2$-norm of the \textit{Lie derivatives} of the \textit{tangential components}, without involving all the other components, up to some good factor. This generalizes our previous results to include new non-linearities that are not present in the Einstein-Yang-Mills system in the Lorenz gauge.  
\end{abstract}

\setcounter{page}{1}
\pagenumbering{arabic}

\section{Introduction}\label{BrielfSummary}\

We study a general class of tensorial quasilinear hyperbolic wave equations in three space-dimensions, of the type
\bea\label{nonlinearsystemoftensorialwaveequations}
\begin{cases}   
\text{For all}   \quad  l \in \{1, 2, \ldots, N \in \N\} \; , \; k_l \in \N \,,    \\
g^{\a\b} \derm_{\a } \derm_\b  \Phi^{(l)}_{V_1V_2 \ldots V_{k_l}} = S^{(l)}_{V_1V_2 \ldots V_{k_l}} (\g \,,  \derm \g \,, \Phi^{(1)}\,, \ldots \,, \Phi^{(N)}\,, \derm \Phi^{(1)}\,, \ldots \,, \derm \Phi^{(N)} )\,,   \\
g^{\a\b} \derm_{\a } \derm_\b  g_{UV} = S^{(0)}_{UV}  (\g \,,  \derm \g  \,, \Phi^{(1)}\,, \ldots \,, \Phi^{(N)}\,,  \derm \Phi^{(1)}\,, \ldots \,, \derm \Phi^{(N)}  ) \; .\end{cases}
\eea
We shall make the dynamical system \eqref{nonlinearsystemoftensorialwaveequations} and the non-linearities, $S^{(l)}$ and $S^{(0)}$\,, more precise in \eqref{badtensorialcoupledwaveequation}-\eqref{goodtensorialcoupledwaveequation}. Here the metic $\g$ is a dynamical unknown Lorentzian metric on a dynamical space-time manifold ${\cal M}$\;, \`a priori unknown, and $\Phi^{(l)}$ are tensorial matter fields of arbitrary order, coupled to the metric and to the space-time in the evolution problem. We fix the notation that the field $\Phi^{(l)}$ is a $k_l$-tensor of arbitrary order. The vectors $U\;, V\;, V_1\;, V_2 \ldots V_{k_l}$ are in the tangent bundle of the unknown manifold ${\cal M }$\,. The metric $m$ and its covariant derivative $\derm$ are defined as follows.

\begin{definition}\label{definitionofMinkowskiandcovariantderivtaiveofMinkwoforafixedgivensystemofcoordinates}
Let  $(x^0, x^1, x^2, x^3)$ be a fixed system of coordinates, where we shall also write, sometimes, $x^0 = t$\;. We define $m$ be the Minkowski metric $(-1, +1, +1, +1)$ in this fixed system of coordinates. We then define $\derm$ be the covariant derivative associate to the metric $m$\;.

In other words, in our fixed system of coordinates $\{x^0, x^1,  x^2, x^3 \}$\;, we have
\bea
m_{00}&=&-1 \;,\qquad  m_{ii}=1\;\quad \text{for}\quad i=1\,, 2\,, 3\,, \quad \text{and}\quad m_{\mu\nu} =0\;\quad \text{for}\quad \mu\neq \nu \,. 
\eea
and 
 \bea
{ \der^{(\bf{m})}}_{ \frac{\pa}{\pa x^\mu}}  \frac{\pa}{\pa x^\nu} = 0 \; .
 \eea
\end{definition}

\begin{definition}\label{DefinitionofPhioasgminustheminkowski}
In our fixed system of coordinates $(t, x^1, x^2, x^3)$\,, we define
\beaa
\Phi^{(0)}_{\mu\nu} := g_{\mu\nu} - m_{\mu\nu} \; .
\eeaa 
Therefore, we fix the notation that the perturbation of the metric is $  \Phi^{(0)}$\,, and thus, $k_0 =2$ since the metric is a 2-tensor. Consequently, we can include the perturbation metric in the fields $\Phi^{(l)}$\,, and recast the whole system with $\Phi^{(l)}$\,, $l \in \{0, 1, \ldots, N \}$\,.
\end{definition}
For clarity of presentation, we start by writing schematically the structure of the non-linearity that we would like to address in this paper. To point out the novelty, we first write a system that contains only the new difficult terms, with less notation as possible, keeping the other terms under the ``$\ldots$", with the understanding that we shall explain later, in \eqref{badtensorialcoupledwaveequation}-\eqref{goodtensorialcoupledwaveequation}, the full structure that we address. We look at a system that looks schematically speaking as
\bea\label{roughideaonnonlinearitytype}
\notag
g^{\a\b} \derm_{\a } \derm_\b  \Phi^{(l)}_{{\cal U} \ldots {\cal U}} &=& \overbrace{ \underbrace{\sum_{l_i\,, l_j \in \{0, \ldots, N \}}  \Phi^{(l_i)} \cdot \Phi^{(l_j)} }_{\text{new terms not present for Einstein-Yang-Mills}}   +   \sum_{l_i\,, l_j \in \{0, \ldots, N \}}  \Phi^{(l_i)} \cdot \derm \Phi^{(l_j)} }^{\text{new ``bad" troublesome terms}} + \dots  \,, \\
\eea
for all $l \in  \{0, \ldots, N \} $\,. Here, ${g}^{\a\b}$ is a function of $\Phi^{(0)}$\,, since $g_{\a\b} = \Phi^{(0)}_{\a\b}  + m_{\a\b}$\,.

\begin{remark}
The wave equation on the metric that we will actually consider is not that on $\Phi^{(0)}_{\mu\nu}$\, but rather on $ \Phi^{(0)}_{\mu\nu} -  h^0_{\mu\nu} $\,, where $h^0_{\mu\nu}$ is the spherically symmetric Schwarzschildian part, defined in Subsection \ref{The initial data set and the higher order energy norm}, that will have infinite energy. However, we still wrote the wave equation on $\Phi^{(0)}_{\mu\nu}$ because it is not our point at this stage in this discussion, rather our point for now is to exhibit the new non-linearities that we address in this paper.
\end{remark}

Yet, we know that such a system \eqref{roughideaonnonlinearitytype} without an additional structure on how exactly such non-linearities enter, cannot have global well-posedness and decay properties, since there are plenty of counter-examples. In order to exhibit the structure of the non-linearities that we would like to address, we construct first the following frame using the fixed system of coordinates $(t, x^1, x^2, x^3)$\,.
\begin{definition}\label{definitionofthenullframusingwavecoordinates}
At a point $p$ in the space-time, let
\bea
L &=&  \pa_{t} + \pa_{r} \, =   \pa_{t} +  \frac{x^{i}}{r} \pa_{i} \, , \\
\underline{L} &=&  \pa_{t} - \pa_{r} =  \pa_{t} -  \frac{x^{i}}{r} \pa_{i} \, ,
\eea
and let $\{e_1, e_{2} \}$ be an orthonormal frame on $\SSS^{2}$. 
We define the sets 
\bea
\cal T&=& \{ L,e_1, e_{2}\} \, ,\\
 \cal U&=&\{ \underline{L}, L, e_1, e_{2}\} \, .
\eea
We call the full set ${ \cal U}$ a null-frame. We call the set ${ \cal T}$\, the tangential frame. 
\end{definition}
\begin{definition}\label{definitionofthesetsaseithertangialorthewholeset}
We define the sets ${\cal S}^l_i\;, \; l \in \{0, 1, \ldots, N \}$\,, $i \in \{ 1, \ldots, k_l \}$\,, where $k_l$ is the order of the tensor $\Phi^{(l)}$\,, as being sets that are either ${\cal T}$ or ${\cal U}$\,, i.e. 
\beaa
{\cal S}^l_i := \begin{cases} & {\cal T} \quad \text{or} \\
& {\cal U} \; . \end{cases}
\eeaa
\end{definition}
\begin{definition}\label{definitionofthetangentialdefivatives}
We define $\rderm$ as a derivative along any vector $V \in {\cal T }$\,, i.e.
\bea
\rderm := \derm_{V}  \;, \quad \text{with $V \in {\cal{T} }$} \,.
\eea
\end{definition}

Now, we look at a dynamical coupled system, such that for all $l \in \{0, 1, \ldots, N \}$\,,
\bea\label{waveequationsthatisschematicwithlessnotation}
\notag
{g}^{\a\b} \derm_{\a } \derm_\b  \Phi^{(l)}_{{\cal U} \ldots {\cal U}} &=&\overbrace{ \underbrace{ \sum_{l_i\,, l_j \in \{0, \ldots, N \}}  \Phi^{(l_i)}_{{\cal S}^{l_i}_1 \ldots {\cal S}^{l_i}_{k_{l_i}} } \cdot \Phi^{(l_j)}_{{\cal S}^{l_j}_1 \ldots {\cal S}^{l_j}_{k_{l_j}} }  }_{\text{new terms not present for Einstein-Yang-Mills}}+   \sum_{l_i\,, l_j \in \{0, \ldots, N \}}  \Phi^{(l_i)}_{{\cal S}^{l_i}_1 \ldots {\cal S}^{l_i}_{k_{l_i}} } \cdot \derm \Phi^{(l_j)}_{{\cal S}^{l_j}_1 \ldots {\cal S}^{l_j}_{k_{l_j}} } }^{\text{new ``bad" troublesome terms}}  \\
&&+ \ldots \,, 
\eea
and
\bea
{g}^{\a\b} \derm_{\a } \derm_\b  \Phi^{(l)}_{{\cal S}^l_1 \ldots {\cal S}^l_{k_l} } &=&  \sum_{l_i\,, l_j \in \{0, \ldots, N \}} \derm  \Phi^{(l_i)}_{{\cal U} \ldots {\cal U}}  \cdot \rderm \Phi^{(l_j)}_{{\cal U} \ldots {\cal U}} + \ldots \;.
\eea

\begin{remark}
Here, the ``$\ldots$" means that we can include more non-linear terms, so as to match our system with a general class of partial differential equations that covers Einstein-vacuum, Einstein-Yang-Mills and more non-linearities. However, at this stage we simplified the system so as to focus on exhibiting early on the novelty regarding the non-linearities that we treat in this paper, for which the previous methods do not work to best of our knowledge.
\end{remark}

We see in the troublesome wave equation in \eqref{waveequationsthatisschematicwithlessnotation} on $\Phi^{(l)}_{\underline{L} \ldots \underline{L} }$\;, that there exist new “bad” terms which are on one hand $ \Phi^{(l_i)}_{{\cal S}^{l_i}_1 \ldots {\cal S}^{l_i}_{k_{l_i}} } \cdot \Phi^{(l_j)}_{{\cal S}^{l_j}_1 \ldots {\cal S}^{l_j}_{k_{l_j}} } $ and on the other hand $ \Phi^{(l_i)}_{{\cal S}^{l_i}_1 \ldots {\cal S}^{l_i}_{k_{l_i}} } \cdot \derm \Phi^{(l_j)}_{{\cal S}^{l_j}_1 \ldots {\cal S}^{l_j}_{k_{l_j}} } $\,.

To the best of our knowledge, in the seminal work of Lindblad-Rodnianski \cite{LR10}, their $L^\infty$-estimate does not work to treat these \textit{new terms} of the type $\Phi^2 $ and $\Phi \cdot \pa \Phi $\,, in contrast to terms of the type $(\pa \Phi)^2$ which were dealt with in \cite{LR10}. The goal of this paper is to prove Theorem \ref{Thetheoremofthepaperdispersiveestimatesweaknullcondition}.

\subsection{Novelty of the current work and comparison with the previous studies}\

\subsubsection{Novelty}\

The system of equations that we treat in this paper, namely \eqref{badtensorialcoupledwaveequation}-\eqref{goodtensorialcoupledwaveequation}, does \textit{not} satisfy the null condition of Christodoulou (\cite{Chr1}) and of Klainerman (\cite{Kl2}), and has a \textit{new different} non-linearities than that one treated by Lindblad and Rodnianski (\cite{LR10}), that pose their own challenges and serious complications. It also covers a more general class than the Einstein-Yang-Mills equations in the Lorenz gauge, \cite{G6}. Our aim is to advance on the study of structures for non-linear hyperbolic partial differential equations.

This being said, let us note that in the seminal work of Lindblad-Rodnianski \cite{LR10}, they proved the non-linear stability of Minkowski space-time in wave coordinates based on a null-frame decomposition of the equations that exhibited a weak-null structure. Their celebrated work was new and proved a result that was not expected to hold true in the harmonic gauge.

However, the original approach of Lindblad-Rodnianski, \cite{LR10}, was based on an $L^\infty$-estimate that does \textit{not} work for the coupled system that we treat in this paper, as far as our knowledge is concerned. Hence, in order to be able to include these \textit{new non-linearities}, such as the ones that we point out in \eqref{waveequationsthatisschematicwithlessnotation}, it is important to provide a \textit{new alternative approach} that replaces the $L^\infty$-estimate of Lindblad-Rodnianski. There was a difficulty encountered by Lindblad-Rodnianski in \cite{LR2}-\cite{LR10}, and also by Lindblad-Tohaneanu in \cite{Lind-Toh}, in trying such a new approach, by attempting to \textit{separate the energy estimates for each component of the tensorial solutions}, and it was \textit{not} resolved there.

In our new approach, we resolve the problem by establishing the desired \textit{separation of the energy estimates of the higher order energy norm for each component of the tensors}, that allows us to replace the $L^\infty$-estimate of Lindblad-Rodnianski, \cite{LR10}, in a way that permits us to include these \textit{new non-linearities}, such as the ones in \eqref{waveequationsthatisschematicwithlessnotation}, for which their $L^\infty$-estimate does \textit{not} work to the best of our knowledge. 

We overcome the difficulty regarding \textit{decoupling} the \textit{higher order energy} estimates for the \textit{tangential components}, by proving suitable bounds on \textit{each component in a frame} of the \textit{tensorial} fields solutions to the \textit{dynamically} coupled system of non-linear wave equations \eqref{waveequationsthatisschematicwithlessnotation}. This is partly achieved by establishing a \textit{new} estimate on the \textit{commutator term} for the Lie derivatives of the fields, that is decoupled for the \textit{tangential components}, up to some good factor, \eqref{Theseperatecommutatortermestimateforthetangentialcomponents}. It is also achieved by recasting the system as a \textit{covariant} system of tensorial hyperbolic wave equations, although we are fixing the system of coordinates, in an approach that is totally \textit{covariant}. This allows us to establish suitable bounds on \textit{each component in a frame} of the \textit{tensorial} fields solutions to the dynamically coupled system of non-linear wave equations, which combined with our novel \textit{decoupled commutator estimate} for the \textit{tangential components}, permits us to decouple the higher order energy estimates for each component of the tensorial solution in a frame decomposition, \eqref{Theveryfinal }. This replaces the $L^\infty$-estimate of Lindblad-Rodnianski that does not work for the non-linearities present in \eqref{waveequationsthatisschematicwithlessnotation}, to the best of our knowledge, in a way that allows us to include these \textit{new non-linear structures}.

More precisely, we prove Theorem \ref{Thetheoremofthepaperdispersiveestimatesweaknullcondition}, which establishes decay in the exterior region for a general class of \textit{dynamical} coupled \textit{quasilinear} wave equations that contains new non-linearities of the type of $\Phi^2$ and $\Phi \cdot \pa \Phi$ for which the previous methods in the literature do not work\,, to the best of our knowledge. We overcome the difficulty regarding treating these non-linearities, $\Phi^2$ and $\Phi \cdot \pa \Phi$\,, party by \textit{decoupling}  the \textit{higher order energy} estimates for the \textit{tangential components}, by proving suitable bounds on \textit{each component in a frame} decomposition of the \textit{tensorial} fields solutions to the dynamically coupled system of non-linear wave equations, \eqref{notupgradedyetbutimprovedboundonenergyofgoodcomponentwithgamprime}-\eqref{notupgradedyetbutimprovedboundonenergyoffullcomponent}. These \textit{energy estimates} on the $L^2$-norm of the higher order energy are made and tailored to address these new non-linear terms of the type $\Phi^2$ and $\Phi \cdot \pa \Phi$\,, which are non-linearities that pose their own challenges and complications due to the \textit{lack of derivatives}. However, we use the fact that the components involved in $\Phi^2$ and $\Phi \cdot \pa \Phi$ are ``good" components and satisfy a good wave equation with a good non-linear structure. Hence, our decoupled energy estimates allow us to control these new non-linear terms using the good non-linear structure satisfied by these components.

\subsubsection{Related works}\

As it is well-known, global existence and decay for a general class of quasilinear wave equations satisfying the \textit{weak null-condition, with small initial data}, is quite an \textit{open problem}, see for example the work of Keir in \cite{Keir1}-\cite{Keir2} and Luk-Oh-Yu in \cite{LOY}, for an important discussion. Also, non-linearities such as terms of the type $\Phi^2$ and $\Phi \cdot \pa \Phi $\,, that we treat in this paper, are known to be ``bad" terms due to the \textit{lack of derivatives}, see for example the work of Kadar in \cite{Kadar}.

The weak-null condition arises from the program of H\"ormander to study the asymptotic system of non-linear wave equations, see \cite{H1}-\cite{H2}-\cite{H3} and also \cite{Lind3}. It was then first clearly identified in the pioneering work of Lindblad-Rodnianski in \cite{LR1}. They showed that the Einstein \textit{vacuum} equations in \textit{wave coordinates} satisfy a certain specific sub-class of the weak-null condition, where the non-linearities enter as $(\pa \Phi )^2$\,. Thereafter, in their seminal work, \cite{LR2}-\cite{LR10}, they established global existence and decay for small asymptotically flat initial data for solutions of the \textit{Einstein vacuum} equations in the harmonic gauge, whereas it was expected to be unstable.

Indeed, it is in fact known that \textit{non-linearities} for the wave equation can can cause blow-up in finite time, see for example the results of John in \cite{John1}-\cite{John2}. In \cite{Chr1} and \cite{Kl2}, Christodoulou and Klainerman independently identified a strong condition, namely the null condition, that ensures global existence for small initial data. However, while this condition is sufficient, it is far from being a necessary condition, see for instance the influential work of Lindblad in \cite{Lind1}-\cite{Lind2} and of Alinhac in \cite{Alin1}. The weak null-condition, \cite{LR1}, was then identified as an example of non-linearities that violate the null condition, but could still, perhaps, have global existence of solutions for small initial data, up to investigation. The race to identify the necessary and sufficient conditions on the \textit{non-linear structure} of quasilinear wave equations that ensure global existence of solutions for small initial data is very much an open problem. So far, the question is quite an \textit{open problem} even within the class of structures that satisfy the \textit{weak-null condition}. In this paper, we identify a general class of quasilinear tensorial wave equations, satisfying the weak-null condition, that ensures global existence and decay in the exterior for small asymptotically flat initial data.

In the case of Einstein \textit{vacuum} equations, without sources, the stability of Minkowski space-time was proven in \cite{F}, \cite{C-K},  \cite{Bieri2}, \cite{Bieri3}, \cite{Bieri}, \cite{LR2}-\cite{LR10}, \cite{Hu1}, \cite{HV}, \cite{Hintz}, \cite{Shen1}-\cite{Shen2}-\cite{Shen3}. In the \textit{exterior} region, it was also proven for the Einstein \textit{vacuum} equations  in  \cite{KN}, \cite{Shen3} and for spacelike-characteristic initial data in \cite{Graf}.

The methods used in these different results employed various techniques. For instance, in the celebrated seminal work of Christodoulou-Klainerman, the \textit{vector field method} was used. In the pioneering work of Hintz-Vasy \cite{HV}, of H\"afner-Hintz-Vasy in \cite{HHV}, and of Hintz in \cite{Hintz}, they made use of \textit{microlocal analysis} techniques. In a seminal work, \cite{DR1}, Dafermos-Rodnianski introduced the method of \textit{$r^p$-weighted energy estimates} and it was used to write new proofs of stability of Minkowski space-time for Einstein \textit{vacuum} equations, for example by Shen in \cite{Shen2}, and it was used in a significant work by Keir in \cite{Keir1}-\cite{Keir2} to prove global existence for a general class of quasilinear wave equations satisfying the \textit{weak-null condition}. There, a geometric foliation was used along the \textit{$r^p$-weighted energy estimates} of Dafermos-Rodnianski \cite{DR1}. These methods were also used in \cite{AAG1}-\cite{AAG2}. The approach that we take here is different, as we make use of a background foliation along with the background symmetry vector fields, as in \cite{LR10} and see also \cite{DR} and \cite{DR2}.

However, the question of stability of Minkowski space-time governed by the Einstein equations coupled to a \textit{general class of non-linear matter}, is still an \textit{open problem}. Concerning the matter prescribed by the \textit{Yang-Mills fields}, to the best of our knowledge, the only known complete result of exterior stability of Minkowski space-time for the Einstein equations coupled to the Yang-Mills fields is the author's result, \cite{G6}-\cite{G7}, which concerns the Einstein-Yang-Mills system in the \textit{Lorenz gauge}. There is also the result of Mondal-Yau, \cite{MY1}-\cite{MY3}, however without estimates on the gauge, but only on the gauge invariant quantities. In \cite{G4}-\cite{G5}, there is also the author's result in higher dimensions, as well as separate energy estimates for \textit{three} space-dimensions, which were used in \cite{G6}-\cite{G7}, to give a full proof of the exterior stability of the \textit{dynamical} Minkowski space-time in \textit{three} space-dimensions in the \textit{Lorenz gauge}. For perturbations of de Sitter space-time, there is the result of Friedrich in \cite{Fried2} and in higher dimensions, the result of Liu-Oliynyk-Wang in \cite{LOW}. As far as our knowledge is concerned, the only result of decay for Yang-Mills fields on a \textit{fixed curved} background is the author's result with H\"afner in \cite{G3H}. For \textit{global existence} on fixed curved spaces-times, without decay, see the result of Chru\'sciel-Shatah's in \cite{CS}, and the author's result in \cite{G1}. On the \textit{flat} space-time, see the global existence result of Eardley-Moncrief in \cite{EM1}-\cite{EM2}  and of Choquet-Bruhat and Christodoulou in \cite{CB-Chri}. In \cite{WYY}, there is a decay result by Wei-Yang-Yu on the \textit{fixed Minkowski} space-time.  There is also a large data result of formation of trapped surfaces by Mondal-Athanasiou-Yau in \cite{AthMonYau}.

There are several results for Einstein equations coupled to \textit{linear} matter, concerning the Einstein-\textit{Maxwell} equations, as in the work of Bieri-Chen-Yau, \cite{BCY}, and of Zipser in \cite{Z}-\cite{Z2}. Using wave coordinates, it was also established by Loizelet in \cite{Loiz1} and also by Speck in \cite{Speck} for a class of electromagnetic fields. On a \textit{fixed} background, there are the results of Andersson-Blue in \cite{AB15}-\cite{AB15_02}, the author’s work in \cite{G2}, the result of Metcalfe-Tataru-Tohaneanu in \cite{MTT}, of Pasqualotto in \cite{Pas1} and the work of Ma in \cite{Ma}.

There are also several other results of stability of Minkowski space-time coupled to matter, such as \cite{Lind-Tay}, \cite{BFJST}, etc ... . We note that the most relevant work for our non-linearities are the Einstein equations coupled to \textit{non-linear} matter that satisfy the \textit{weak-null condition}, see the following Subsection \ref{The coupled dynamical system of tensorial quasilinear wave equations} for a precise presentation of the system of tensorial quasilinear wave equations that we study in this paper.

\subsection{The coupled dynamical system of tensorial quasilinear wave equations}\label{The coupled dynamical system of tensorial quasilinear wave equations}

Now, we would like to give a more precise description of the system of quasilinear wave equations that we treat here, so as to be able to give the precise statement of our Theorem \ref{Thetheoremofthepaperdispersiveestimatesweaknullcondition}. For this, we need the following definitions.

\begin{definition}\label{definitionofbigOonlyforAandhadgradientofAandgardientofh}
Let $\Phi^{(l)}$\,, $l \in \{0, 1, \ldots, N \}$\,, be a family of tenors as in \eqref{waveequationsthatisschematicwithlessnotation}. Let $K$ be a tensor that is either $\Phi^{(l)}$ or $\derm \Phi^{(l)}$\;. Let $ P_n^l (\Phi^{(l)} )$\,, $ R_n^l (\derm \Phi^{(l)} )$ be tensors that are either polynomials of degree $n$ (or identically vanishing), and $Q_1 (K)$ a tensor that is a polynomial of degree $1$ such that $Q_1 (0) = 0$ and $Q_1 \neq 0$\;, of which the coefficients are components of the metric $\textbf m$ and of the inverse metric $\textbf m^{-1}$, defined in Definition \ref{definitionofMinkowskiandcovariantderivtaiveofMinkwoforafixedgivensystemofcoordinates}, and of which the variables are components of the covariant tensor $\Phi^{(l)}$\,, $\derm \Phi^{(l)}$ and $K$\,, respectively, leaving some indices free, so that the following product gives a tensor that we define as,
\bea
O_{\mu_{1} \ldots \mu_{k} } (K  ) &:=& Q_1 ( K) \cdot  \sum_{l=0}^{N} \sum_{n=0}^{\infty} \Big(   P_n^l ( \Phi^{(l)} ) +  R_n^l ( \derm \Phi^{(l)}  ) \Big) \; .
\eea
Similarity, for fixed vectors $V_1, \ldots , V_k$\,, we define 
\bea
O_{\mu_{1} \ldots \mu_{k} } (K_{V_1 \ldots  V_k}  ) &:=& Q_1 ( K_{V_1 \ldots  V_k}) \cdot  \sum_{l=0}^{N} \sum_{n=0}^{\infty} \Big(   P_n^l ( \Phi^{(l)} ) +  R_n^l ( \derm \Phi^{(l)}  ) \Big) \; .
\eea
where here $Q_1 (K_{V_1, \ldots , V_k}  )$ is a tensor that is a polynomial of degree $1$ such that $Q_1 (0) = 0$ and $Q_1 \neq 0$\;, of which the coefficients are components of the metric $\textbf m$ and of the inverse metric $\textbf m^{-1}$, and of which the variables are $K_{V_1, \ldots , V_k}$\,.

For a family of tensors $K^{(1)}, \ldots,  K^{(m)}$, where each tensor $K^{(j)}$ is again either $\Phi^{(j)}$ or $\derm \Phi^{(j)}$\;, we define
\bea
O_{\mu_{1} \ldots \mu_{k} } (K^{(1)} \cdot \ldots \cdot K^{(m)} ) &:=& \prod_{j=1}^{m} Q_{1}^{j} ( K^{(j)} ) \cdot  \sum_{l=0}^{N} \sum_{n=0}^{\infty} \Big(  P_n^l ( \Phi^{(l)} ) + R_n^l ( \derm \Phi^{(l)}  ) \Big) \; .
\eea
where again $P_n^{l} (\Phi^{(l)}) $ and $P_n^{l} (\derm \Phi^{(l)}) $ are tensors that are polynomials of degree $n$ (or identically vanishing), and $Q_1^j (K^{(j)})$, is a tensor that is polynomials of degree $1$, with $Q_1^j (0) = 0$ and $Q_1^j \neq 0$\;, of which the coefficients are components of the metric $\textbf m$ and of the inverse metric $\textbf m^{-1}$, and of which the variables are components of the covariant tensors $\Phi^{(l)}$\,, $\derm \Phi^{(l)}$ and $K^{(l)}$\,, respectively, leaving some indices free, so that at the end the whole product $\prod_{j=1}^{m} Q_{1}^{j} ( K^{(j)} ) \cdot \sum_{l=0}^{N} \sum_{n=0}^{\infty}  \Big( P_n^l ( \Phi^{(l)} ) + R_n^l ( \derm \Phi^{(l)}  ) \Big)$ gives a tensor which we define as $O_{\mu_{1} \ldots \mu_{k} } (K^{(1)} \cdot \ldots \cdot K^{(m)} )$. To lighten the notation, we shall sometimes drop the indices and just write $O (K^{(1)} \cdot \ldots \cdot K^{(m)} )$\;.

Similarly, for a fixed family of vectors $ V^l_1, \ldots , V^l_{k_l}$\,,
\bea
\notag
O_{\mu_{1} \ldots \mu_{k} } (K^{(1)}_{V^1_1 \ldots  V^1_{k_1}} \cdot \ldots \cdot K^{(m)}_{V^m_1 \ldots  V^m_{k_m}} ) &:=& \prod_{j=1}^{m} Q_{1}^{j} ( K^{(j)}_{V^j_1 \ldots  V^j_{k_j}}  ) \cdot  \sum_{l=0}^{N} \sum_{n=0}^{\infty} \Big(  P_n^l ( \Phi^{(l)} ) + R_n^l ( \derm \Phi^{(l)}  ) \Big) \; .\\
\eea
where again $Q_1^j (K^{(j)}_{V^j_1 \ldots  V^l_{k_j}}  )$ is a tensor that is a polynomial of degree $1$ such that $Q_1^j (0) = 0$ and $Q_1^j \neq 0$\;, of which the coefficients are components of the metric $\textbf m$ and of the inverse metric $\textbf m^{-1}$, and of which the variables are $K^{(j)}_{V^j_1 \ldots  V^j_{k_j}} $\,.
\end{definition}

\begin{remark}
Each big $O$ is different, i.e. each $O$ has different polynomials in its construction.
\end{remark}

\begin{definition}
Let $m^{\mu\nu}$ be the inverse of $m_{\mu\nu}$\,. We define
\bea\label{definitionofsmallh}
(\Phi^{(0)})^{\mu\nu} &:=& m^{\mu\mu^\prime}m^{\nu\nu^\prime}\Phi^{(0)}_{\mu^\prime\nu^\prime} \\
\label{definitionsofbigH}
H^{\mu\nu} &:=& g^{\mu\nu}-m^{\mu\nu} \,.
\eea
Then, it is then well-known that one can prove that
\bea\label{linkbetweenbigHandsamllh}
H^{\mu\nu}= - (\Phi^{(0)})^{\mu\nu}+ O^{\mu\nu}( (\Phi^{(0)})^2) \, .
\eea
\end{definition}

We study the following coupled system of tensorial non-linear wave equations, for all $l \in \{0, 1, \ldots, N \} $\,,
\bea\label{badtensorialcoupledwaveequation}
\notag
&& \big[- (\Phi^{(0)})^{\a\b}   + O^{\a\b}( (\Phi^{(0)})^2) + m^{\a\b} \big] \cdot \derm_{\a } \derm_\b  \Phi^{(l)}_{{\cal U} \ldots {\cal U}} \\
\notag
 &=&    \sum_{l_i\,,\, l_j \in \{0, 1, \ldots, N \}} \big[ \; O(   \Phi^{(l_i)}_{{\cal S}^{l_i}_1 \ldots {\cal S}^{l_i}_{k_{l_i}} } \cdot \Phi^{(l_j)}_{{\cal S}^{l_j}_1 \ldots {\cal S}^{l_j}_{k_{l_j}} }  ) + O( \Phi^{(l_i)}_{{\cal S}^{l_i}_1 \ldots {\cal S}^{l_i}_{k_{l_i}} } \cdot \derm \Phi^{(l_j)}_{{\cal S}^{l_j}_1 \ldots {\cal S}^{l_j}_{k_{l_j}} } )  \; ,\\
 \notag
 && + O( \derm  \Phi^{(l_i)}_{{\cal S}^{l_i}_1 \ldots {\cal S}^{l_i}_{k_{l_i}} } \cdot \derm \Phi^{(l_j)}_{{\cal S}^{l_j}_1 \ldots {\cal S}^{l_j}_{k_{l_j}} } )  +  \sum_{ l_s \in \{1, \ldots, N \}} O(  \Phi^{(l_i)} \cdot  \Phi^{(l_j)} \cdot \Phi^{(l_s)} )  \\
 \notag
 && + O(  \Phi^{(l_i)} \cdot \rderm  \Phi^{(l_j)}  )    + O(  \derm  \Phi^{(l_j)} \cdot  ( \Phi^{(l_i)} )^2  )  + O(  \Phi^{(l_i)} \cdot (\derm  \Phi^{(l_j)})^2  )   \\
 && + O(  \derm \Phi^{(l_i)} \cdot \rderm  \Phi^{(l_j)}  ) \; \big] \;,  
\eea
and
\bea\label{goodtensorialcoupledwaveequation}
\notag
 && \big[- (\Phi^{(0)})^{\a\b}   + O^{\a\b}( (\Phi^{(0)})^2) + m^{\a\b} \big]  \cdot  \derm_{\a } \derm_\b  \Phi^{l}_{{\cal S}^l_1 \ldots {\cal S}^l_{k_l} } \\
 \notag
 &=&  \sum_{l_i\,,\, l_j \in \{0, 1, \ldots, N \}} \big[\;   \sum_{ l_s \in \{1, \ldots, N \}} O(  \Phi^{(l_i)} \cdot  \Phi^{(l_j)} \cdot \Phi^{(l_s)} )  + O(  \Phi^{(l_i)} \cdot \rderm  \Phi^{(l_j)}  ) \\
 &&   + O(  \derm  \Phi^{(l_j)} \cdot  ( \Phi^{(l_i)} )^2  )  + O(  \Phi^{(l_i)} \cdot (\derm  \Phi^{(l_j)})^2  )   + O(  \derm \Phi^{(l_i)} \cdot \rderm  \Phi^{(l_j)}  ) \;  \big]  \;.
\eea

\begin{remark}\label{remarkonthewaveequationforthemetricwherewesubstracttheSchwarzschildianpart}
The tensorial wave equation on the metric that we will consider is in fact on  $(\Phi^{(0)} -h^0)_{\mu\nu}$\,, where $h^0$ is the spherically symmetric Schwarzschildian part that has infinite energy, defined in Definition \ref{theactualmetrich1withfiniteenergy}. Hence, the wave operator on the metric that we consider, namely $(\Phi^{(0)} -h^0)$\,, is 
\beaa
  && \big[- (\Phi^{(0)})^{\a\b}   + O^{\a\b}( (\Phi^{(0)})^2) + m^{\a\b} \big]  \cdot \derm_{\a } \derm_\b  (\Phi^{(0)} -h^0)_{\mu\nu} \\
    &=&  \big[- (\Phi^{(0)})^{\a\b}   + O^{\a\b}( (\Phi^{(0)})^2) + m^{\a\b} \big]  \cdot  \derm_{\a } \derm_\b  \Phi^{(0)}_{\mu\nu} \\
 &&- \big[- (\Phi^{(0)})^{\a\b}   + O^{\a\b}( (\Phi^{(0)})^2) + m^{\a\b} \big]  \cdot  \derm_{\a } \derm_\b  h^0_{\mu\nu} \,,
 \eeaa
 where the first term in the last equality is given by \eqref{badtensorialcoupledwaveequation}
 and \eqref{goodtensorialcoupledwaveequation} and the last term is prescribed by Definition \ref{theactualmetrich1withfiniteenergy}.
\end{remark}

\begin{notation}
Here, in the right hand side of \eqref{badtensorialcoupledwaveequation}, we have by definition that
\beaa
O(   \Phi^{(l_i)}_{{\cal S}^{l_i}_1 \ldots {\cal S}^{l_i}_{k_l} } \cdot \Phi^{(l_j)}_{{\cal S}^{l_j}_1 \ldots {\cal S}^{l_j}_{k_l} }  ) &:=& \sum_{V^{l_i}_1 \in {\cal S}^{l_i}_1 \,, \ldots \,, V^{l_i}_{k_{l_i}} \in  {\cal S}^{l_i}_{k_{l_i}}  } \; \; \;   \sum_{W^{l_j}_1 \in {\cal S}^{l_j}_1 \,, \ldots \,, V^{l_j}_{k_{l_j}} \in  {\cal S}^{l_j}_{k_{l_j}}  } O(   \Phi^{(l_i)}_{V^{l_i}_1 \ldots V^{l_i}_{k_l} } \cdot \Phi^{(l_j)}_{W^{l_j}_1 \ldots W^{l_j}_{k_l} }  ) \,,
\eeaa
and so on for all the other big $O$\,. Since these are big $O$\,, this means not all these terms should be present, but they can be present, and they can be mixed terms since $ V^{l_i}_k $ is not necessarily equal to $W^{l_j}_k$\,. However, on the left hand side of equation \eqref{goodtensorialcoupledwaveequation}, we have by definition that when we write $ \Phi^{l}_{{\cal S}^l_1 \ldots {\cal S}^l_{k_l} }$\,, that the equations are satisfied for all  $\Phi^{l}_{V^l_1 \ldots V^l_{k_l} }$\,, with 
$V^{l}_1 \in {\cal S}^{l}_1 \,, \ldots \,, V^{l}_{k_{l}} \in  {\cal S}^{l}_{k_{l}}$\,.

In other words, all the elements in the sets should verify the equation \eqref{goodtensorialcoupledwaveequation}, even if they do not all appear in \eqref{badtensorialcoupledwaveequation}. Differently speaking, if one element of the set $V^{l}_{k_{l}} \in  {\cal S}^{l}_{k_{l}}$ appears in \eqref{badtensorialcoupledwaveequation}, this implies that the whole elements of the set ${\cal S}^{l}_{k_{l}}$ should satisfy \eqref{goodtensorialcoupledwaveequation}. This is because we can decouple the commutator term for the higher order energy, at the level of the Lie derivatives of the fields, only for sets of the type ${\cal U}$ or ${\cal T}$\,, and not for each element in the set. Recall that in Definition \ref{definitionofthesetsaseithertangialorthewholeset}, the definition of the sets ${\cal S}^{l}_{i}$ are that they are either ${\cal U}$ or ${\cal T}$\,.
\end{notation}

\subsection{The initial data set and the norms}\label{The initial data set and the higher order energy norm}\

We assume that we are given an initial data set $(\Sigma\,, \Phi^{(l)} \,, \pa_t \Phi^{(l)}\,, l \in \{0, 1, \ldots, N \} )$\,, where $\Sigma$ is diffeomorphic to $\R^3$ and where $\Phi^{(0)}$ is a Lorentzian metric, 2-tensor, defined on all points at $\Sigma$\,, and $\Phi^{(l)}$ are tensors of arbitrary order, and of arbitrary nature, defined on $\Sigma$ \,.

\begin{remark}
We note that such initial data set can be constructed from the Cauchy formulation of the Einstein equations coupled with matter, where one is prescribed only the tensors restricted on $\Sigma$ without the components involving $t$\,, which one can then complete using the Einstein constraint equations and gauge fixing. It is not our point to construct such an initial data set. We know that one can construct such an initial data set, while respecting the Einstein constraint equations, which also involves making a choice since the Cauchy formulation of the Einstein equations does not prescribe $\pa_t \Phi^{(l)}$\,. However, this being said, the metric for the Einstein equations will involve a spherically symmetric Schwarzschildian part, which motivates the following definition. 
\end{remark}

Recall that the metric $ m_{\mu\nu}$ is defined in Definition \ref{definitionofMinkowskiandcovariantderivtaiveofMinkwoforafixedgivensystemofcoordinates} to be the Minkowski metric in our fixed system of coordinates $(x^0:= t, x^1, x^2, x^3)$\,. 
\begin{definition}\label{theactualmetrich1withfiniteenergy}
First, we define 
\bea
r := \sqrt{ (x^1)^2 + (x^2)^2 +(x^3)^2  }\;.
\eea
Then, let $\chi$ be a smooth function such as
 \bea\label{defXicutofffunction}
\chi (r)  := \begin{cases} 1  \quad\text{for }\quad r \geq \frac{3}{4} \;  ,\\
0 \quad\text{for }\quad r \leq \frac{1}{2} \;. \end{cases} 
\eea
For $M \geq 0 $\,, we define $h^0$\,, such that for $t > 0$\,,
\bea\label{guessonpropagationofthesphericallsymmetricpart}
h^0_{\mu\nu} := \chi(r/t) \cdot \chi(r)\cdot \frac{M}{r}\de_{\mu\nu}  \;  ,
\eea
and for $t=0$\,, 
 \bea\label{definitionofthesphericallysymmtericpartofinitialdata}
h^0_{\mu\nu} ( t= 0) := \chi(r)\cdot \frac{M}{r}\cdot \de_{\mu\nu}  \; .
\eea
Therefore, considering the definition of $\Phi^{(0)}$ in Definition \ref{DefinitionofPhioasgminustheminkowski}, we have
\bea
\Phi^{(0)}_{\mu\nu} - h^{0}_{\mu\nu} := g_{\mu\nu} - m_{\mu\nu}- h^0_{\mu\nu}  \;  .
\eea
\end{definition}
\begin{remark}
We are going to consider the wave equation on $\Phi^{(0)} - h^{0}_{\mu\nu}  $ since $h^0_{\mu\nu}$ will have infinite energy, see Remark \ref{remarkonthewaveequationforthemetricwherewesubstracttheSchwarzschildianpart}.
\end{remark}

\begin{definition}\label{defoftheweightw}
We define the weight $w$\,, by
\bea
w_{\gamma} (r-t):=\begin{cases} (1+|r-t|)^{1+2\gamma} \quad\text{when }\quad r-t>0 \;, \\
         1 \,\quad\text{when }\quad r-t<0 \; , \end{cases}
\eea
for some $\ga > 0 $ that will be determined later.
\end{definition}

\begin{definition}

We define the tensor $E_{\mu\nu} $ as the euclidian metric in our fixed system of coordinates, 
\bea
\notag
E_{\mu\nu} := m(\frac{\pa}{\pa x^\mu}, \frac{\pa}{\pa x^\nu})+2 m(\frac{\pa}{\pa x^\mu}, \frac{\pa}{\pa t}) \cdot m(\frac{\pa}{\pa x^\nu}, \frac{\pa}{\pa t}) \, .
\eea
For a tensor of arbitrary order, say $K_{\a_1 \ldots \a_{k} }$\;, we define
\beaa
 | K  |^2= E^{\a_1\b_1} \ldots  E^{\a_k\b_k}   < K_{\a_1 \ldots \a_{k} } ,   K_{\b_1 \ldots \b_{k} }>  \, .
\eeaa
Therefore, we have
\beaa
 | K  |^2 &=&  \sum_{\a_1, \ldots, \a_{k}  \in  \{t, x^1, \ldots, x^n \}} | K_{\a_1 \ldots \a_{k} }|^2   \, ,
\eeaa
Also, we have
\beaa
 | K  |^2 &\sim& \sum_{U_1, \ldots, U_{k}  \in  {\cal U } } | K_{U_1 \ldots U_{k} }|^2   \, ,
\eeaa
For an arbitrary given family of vectors $V_1, \ldots,  V_{k} $\,, we define 
 \bea
 \notag
|  \derm  K_{V_1 \ldots  V_{k}}  |^2 &=&  \sum_{\mu \in  \{t, x^1, x^2, x^3 \}} |  \derm_{\mu} K_{V_1 \ldots  V_{k}} |^2    \, .
\eea
We define the norm of the tangential covariant derivative as 
 \bea
 \notag
|  \rderm  K |^2 &:=&  \sum_{U \in  {\cal T}}  \; \sum_{\a_1, \ldots, \a_k \in   \{t, x^1, x^2, x^3 \}}   |  \derm_{U} K_{\a_1 \ldots  \a_{k}} |^2   \, ,
\eea
where ${\cal T}$ is defined in Definition \ref{definitionofthenullframusingwavecoordinates}.
\end{definition}

\begin{definition} \label{DefinitionofMinkowskivectorfields}
Let $ x_{\b} = m_{\mu\b} x^{\mu} \;, $ \,$Z_{\a\b} = x_{\b} \pa_{\a} - x_{\a} \pa_{\b}  \;, $ and $S = t \pa_t + \sum_{i=1}^{n} x^i \pa_{i}  \; .$ The Minkowski vector fields are the vectors of the following set
\bea
{\cal Z}  := \big\{ Z_{\a\b}\;,\; S\;,\; \pa_{\a} \, \,  | \, \,   \a\;,\; \b \in \{ 0, 1, 2, 3 \} \big\}  \; .
\eea
Vectors belonging to ${\cal Z}$ will be denoted by $Z$\;.
\end{definition}

\begin{definition} \label{DefinitionofZI}
We define
\bea
Z^I :=Z^{\iota_1} \ldots Z^{\iota_k} \quad \text{for} \quad I=(\iota_1, \ldots,\iota_k),  
\eea
where $\iota_i$ is an $11$-dimensional integer index, with $|\iota_i|=1$, and $Z^{\iota_i}$ representing each a vector field from the family ${\cal Z}$. For a tensor $T$, of arbitrary order, either a scalar or valued in the Lie algebra, we define the Lie derivative as
\bea
\Lie_{Z^I} T :=\Lie_{Z^{\iota_1}} \ldots \Lie_{Z^{\iota_k}} T \quad \text{for} \quad I=(\iota_1, \ldots,\iota_k) .
\eea
\end{definition}

\begin{definition}\label{theexteriordomain}
Let ${\cal K}$ be a compact set in our given Cauchy hypersurface $\Sigma$\,. We define $\Sigma^{ext} $ as being the evolution in time $t$ in our fixed system of coordinates of $\Sigma$ in the future of the causal complement of ${\cal{K}}$\;. We define $\Sigma_t^{ext} $ as being a slice at a fixed time $t$\,.

\end{definition}

\subsection{The condition on the fixed system of coordinates}\

We assume that the metric in our fixed system of coordinates $(t, x^1, x^2, x^3)$ satisfies the following condition, which holds true for wave coordinates, see for example \cite{G6},
\bea\label{wavecoordinatesestimateonLiederivativesZonmetric}
\notag
| \derm ( \Lie_{Z^J} \Phi^{(0)}_{\cal T L} )  | &\les& \sum_{|K| \leq |J| } | \rderm  ( \Lie_{Z^K}  \Phi^{(0)})  |  + \sum_{|K|+ |M| \leq |J|}  O (|\Lie_{Z^K} \Phi^{(0)}| \cdot |\derm ( \Lie_{Z^M} \Phi^{(0)} ) | )\; . \\
\eea
Then in view of \eqref{linkbetweenbigHandsamllh}, we also have
\beaa
| \derm ( \Lie_{Z^J}  H_{\cal T L} )  | &\les&\sum_{|K| \leq |J| }  |  \rderm ( \Lie_{Z^K} H ) |  + \sum_{|K|+ |M| \leq |J|}  O (| \Lie_{Z^K} H| \cdot |\derm ( \Lie_{Z^M} H ) | ) \; .
\eeaa
\begin{remark}\label{estimatesonPhioalsoholdforbigH}
In fact, in general, thanks to \eqref{linkbetweenbigHandsamllh}, every estimate in this paper that holds for $\Phi^{(0)}$ will also hold for $H$\,. However, one has to be careful about $\Phi^{(0)} - h^0$\,, since $h^0$ has already a given prescribed form in \eqref{guessonpropagationofthesphericallsymmetricpart}.
\end{remark}

\subsection{The statement of the theorem}\

In this paper, we will prove the following theorem.

\begin{theorem}\label{Thetheoremofthepaperdispersiveestimatesweaknullcondition}
We consider in a certain system of coordinates that satisfies the condition \eqref{wavecoordinatesestimateonLiederivativesZonmetric}, which is the case for wave coordinates, the dynamical tensorial system of coupled quasilinear wave equations \eqref{badtensorialcoupledwaveequation}-\eqref{goodtensorialcoupledwaveequation}, on the dynamical unknowns $\Phi^{(l)}$\,, $l \in \{0,1, \ldots, N \}$\,, $N \in \N$\,. 
We assume that we have an initial data set that is asymptotically flat, as prescribed in Subsection \ref{The initial data set and the higher order energy norm}, with $M$ the mass of the initial data of the metric as in Definition \ref{theactualmetrich1withfiniteenergy}. We consider the energies defined in Definitions \ref{Definitionofgoodenergyandbadenergywithdifferentweightswiththefieldsspeaaretly} and  \ref{Definitionofenergywithallthefieldsforgoodcomponentsandenergyforfullcomponents}, with the weight prescribed as in Definition \ref{defofthenewweightwgammaprime}.

Let $\de \in \R$ with $0 < \de < \frac{1}{4} $ and let $\ga^\prime = 1 + 2\de$\,. Let $\ga \in \R$ be such that $ 0 < 5 \de \leq \ga < \frac{1}{2} +2\de = \ga^\prime -\frac{1}{2} $\,. Then, there exist a $K \in \N$ large enough, and an $\eps$ small enough, such that if
\beaa
M + \E^{\textit{full}\,,\, \gamma}_{K} (0) + \E^{\textit{good}\,,\, \gamma^{\prime}}_{K} (0)  < \eps^2 \,,
\eeaa
then, we have for all time $t$\,, in the entire exterior region,

\beaa
\E^{\textit{full}\,,\, \gamma}_{K} (t) \leq C \cdot \eps \cdot (1 +t)^{c\cdot \eps} \,.
\eeaa
Furthermore, for all $|I| \leq K -2 $\,, we have in the exterior region, the following estimates on the fields, for all $l \in \{1, \ldots, N \} $\,,

   \beaa
\notag
 |\derm   \Lie_{Z^I} \Phi^{(l)} (t,x)  |     + |\derm (  \Lie_{Z^I} \Phi^{(0)} - h^0 ) (t,x)  |    &\les&\frac{\eps }{(1+t+|q|)^{1-c \cdot \eps} \cdot (1+|q|)^{1+\ga}} \,, \\
\notag
   |  \Lie_{Z^I} \Phi^{(l)} (t,x)  |  + |  \Lie_{Z^I} (\Phi^{(0)} - h^0 ) (t,x)  |  &\les&  \frac{\eps }{(1+t+|q|)^{1-c \cdot \eps} \cdot (1+|q|)^{ \ga}}  \,,
      \eeaa
      and for $\Phi^{(0)}$\,, 
      \beaa
            \notag
      |\derm \Lie_{Z^I} \Phi^{(0)}   (t,x)  |    &\les&     \frac{\eps }{(1+t+|q|)^{1- c \cdot \eps } \cdot (1+|q|)^{1+c\cdot \eps}} \,,  \\
    \notag
 |   \Lie_{Z^I} \Phi^{(0)}  (t,x)  | &\les&  \frac{\eps }{(1+t+|q|)^{1-c \cdot \eps}  \cdot (1+|q|)^{c\cdot \eps} } \,. 
      \eeaa
     
\end{theorem}

\section{Strategy of the proof}

\subsection{The continuity argument}\

We define the following higher order energy norm.
\begin{definition}
We define the energy norm as follows
\bea\label{theboundinthetheoremonEnbyconstantEN}
&& \E _{k} (t) \\
\notag
&:=&  \sum_{|J|\leq k} \big[ \|w_{\gamma}^{1/2}   \derm ( \Lie_{Z^J} (\Phi^{(0)}- h^{0})   (t,\cdot) )  \|^2_{L^2 (\Sigma_t^{ext} ) } +  \sum_{l \in \{1, \ldots, N  \}}\|w_{\gamma}^{1/2}   \derm ( \Lie_{Z^J}  \Phi^{(l)}   (t,\cdot) )  \|^2_{L^2 (\Sigma_t^{ext} )  }  \big] \; ,
\eea
where the integration is taken on $\Sigma_t^{ext} ({\cal K})$ with respect to the Lebesgue measure $dx_1 dx_2 dx_3$\,, and where the weight $w_{\ga}$ is given in Definition \ref{defoftheweightw}.
\end{definition}

We are going to run a bootstrap argument on the higher order energy \eqref{theboundinthetheoremonEnbyconstantEN} with the weight $w_{\gamma}$\,, using the fact that the energy grows continuously in time $t$\,. We assume that the higher order energy satisfies the following bound, for all time $t$ in any interval $[0, T]$\,, with $T > 0$\,,
\bea\label{aprioriestimate}
\E_{ k } (t)  \leq  E(k) \cdot \eps \cdot (1 +t)^\delta \;,
\eea
for some $\eps > 0$ and $\de > 0 $ to be determined later.
\begin{notation}
Here we set always the constant $E(k) := 1$ for any $k \in \N$, but this constant will be used in our estimates to indicate the number of derivatives for which we are using our bootstrap assumption \eqref{aprioriestimate}.
\end{notation}
 Then, we are going to prove that this \`a priori bound implies for $k=K$ large enough, a strictly better bound, say it implies that 
\bea
\E_{ K } (t)  \leq \frac{1}{2} \cdot E(K)  \cdot \eps \cdot (1 +t)^\delta \;.
\eea
Thus, by continuity of $\E _{K} (t) $\,, and by the fact that we start with an initial data $\E _{K} (0) $ satisfying the bound, this implies that bound \eqref{aprioriestimate} is an actual true bound on the whole time interval $[0, \infty)$ by well-posedness of the equations \eqref{badtensorialcoupledwaveequation}-\eqref{goodtensorialcoupledwaveequation} in the exterior region.

\subsection{The recursive formula of Gr\"onwall type inequalities}\

The way we are going to prove the strict upgrade of inequality \eqref{aprioriestimate} is by proving a recursive formula of Gr\"onwall type inequalities, such that for initial data small enough, it allows one to upgrade \eqref{aprioriestimate} by Gr\"onwall  lemma. More precisely, our goal is to prove an inequality of the following type, for $k \in \N$\,, and most crucially, without using a bootstrap assumption, \eqref{aprioriestimate}, on higher derivatives than a certain $k=K$ large enough\,,

     \bea\label{Theinequalityontheenergytobeusedtoapplygronwallrecursively}
   \notag
     \E_{k} (t) &\les&     \eps  \cdot (1+t)^{ c_0 \cdot \eps }   +   \eps \cdot \E_{k} (0)    +         \int_{0}^{t}         \frac{\eps  }{(1+\tau )^{1-      c \cdot \eps } }   \cdot \E_{k-1} (\tau) \cdot   d\tau  \\
&& +         \int_{0}^{t}        \frac{\eps  }{(1+\tau)}   \cdot \E_{k} (\tau) \cdot   d\tau    \; .
\eea
Here $\E_{k-1} = 0$ when $k=0$\,. Thus, by choosing $\eps$ and the initial data $\E_{k} (0)  $ small enough, we would be able to upgrade \eqref{aprioriestimate} for a certain $k$ large enough, without using an assumption \eqref{aprioriestimate} on higher derivatives than $k$\,. One recurrent argument, among others, that we use in order to close the bootstrap (i.e. not involve more derivatives than what we upgrade for), is to use the fact that when one differentiates a product, one term does not get more than half of the derivatives, and we can then use the bootstrap assumption on that term, using only half of the derivatives and a certain loss of derivatives. 

We note that it is crucial to have the factor $\frac{1}{(1+\tau)}$ infront of the leading term with the highest number of derivatives, that is $\E_{k} (\tau) $\,. A a factor with an $\eps$ loss in the decay rate, namely $\frac{\eps  }{(1+\tau)^{1-      c \cdot \eps } }$ is tolerable only for the lower order derivatives $\E_{k-1} (\tau) $\,, so that we could upgrade the bootstrap assumption \eqref{aprioriestimate}. Also, it is important to have an $\eps$\,, or quantities that could be chosen arbitrarily small, for the other constants, as this is how we can \textit{upgrade}\,.
We are going to show that there exists a certain $\de > 0 $ in \eqref{aprioriestimate}, to be determined later, and that there exists a certain $\eps > 0$ that can be chosen arbitrarily small, such that \eqref{Theinequalityontheenergytobeusedtoapplygronwallrecursively} holds for $k \geq K$\,, without using the bootstrap assumption \eqref{aprioriestimate} on more than $k$ derivatives. Therefore, for $\eps$ small enough and for initial data  $\E_{k}$ small enough, one could upgrade \eqref{aprioriestimate}. It is however, crucial that there exists $K$ such that \eqref{Theinequalityontheenergytobeusedtoapplygronwallrecursively} holds for all $k \geq K$ without using an assumption on the growth of higher derivatives than $k$\,. This is simply called closing a bootstrap argument.

\subsection{The idea of the proof}\

In fact, our strategy to establish \eqref{Theinequalityontheenergytobeusedtoapplygronwallrecursively}, is to use the bootstrap assumption and our new decoupled energy estimates from Subsection \ref{The decoupled higher order energy estimates}, in order to transform the $\eps$ factor in \eqref{aprioriestimate} into a power that controls the growth of the energy, i.e. to prove first that \eqref{aprioriestimate} implies a better estimate that is
\bea\label{non-upgradedimprovedgrowthonenergy}
\E_{ k } (t)  \leq   E(k + k_0) \cdot \eps \cdot (1 +t)^{c\cdot \eps} \;.
\eea
with a loss of $k_0$ derivatives, thus not closing yet the bootstrap argument. Here the constant $ E(k+k_o) := 1$ is just a notation to tell us that we used our bootstrap assumption \eqref{aprioriestimate}  on no more than $k+k_0$ derivatives, thus, such an estimate is not yet an upgrade since we use more derivatives than what we upgrade for, namely $k$ derivatives on the left hand side of \eqref{non-upgradedimprovedgrowthonenergy}. The way we achieve this is by proving first a Gr\"onwall inequality of the type \eqref{Theinequalityontheenergytobeusedtoapplygronwallrecursively} while using more derivatives than what we upgrade for. Then, we we are going to use these \`a priori improved decay, with a $k_0$ loss of derivatives, to actually close the bootstrap argument. In other words, we are going to use \eqref{non-upgradedimprovedgrowthonenergy} to prove that
\bea\label{upgradedenergyestimate}
\E_{ K } (t)  \leq  \frac{1}{2} \cdot E(K) \cdot \eps \cdot (1 +t)^{c\cdot \eps} \;.
\eea
for $K$ large enough, hence allowing us to close the bootstrap argument. Hence, it is crucially important to keep track on the number of derivatives on which we use the bootstrap assumption \eqref{aprioriestimate}.

In order to prove the first improved \`a priori bound namely \eqref{non-upgradedimprovedgrowthonenergy}, we are going to upgrade first the energy for some good components with a certain weight $w_{\ga^{\prime}}$ higher than the weight used in the bootstrap assumption $w_{\ga}$\,, and then use this upgrade to establish \eqref{non-upgradedimprovedgrowthonenergy} with the weight $w_{\ga}$\,. To achieve this, recall that a Gr\"onwall lemma allows one to transform a factor into a power, so if we prove a suitable Gr\"onwall inequality on the energy where $\eps$ enters as a factor, we would be able to upgrade our bootstrap assumption to prove \eqref{upgradedenergyestimate}.
Then, we are going to use the structure of the equations and the crucial fact that when one differentiates a product, one term does not get more than half of the derivatives, to prove that \eqref{non-upgradedimprovedgrowthonenergy} implies \eqref{upgradedenergyestimate}.

\subsection{The decoupled higher order energy estimates}\label{The decoupled higher order energy estimates}\

The way we are going to prove \eqref{Theinequalityontheenergytobeusedtoapplygronwallrecursively} is by using weighted and decoupled energy estimates. For this we need the following definitions for the weights.

\begin{definition}
Let $q := r - t $\,, where $t$ and $r$ are defined using our fixed system of coordinates.
\end{definition}
We recall that the weight $w$ was defined in Definition \ref{defoftheweightw} for some $\gamma > 0$\,.

\begin{definition}\label{defofthenewweightwgammaprime}
We define the weight $w_{\ga^\prime}$\,, by
\bea
w_{\gamma^{\prime}}(q):=\begin{cases} (1+|r-t|)^{1+2\gamma^{\prime}} \quad\text{when }\quad q>0 \;, \\
         1 \,\quad\text{when }\quad q<0 \; , \end{cases}
\eea
for some $\gamma^{\prime} > 0 $ that will be determined later.
\end{definition}

\begin{definition}\label{defwidehatw}
We define $\widehat{w}_{\gamma^{\prime}}$ by 
\beaa
\widehat{w}_{\gamma^{\prime}}(q)&:=&\begin{cases} (1+|q|)^{1+2\gamma^{\prime}} \quad\text{when }\quad q>0 , \\
        (1+|q|)^{2\mu}  \,\quad\text{when }\quad q<0 , \end{cases} 
\eeaa
for $\gamma^{\prime} > 0$ and $\mu < 0$\,. We then have 
\bea\label{derivativeoftheweightintermsoftheweight}
\widehat{w}_{\gamma^{\prime}}^{\prime}(q) := \frac{\pa \widehat{w}_{\gamma^{\prime}}}{\pa q}  (q) \sim \frac{\widehat{w}_{\gamma^{\prime}}(q)}{(1+|q|)} \; .
\eea
\end{definition}
We recapitulate the following two lemmas from \cite{G5}-\cite{G6}.

\begin{lemma}
For 
\bea
| H| < \frac{1}{3} \; ,
\eea
where $H$ is defined in Definition \ref{definitionsofbigH}, and for $\ga^{\prime} > 0$ and $\mu < 0$\,, and for $\Phi^{(l)}_{V^l_1 \ldots  V^l_{k_l}}$ decaying sufficiently fast at spatial infinity, for a certain family of fixed vectors $V^l_1 \ldots  V^l_{k_l}$\,, we have the following energy estimate
    \bea\label{energyestimatewithoutestimatingthetermsthatinvolveBIGHbutbydecomposingthemcorrectlysothatonecouldgettherightestimate}
   \notag
 &&     \int_{\Sigma^{ext}_{t} }  | \derm  \Lie_{Z^I} \Phi^{(l)}_{V^l_1 \ldots  V^l_{k_l}} |^2     \cdot   w_{\gamma^{\prime}} (q)  \cdot d^{3}x     + \int_{0}^{t}  \int_{\Sigma^{ext}_{\tau} }      | \rderm_t  \Lie_{Z^I} \Phi^{(l)}_{V^l_1, \ldots , V^l_{k_l}}  |^2    \cdot d\tau \cdot   \widehat{w}_{\gamma^{\prime}}^\prime (q) d^{3}x \\
  \notag
  &\les &       \int_{\Sigma^{ext}_{0} }  |\derm  \Lie_{Z^I} \Phi^{(l)}_{V^l_1, \ldots , V^l_{k_l}} |^2     \cdot   w_{\gamma^{\prime}}(q)  \cdot d^{3}x \\
    \notag
     && +  \int_{0}^{t}  \int_{\Sigma^{ext}_{\tau} } \Big(  | H_{LL } | \cdot | \derm  \Lie_{Z^I} \Phi^{(l)}_{V^l_1, \ldots , V^l_{k_l}} |^2 + |  H | \cdot | \rderm  \Lie_{Z^I} \Phi^{(l)}_{V^l_1, \ldots , V^l_{k_l}} | \cdot | \derm  \Lie_{Z^I} \Phi^{(l)}_{V^l_1, \ldots , V^l_{k_l}}  |  \Big)  \cdot d\tau \cdot   \widehat{w}_{\gamma^{\prime}}^\prime (q) d^{3}x \\
 \notag
&& + \int_{0}^{t}  \int_{\Sigma^{ext}_{\tau} }  \Big( ( | \derm H_{LL} |  + |\rderm H| ) \cdot | \derm  \Lie_{Z^I} \Phi^{(l)}_{V^l_1, \ldots , V^l_{k_l}} |^2 \\
   \notag
&& \quad \quad \quad \quad\quad   +  | \derm H | \cdot  | \rderm  \Lie_{Z^I} \Phi^{(l)}_{V^l_1, \ldots , V^l_{k_l}} | \cdot  | \derm  \Lie_{Z^I} \Phi^{(l)}_{V^l_1, \ldots , V^l_{k_l}} | \Big)  \cdot d\tau \cdot   w_{\gamma^{\prime}} (q) d^{3}x \; \\
   &&  +  \int_{0}^{t}  \int_{\Sigma^{ext}_{\tau} }  |  g^{\mu\a} \derm_{\mu } \derm_\a  \Lie_{Z^I} \Phi^{(l)}_{V^l_1, \ldots , V^l_{k_l}} | \cdot |  \derm_t  \Lie_{Z^I} \Phi^{(l)}_{V^l_1, \ldots , V^l_{k_l}} |  \cdot d\tau \cdot  w_{\gamma^{\prime}}(q) d^{3}x  \, . 
 \eea
 
 \end{lemma}

 \begin{lemma}
 For any family of vectors $V^{l}_1 \in {\cal S}^{l}_1 \,, \ldots \,, V^{l}_{k_{l}} \in  {\cal S}^{l}_{k_{l}}$\,, where we recall that these sets are defined to be either $\cal T$ or $\cal U$\,, we have 
  \bea\label{Theseperatecommutatortermestimateforthetangentialcomponents}
\notag
&&| g^{\la\mu}    \derm_{\la}   \derm_{\mu} \Lie_{Z^I} \Phi^{(l)}_{V^l_1 \ldots  V^l_{k_l}}   - \Lie_{Z^I}  ( g^{\la\mu} \derm_{\la}   \derm_{\mu}  \Phi^{(l)}_{V^l_1 \ldots  V^l_{k_l}} )  |  \\
  \notag
   &\les&   \sum_{|K| \leq |I| -1 }  | g^{\la\mu} \cdot \derm_{\la}   \derm_{\mu}  \Lie_{Z^K} \Phi^{(l)}_{V^l_1 \ldots  V^l_{k_l}} | \\
   \notag
&&+  \frac{1}{(1+t+|q|)}  \cdot \sum_{|K|\leq |I|,}\,\, \sum_{|J|+(|K|-1)_+\le |I|} \,\,\, | \Lie_{Z^{J}} H |\, \cdot | \derm \Lie_{Z^K} \Phi^{(l)}_{V^l_1 \ldots  V^l_{k_l}}  | \\
&& + \underbrace{  \frac{1}{(1+|q|)}}_{\text{bad factor}}  \cdot \sum_{|K|\leq |I|,}\,\, \sum_{|J|+(|K|-1)_+\le |I|} \,\,\, | \Lie_{Z^{J}} H_{L  L} |\, \cdot  \underbrace{\big( \sum_{ V^{\prime}_1 \in {\cal S}^{l}_1 \,, \ldots \,, V^{\prime}_{k_{l}} \in  {\cal S}^{l}_{k_{l}} } | \derm  \Lie_{Z^I} \Phi^{(l)}_{V^\prime_1 \ldots  V^\prime_{k_l}}   | \:  \big) }_{\text{decoupled tangential components}} \; ,
\eea
  where $(|K|-1)_+=|K|-1$ if $|K|\geq 1$ and $(|K|-1)_+=0$ if $|K|=0$\,. 
   \end{lemma}

The fact that we can decouple the controls for the $L^2$-norm on \textit{each component} $\derm \Lie_{Z^I}  \Phi_{V^l_1 \ldots  V^l_{k_l}} $ is necessary to treat the terms $   \Phi^{(l_i)}_{{\cal S}^{l_i}_1 \ldots {\cal S}^{l_i}_{k_{l_i}} } \cdot \Phi^{(l_j)}_{{\cal S}^{l_j}_1 \ldots {\cal S}^{l_j}_{k_{l_j}} }  $ and $\Phi^{(l_i)}_{{\cal S}^{l_i}_1 \ldots {\cal S}^{l_i}_{k_{l_i}} } \cdot \derm \Phi^{(l_j)}_{{\cal S}^{l_j}_1 \ldots {\cal S}^{l_j}_{k_{l_j}} } $\,. However, to successfully achieve this, we crucially need to deal with the fact that the commutator estimate does not decouple, i.e. with the fact that the estimate on the term
\bea\label{commutatorterm}
 |g^{\a\b} \derm_{\a } \derm_\b  \Lie_{Z^I} \Phi^{(l)}_{V^l_1 \ldots  V^l_{k_l}} - \Lie_{Z^I} ( g^{\a\b} \derm_{\a } \derm_\b  \Phi^{(l)}_{V^l_1 \ldots  V^l_{k_l}}  ) | \; ,
 \eea
contains terms that involve other components different than $ \Lie_{Z^I} \Phi^{(l)}_{V^l_1 \ldots  V^l_{k_l}} $\,. We provide a new suitable commutator estimate in \eqref{Theseperatecommutatortermestimateforthetangentialcomponents} that we proved on \cite{G6}. Indeed, unlike the case of the Einstein vacuum equations, and also unlike the case of the Einstein-Maxwell system, in the case of this new class of Einstein equations coupled to non-linear matter, we need to have a \textit{decoupled} estimate for \textit{each component} of the commutator term \eqref{commutatorterm}\;.

We solved the problem by proving in \cite{G6}, a suitable commutator estimate \eqref{Theseperatecommutatortermestimateforthetangentialcomponents}, where we decoupled the components with respect to the sets $\cal T$ and $\cal U$\,. With this \textit{new} estimate \eqref{Theseperatecommutatortermestimateforthetangentialcomponents}, we notice that the terms with the bad factor $\frac{1}{(1+|q|)}$ enter \textit{only} with the tangential components. In fact, this is the bad term in the estimate of the commutator term, that is behind an imposed troublesome Grönwall inequality in the proof that follows, that needs to be dealt with separately. The estimate used previously in the literature contains all the components for the factor $\frac{1}{(1+|q|)}$\,, and does \textit{not} allow us to decouple the higher order energy estimates.
  
Instead, we were able to establish this new estimate \eqref{Theseperatecommutatortermestimateforthetangentialcomponents} in \cite{G6}, thanks to the following inequality. For an arbitrary tensor $\Psi$ and for an vector $ U \in {\cal U}$ and for any family of vectors $V_1 \in {\cal S}_1 \,, \ldots \,, V_{k} \in  {\cal S}_{k}$\,, where each of the sets ${\cal S}_{l}$ is either ${\cal T}$ or ${\cal U}$\,, we have
 \bea
 |\derm \Psi_{UV_1\ldots V_k } | \les \sum_{|I| \leq 1}  \frac{1}{(1+t+|q|)} \cdot | \Lie_{Z^I} \Psi |  +   \sum_{ U^{\prime} \in {\cal U}\,, V^{\prime}_1 \in {\cal S}_1 \,, \ldots \,, V^{\prime}_{k} \in  {\cal S}_{k} }    \frac{1}{(1+|q|)} \cdot  | \Lie_{Z^I} \Psi_{U^\prime V^\prime_1\ldots V^\prime_k } | \; . 
 \eea

  Whereas the terms that contain the full components in \eqref{Theseperatecommutatortermestimateforthetangentialcomponents}, have the good factor $\frac{1}{(1+t+|q|)} $.
  Hence, we get an estimate on the commutator term of the higher order energy, where the term with the weak factor is insensitive to the bad components ${\underline{L}}$ that bothered us. Thus, if the components ${\underline{L}}$ do not satisfy the good wave equation \eqref{goodtensorialcoupledwaveequation}, we can decouple the higher order energy estimates for the good tangential components. If the component ${\underline{L}}$  is in the set that verifies the good wave equation, then the decoupling is not needed for that set. This is why the sets for the good wave equations \eqref{goodtensorialcoupledwaveequation} are either $\cal T$ or $\cal U$\,.

However, for the energy estimate to be successful, it is of crucial importance that the bounds in the space-time integrand on the right hand side of the inequality are the correct bounds. It is the case in our bounds since we have the following 

\bea\label{goodbehaviourofthemetricinthebulkofenergyestimates}
    \notag
     && \int_{t_1}^{t_2}  \int_{\Sigma^{ext}_{\tau} }   \overbrace{| H_{LL } |}^{\text{good component}} \cdot  \overbrace{| \derm  \Lie_{Z^I} \Phi^{(l)}_{V^l_1 \ldots  V^l_{k_l}}  |^2 }^{\text{decoupled bad derivatives}}     \cdot d\tau \cdot   \widetilde{w} ^\prime (q) d^{3}x \\
  \notag   
     && + \int_{t_1}^{t_2}  \int_{\Sigma^{ext}_{\tau} }  \overbrace{|  H |}^{\text{bad component}}  \cdot    \overbrace{ | \rderm   \Lie_{Z^I} \Phi^{(l)}_{V^l_1 \ldots  V^l_{k_l}}  | }^{\text{good derivative}}  \cdot  \overbrace{| \derm  \Lie_{Z^I}  \Phi^{(l)}_{V^l_1 \ldots  V^l_{k_l}}  |}^{\text{decoupled bad derivative}}  \Big)  \cdot d\tau\cdot   \widetilde{w} ^\prime (q) d^{3}x \; .  \\
 \notag
&& + \int_{t_1}^{t_2}  \int_{\Sigma^{ext}_{\tau} } \Big( \overbrace{ | \derm H_{LL} | }^{\text{good component}} +  \overbrace{| \rderm H  |}^{\text{good derivative}}  \Big) \cdot  \overbrace{| \derm \Phi^{(l)}_{V^l_1 \ldots  V^l_{k_l}}  |^2}^{\text{decoupled bad derivatives}} \cdot d\tau \cdot   \widetilde{w} (q) d^{3}x \\
     && + \int_{t_1}^{t_2}  \int_{\Sigma^{ext}_{\tau} }  \underbrace{ | \derm H |}_{\text{bad derivative}}   \cdot    \underbrace{| \rderm  \Lie_{Z^I}  \Phi^{(l)}_{V^l_1 \ldots  V^l_{k_l}}   |}_{\text{good derivative}}  \cdot   \underbrace{ | \derm  \Lie_{Z^I}  \Phi^{(l)}_{V^l_1 \ldots  V^l_{k_l}}  |}_{\text{decoupled bad derivative}} \Big)  \cdot d\tau \cdot   \widetilde{w} (q) d^{3}x \; .
\eea
Also, in the commutator estimate \eqref{Theseperatecommutatortermestimateforthetangentialcomponents}, we have for $V^{l}_1 \in {\cal S}^{l}_1 \,, \ldots \,, V^{l}_{k_{l}} \in  {\cal S}^{l}_{k_{l}}$\,, where the sets are defined as in Definition \ref{definitionofthesetsaseithertangialorthewholeset}, that 

 \bea
\notag
&&  \overbrace{\frac{1}{(1+t+|q|)} }^{\text{good factor}}  \cdot \sum_{|K|\leq |I|,}\,\, \sum_{|J|+(|K|-1)_+\le |I|} \,\,\,  \overbrace{ | \Lie_{Z^{J}} H |\, \cdot | \derm ( \Lie_{Z^K}  \Phi )  |}^{\text{bad components}} \\
\notag
&& + \underbrace{  \frac{1}{(1+|q|)}}_{\text{bad factor}}  \cdot \sum_{|K|\leq |I|,}\,\, \sum_{|J|+(|K|-1)_+\le |I|} \,\,\,  \underbrace{  | \Lie_{Z^{J}} H_{L  L} |}_{\text{good component}} \, \cdot  \underbrace{\big( \sum_{ V^{\prime}_1 \in {\cal S}^{l}_1 \,, \ldots \,, V^{\prime}_{k_{l}} \in  {\cal S}^{l}_{k_{l}} }  | \derm ( \Lie_{Z^K} \Phi^{(l)}_{V^\prime_1 \ldots  V^\prime_{k_l}}   )  | \:  \big) }_{\text{decoupled tangential components}} \; . \\
\eea

   \section{\`A priori pointwise decay on the fields}\

We also recall the Klainerman-Sobolev inequality, see \cite{Kl1}, that we crucially need in conjunction with our energy estimates, to translate an \`a priori bound on the energy, at the level of the $L^2$ of the derivatives  of the fields, into an \`a priori pointwise bound on the fields. Indeed, the \textit{weighted} version of the Klainerman-Sobolev embedding gives, see \cite{G4}, for a tensor of arbitrary order
   \bea\label{weightedKlainermanSoboloevinequality}
\notag
| \derm \Lie_{Z^I} \Phi  | \cdot (1+t+|q|) \cdot \big[ (1+|q|) \cdot w_{\gamma^{\prime}}(q)\big]^{1/2} \leq
C \cdot \sum_{|K|\leq |I| + 2 } \|\big(w_{\gamma^{\prime}}(q)\big)^{1/2} \derm \Lie_{Z^K} \Phi (t,\cdot)\|_{L^2 (\Sigma^{ext}_{t} ) } \; , \\
\eea
where here $\ga^\prime $ is generic, i.e. arbitrary and therefore the inequality can be used when $\ga^\prime $ is equal to $\ga$. The weighted Klainerman-Sobolev inequality allows us to translate \`a priori bounds on the energy into \`a priori bounds on fields. 

We also have the following estimate
\bea\label{goodderivative}
(1 + t + |q| ) \cdot |\rderm \Psi | + (1 +  |q| ) \cdot  |\derm \Psi | &\les& \sum_{|I| \leq 1} | \Lie_{Z^I}  \Psi | \, .
\eea

\begin{lemma}\label{estimateonthesourcetermsforhzerothesphericallsymmtrpart}
 We have for all $|I| \leq K $,       
 \beaa
 \notag
|   \Lie_{Z^I} h^0 (t,x)  |   &\les&   \frac{M }{(1+t+|q|) }  \; .
      \eeaa
      and
 \beaa
 \notag
|\derm  \Lie_{Z^I} h^0 (t,x)  |   &\les&   \frac{M }{(1+t+|q|)^{2} } \; .
      \eeaa
Let $M \leq \eps$\,, then we have for all $|I| \leq K $,
\beaa
| \Lie_{ Z^I}  ( g^{\la\mu} \derm_{\la}   \derm_{\mu}    h^0 ) | &\les& E ( |I| + 2)  \cdot \begin{cases} c (\gamma) \cdot \frac{\eps }{(1+t+|q|)^{4-\delta}(1+|q|)^{\delta}},\quad q>0 \;, \\
\frac{\eps}{ (1+t+|q|)^{3}} ,\quad q<0 \;.
\end{cases}
\eeaa

\end{lemma}

Based on the bootstrap assumption \eqref{aprioriestimate}, on the weighted Klainerman-Sobolev inequality \eqref{weightedKlainermanSoboloevinequality}, and on \eqref{goodderivative}, we have the following \`a priori pointwise decay on the fields.

\begin{lemma}\label{aprioridecayestimates}
Let $M\leq \eps$\,, $\gamma > \delta > 0 $\,, and $0< \de \leq 1$\,. We have for all $|I| \leq K $, in the entire exterior region, the following estimates for all $l \in \{1, \ldots, N \} $,
   \beaa
\notag
 |\derm \Lie_{Z^I} \Phi^{(0)}   (t,x)  |    &\leq&  E ( |I| + 2)  \cdot \frac{\eps }{(1+t+|q|)^{1-\delta} \cdot (1+|q|)^{1+\de}} \,,  \\
\notag
 |\derm   \Lie_{Z^I} \Phi^{(l)} (t,x)  |     + |\derm (  \Lie_{Z^I} \Phi^{(0)} - h^0 ) (t,x)  |    + &\leq& E ( |I| + 2)  \cdot \frac{\eps }{(1+t+|q|)^{1-\delta} \cdot (1+|q|)^{1+\ga}} \,, \\
 \notag
|   \Lie_{Z^I} \Phi^{(0)}  (t,x)  |   &\leq&  E ( |I| + 2) \frac{\eps }{(1+t+|q|)^{1-\delta}  \cdot (1+|q|)^{\de} } \,, \\
\notag
   |  \Lie_{Z^I} \Phi^{(l)} (t,x)  |  + |  \Lie_{Z^I} (\Phi^{(0)} - h^0 ) (t,x)  |  &\leq&   E ( |I| +2)  \cdot \frac{\eps }{(1+t+|q|)^{1-\delta} \cdot (1+|q|)^{ \ga}} \,, \\
   \notag
 |  \rderm \Lie_{Z^I} \Phi^{(0)}  (t,x)  |   &\leq& E ( |I| + 3)  \cdot \frac{\eps }{(1+t+|q|)^{2-\delta} \cdot (1+|q|)^{\de} } \,, \\
 \notag
  |  \rderm  \Lie_{Z^I} \Phi^{(l)}  (t,x)  |  + |\rderm (  \Lie_{Z^I} \Phi^{(0)} - h^0 ) (t,x)  |    &\leq& E ( |I| + 3)  \cdot \frac{\eps }{(1+t+|q|)^{2-\delta} \cdot (1+|q|)^{\ga} } \,. \\
      \eeaa
      
    Based on Remark \ref{estimatesonPhioalsoholdforbigH}, the estimates on $\Phi^{(0)}$ also hold for $H$\,.
      \end{lemma}

\begin{proof}
Using the weighted Klainerman-Sobolev inequality with the weight $w_{\gamma}$\, we get \`a priori pointwise estimates on $\derm \Lie_{Z^I} \Phi^{(l)} $ and $\derm \Lie_{Z^I} ( \Phi^{(0)} - h^0)$\,, and by integrating we get \`a priori pointwise estimates on $\Lie_{Z^I} \Phi^{(l)} $ or $\Lie_{Z^I}( \Phi^{(0)} - h^0)$. By using \eqref{goodderivative}, we get the estimates on $ \rderm  \Lie_{Z^I} \Phi^{(l)}  $. Using Lemma \ref{estimateonthesourcetermsforhzerothesphericallsymmtrpart}, we get the estimates on $\derm \Lie_{Z^I} \Phi^{(0)} $ and $\Lie_{Z^I} \Phi^{(0)} $\,. Then, using Lemma \ref{estimateonthesourcetermsforhzerothesphericallsymmtrpart}, we get the estimates on $h^0$ and therefore on $\derm \Lie_{Z^I} ( \Phi^{(0)}$  and $\Lie_{Z^I} \Phi^{(0)} $ and on $ \rderm  \Lie_{Z^I} \Phi^{(0)}  $. For details, see \cite{G6}.

\end{proof}

Based on our \`a priori decay estimates from Lemma \ref{aprioridecayestimates}, as well as the condition on our metric \eqref{wavecoordinatesestimateonLiederivativesZonmetric}, we have the following improved \`a priori decay estimates for some components of the metric. 

\begin{lemma}\label{estimategoodcomponentspotentialandmetric}
Let $M\leq \eps$\,. Under the bootstrap assumption for $\eps \leq 1$\,, $\gamma > \delta > 0 $ and for $0< \de \leq \frac{1}{4}$, we have both in the interior and in the exterior regions, 
                    \bea
        \notag
\frac{1}{(1+|q|)} \cdot |  \Lie_{Z^I} h_{\cal T L} |    + \frac{1}{(1+|q|)} \cdot |  \Lie_{Z^I} H_{\cal T L} |                    &\les&   E ( |I| + 3)  \cdot \frac{ \eps   }{ (1+t+|q|) } \; .
       \eea
In the exterior region, we have 
                    \bea
        \notag
 |  \Lie_{Z^I} h_{\cal T L} |    +  |  \Lie_{Z^I} H_{\cal T L} |                    &\les&   E ( |I| + 3)  \cdot \frac{ \eps   }{ (1+t+|q|) } \; .
       \eea
\end{lemma}

\begin{proof}
The conditions on our metric in our fixed system of coordinates, \eqref{wavecoordinatesestimateonLiederivativesZonmetric}, controls $ | \derm_{\underline{L}} ( \Lie_{Z^J}  H_{\cal T L} )  |$ and given that for all $U \in {\cal U}$, we have $\derm_{\underline{L}} V = 0$, this controls then $ | \pa_{\underline{L}} ( \Lie_{Z^J}  H_{\cal T L} )  |$. By integration along $q$ as in \cite{G4}, we could use this to translate these into estimates on $ \Lie_{Z^J}  H_{\cal T L}   $ and $ \Lie_{Z^J}  h_{\cal T L}  $\,, as follows
\bea\label{partialderivativeofLiederivativeofsmallhTangentialLcompo}
 | \pa  \Lie_{Z^I}  h_{\cal T L}  |   &\les&  \sum_{|J|\leq |I|} |   \rderm  \Lie_{Z^J} h |   + \sum_{|K|+ |M| \leq |I|}  O (|\Lie_{Z^K} h| \cdot |\derm ( \Lie_{Z^M} h )| )     \; ,
\eea
and
\bea\label{partialderivativeofLiederivativeofHTangentialLcompo}
 | \pa  \Lie_{Z^I}  H_{\cal T L}  |   &\les&   \sum_{|J|\leq |I|} |   \rderm  \Lie_{Z^J} H |  + \sum_{|K|+ |M| \leq |I|}  O (|\Lie_{Z^K} H| \cdot |\derm ( \Lie_{Z^M} H )| )    \; .
\eea
Now, we can use Lemma \ref{aprioridecayestimates} and obtain the desired result for $M\leq \eps$, and for $\de \leq \frac{1}{4}$, 
         \beaa
       \frac{1}{(1+|q|)} \sum_{|I|\leq \, |J|}| \Lie_{Z^I} H_{LL}|    &\les&     E ( |J| + 3) \cdot  \begin{cases}   \frac{\eps}{ (1+ t + | q | ) (1+  | q | )}  ,\quad\text{when }\quad q>0 \; ,\\
\notag
     \frac{\eps}{ (1+ t + | q | ) (1+  | q | )^{\frac{1}{2}-2\de}}  , \,\quad\text{when }\quad q<0 \; , \end{cases} \\
\notag
 &\leq&    E (  |J|  +3)  \cdot \frac{\eps }{(1+t+|q|) }     \; .
      \eeaa
\end{proof}

\begin{remark}
In our application of the decoupled energy estimates to prove dispersive estimates for our specific coupled system \eqref{badtensorialcoupledwaveequation}-\eqref{goodtensorialcoupledwaveequation}, we only need the estimates in the exterior region, since the \textit{new} non-linearities for the coupled governing system \eqref{badtensorialcoupledwaveequation}-\eqref{goodtensorialcoupledwaveequation} that we are going to deal with, do not decay in the interior region. However, we have shown here the estimate on the metric in the interior region, to demonstrate that our decoupling of the higher order energy estimates works also in the interior region. It is only at the level of the new non-linearities that the argument does not work in the interior region, and it is not the decoupling of the energy estimates itself.
\end{remark}

\section{The decoupled energy estimates for the tangential components}\

We are going to use first our decoupled energy estimates, using \eqref{energyestimatewithoutestimatingthetermsthatinvolveBIGHbutbydecomposingthemcorrectlysothatonecouldgettherightestimate} and \eqref{Theseperatecommutatortermestimateforthetangentialcomponents}, that are decoupled for the tangential components, to prove a decoupled energy estimate for \eqref{goodtensorialcoupledwaveequation}. More precisely, the goal of this Section is prove Proposition \ref{Theveryfinal }.

\begin{definition}\label{Definitionofgoodenergyandbadenergywithdifferentweightswiththefieldsspeaaretly}
We define the decoupled, or separate, higher order energy norm for the good components that satisfy the equation \eqref{goodtensorialcoupledwaveequation}, as follows, for $l \in \{1, \ldots, N \}$\,, 
\bea\label{thedecoupledeneryforthegoodcomponents}
\notag
 \E^{\textit{good}\,,\, \gamma^{\prime}}_{k} (\Phi^{(l)})  (t) &:=&  \sum_{|J|\leq k} \; \;   \sum_{V^{l}_1 \in {\cal S}^{l}_1 \,, \ldots \,, V^{l}_{k_{l}} \in  {\cal S}^{l}_{k_{l}}  } \|w_{\gamma^{\prime}}^{1/2}   \derm \Lie_{Z^J}  \Phi^{(l)}_{V^l_1 \ldots  V^l_{k_l}}   (t,\cdot)   \|^2_{L^2 (\Sigma_t^{ext} )  }   \; ,\\
 \notag
  \E^{\textit{good}\,,\, \gamma^{\prime}}_{k} (\Phi^{(0)} - h^0)  (t) &:=& \sum_{|J|\leq k} \; \;   \sum_{V^{0}_1 \in {\cal S}^{0}_1 \,, V^{0}_{2} \in  {\cal S}^{0}_{2}  } \|w_{\gamma^{\prime}}^{1/2}   \derm \Lie_{Z^J} (  \Phi^{(0)} - h^0) _{V^0_1V^0_{2}}   (t,\cdot)   \|^2_{L^2 (\Sigma_t^{ext} )  }   \; ,\\
\eea
and we define the energy for the full components as 
\bea\label{thedecoupledeneryforthefullcomponents}
\notag
 \E^{full\,,\, \gamma^{\prime}}_{k}  (\Phi^{(l)}) (t) &:=&  \sum_{|J|\leq k} \; \;  \sum_{V^{l}_1 \in {\cal U} \,, \ldots \,, V^{l}_{k_{l}} \in  {\cal U}  } \|w_{\gamma^{\prime}}^{1/2}   \derm  \Lie_{Z^J}  \Phi^{(l)}_{V^l_1 \ldots  V^l_{k_l}}   (t,\cdot)   \|^2_{L^2 (\Sigma_t^{ext} )  }  \; .\\
 \notag
  \E^{\textit{full}\,,\, \gamma^{\prime}}_{k} (\Phi^{(0)} - h^0)  (t) &:=& \sum_{|J|\leq k} \; \;   \sum_{V^{0}_1 \in {\cal U} \,, V^{0}_{2} \in  {\cal U}  } \|w_{\gamma^{\prime}}^{1/2}   \derm  \Lie_{Z^J} (  \Phi^{(0)} - h^0) _{V^0_1V^0_{2}}   (t,\cdot)   \|^2_{L^2 (\Sigma_t^{ext} )  }   \; .\\
\eea
\end{definition}

\begin{proposition}\label{Theveryfinal }
Let $M \leq \eps$\,, $\ga \geq 3 \de $\,, $0 < \de \leq \frac{1}{4}$ and $\ga^\prime < 1+ \ga-\de$\,. For $\eps$ small enough, for $\mu $ small enough, and for $\de$ small enough depending on $\mu$\,, we get the following decoupled energy estimate for all $l \in \{1, \ldots, N \}$\,,

  \bea\label{decoupledenergyestimateongoodcomponentsforfieldslgretarto1}
     \notag
 &&     \E^{\textit{good}\,,\, \gamma^{\prime}}_{k}  (\Phi^{(l)})  (t)   +  \sum_{V^{l}_1 \in {\cal S}^{l}_1 \,, \ldots \,, V^{l}_{k_{l}} \in  {\cal S}^{l}_{k_{l}}  } \;\; \sum_{|I|\leq k}   \int_{0}^{t}  \int_{\Sigma^{ext}_{\tau} }      | \rderm_t  \Lie_{Z^I} \Phi^{(l)}_{V^l_1, \ldots , V^l_{k_l}}  |^2    \cdot d\tau \cdot   \frac{\widehat{w}_{\gamma^{\prime}} (q)}{(1+|q|)}  d^{3}x \\
  \notag
  &\les &E ( k + 3)  \cdot \big[    \E^{\textit{good}\,,\, \gamma^{\prime}}_{k}  (\Phi^{(l)})  (0)  +    \eps^2  \\
    \notag
  &&+  \sum_{V^{l}_1 \in {\cal S}^{l}_1 \,, \ldots \,, V^{l}_{k_{l}} \in  {\cal S}^{l}_{k_{l}}  }  \sum_{|I|\leq k}  \int_{0}^{t}  \int_{\Sigma^{ext}_{\tau} }  \frac{(1+\tau )}{\eps} \cdot  |  \Lie_{Z^I}  ( g^{\mu\a} \derm_{\mu } \derm_\a  \Phi^{(l)}_{V^l_1 \ldots  V^l_{k_l}} ) |^{2}    \cdot w_{\gamma^{\prime}}(q) \cdot  d^{3}x  \cdot d\tau \\
&& + \int_{0}^{t}   \frac{ \eps }{ (1+\tau)  }       \cdot \underbrace{\E^{\textit{good}, \gamma^{\prime}}_{k}  (\Phi^{(l)})  (\tau)}_{\text{decoupled energy}}  d\tau  \big] \; ,
 \eea
 
and

  \bea
     \notag
 &&     \E^{\textit{full}\,,\, \gamma^{\prime}}_{k}  (\Phi^{(l)})  (t)   +  \sum_{V^{l}_1 \in {\cal U} \,, \ldots \,, V^{l}_{k_{l}} \in  {\cal U} } \;\;  \sum_{|I|\leq k}   \int_{0}^{t}  \int_{\Sigma^{ext}_{\tau} }      | \rderm_t  \Lie_{Z^I} \Phi^{(l)}_{V^l_1, \ldots , V^l_{k_l}}  |^2    \cdot d\tau \cdot   \frac{\widehat{w}_{\gamma^{\prime}} (q)}{(1+|q|)}  d^{3}x \\
  \notag
  &\les &E ( k + 3)  \cdot \big[    \E^{\textit{full}\,,\, \gamma^{\prime}}_{k}  (\Phi^{(l)})  (0)  +    \eps^2  \\
    \notag
  &&+  \sum_{V^{l}_1 \in {\cal U} \,, \ldots \,, V^{l}_{k_{l}} \in  {\cal U} }  \sum_{|I|\leq k}  \int_{0}^{t}  \int_{\Sigma^{ext}_{\tau} }  \frac{(1+\tau )}{\eps} \cdot | \Lie_{Z^I} ( g^{\mu\a} \derm_{\mu } \derm_\a   \Phi^{(l)}_{V^l_1, \ldots , V^l_{k_l}} )  |^2   \cdot w_{\gamma^{\prime}}(q) \cdot  d^{3}x  \cdot d\tau \\
&& + \int_{0}^{t}   \frac{ \eps }{ (1+\tau)  }       \cdot \E^{\textit{full}, \gamma^{\prime}}_{k}  (\Phi^{(l)})  (\tau)  d\tau  \big] \; .  
 \eea

 Similarly, for $\Phi^{(0)} - h^0 $, using this time \eqref{Thesourceofthesphericallysymmetricpartinthebulktimesthederivative}, we also get
 
   \bea\label{decoupledenergyestimateongoodcomponentsforthemetricphiominusho}
     \notag
 &&     \E^{\textit{good}\,,\, \gamma^{\prime}}_{k}  (\Phi^{(0)} - h^0)  (t)   +  \sum_{V^{0}_1 \in {\cal S}^{0}_1 \,, V^{0}_{2} \in  {\cal S}^{0}_{2}  } \; \;  \sum_{|I|\leq k}   \int_{0}^{t}  \int_{\Sigma^{ext}_{\tau} }      | \rderm_t  \Lie_{Z^I} (\Phi^{(0)} - h^0)_{V^0_1 V^0_{2}}   |^2    \cdot d\tau \cdot   \frac{\widehat{w}_{\gamma^{\prime}} (q)}{(1+|q|)}  d^{3}x \\
  \notag
  &\les &E ( k + 3)  \cdot \big[    \E^{\textit{good}\,,\, \gamma^{\prime}}_{k}  (\Phi^{(l)})  (0)  +    \eps^2  \\
    \notag
  &&+ \sum_{V^{0}_1 \in {\cal S}^{0}_1 \,, V^{0}_{2} \in  {\cal S}^{0}_{2}  }    \sum_{|I|\leq k}  \int_{0}^{t}  \int_{\Sigma^{ext}_{\tau} }  \frac{(1+\tau )}{\eps} \cdot  |  \Lie_{Z^I}  ( g^{\mu\a} \derm_{\mu } \derm_\a  \Phi^{(0)}_{V^0_1 V^0_{2}}  ) |^{2}    \cdot w_{\gamma^{\prime}}(q) \cdot  d^{3}x  \cdot d\tau \\
&& + \int_{0}^{t}   \frac{ \eps }{ (1+\tau)  }       \cdot \underbrace{\E^{\textit{good}, \gamma^{\prime}}_{k}  (\Phi^{(0)} - h^0)  (\tau)}_{\text{decoupled energy}}  d\tau  \big] \; ,
 \eea
 
and

  \bea
     \notag
 &&     \E^{\textit{full}\,,\, \gamma^{\prime}}_{k}  (\Phi^{(0)} - h^0)  (t)   +  \sum_{V^{0}_1 \in {\cal U} \,, V^{0}_{2} \in  {\cal U} } \;\; \sum_{|I|\leq k}   \int_{0}^{t}  \int_{\Sigma^{ext}_{\tau} }      | \rderm_t  \Lie_{Z^I} (\Phi^{(0)} - h^0)_{V^0_1 V^0_{2}}   |^2    \cdot d\tau \cdot   \frac{\widehat{w}_{\gamma^{\prime}} (q)}{(1+|q|)}  d^{3}x \\
  \notag
  &\les &E ( k + 3)  \cdot \big[    \E^{\textit{full}\,,\, \gamma^{\prime}}_{k}  ((\Phi^{(0)} - h^0))  (0)  +    \eps^2  \\
    \notag
  &&+  \sum_{V^{0}_1 \in {\cal U} \,, V^{0}_{2} \in  {\cal U} }  \sum_{|I|\leq k}  \int_{0}^{t}  \int_{\Sigma^{ext}_{\tau} }  \frac{(1+\tau )}{\eps} \cdot |  \Lie_{Z^I} (g^{\mu\a} \derm_{\mu } \derm_\a   \Phi^{(0)}_{V^0_1 V^0_{2}}  ) |^2   \cdot w_{\gamma^{\prime}}(q) \cdot  d^{3}x  \cdot d\tau \\
&& + \int_{0}^{t}   \frac{ \eps }{ (1+\tau)  }       \cdot \E^{\textit{full}, \gamma^{\prime}}_{k}  (\Phi^{(0)} - h^0)  (\tau)  d\tau  \big] \; .  
 \eea

     \begin{remark}
We recall again that $E ( k + 3)  := 1$ is here only to remind us that we used our bootstrap assumption \eqref{aprioriestimate} on $k + 3$ derivatives of the fields, with the weight $w_{\ga}$\,.
\end{remark}
    
    \begin{remark}\label{noneedtoputinequalityongamprimeaslongastheyaretransferable}
In fact, we use $w_{\ga^\prime}$ as an auxiliary weight, that will later help us updated the bounds on the energy with the weight $w_{\ga}$\,.
\end{remark}

   \begin{remark}
   As we will see, the loss of 3 derivatives in $E ( k + 3)$, rather than having $E(k)$, comes from the commutator estimate (see Subsection \ref{estimatethecommutatortermusingbootstraponallderivatives}).
        \end{remark}
        
           \begin{remark}
           The restriction that $0 < \de \leq \frac{1}{4} $ could be improved in the exterior region. It comes from Lemma \ref{estimategoodcomponentspotentialandmetric} so that the estimate on the good components of the metric hold both in the interior and the exterior region, which would allow the decoupled energy estimate to hold both in the interior and the exterior region.
              \end{remark}
 \end{proposition}

To achieve this, first, we use \eqref{energyestimatewithoutestimatingthetermsthatinvolveBIGHbutbydecomposingthemcorrectlysothatonecouldgettherightestimate},
    \bea\label{decoupledenergyestimatewithoutcommutatorestimate}
   \notag
 &&      \int_{\Sigma^{ext}_{t} }  | \derm  \Lie_{Z^I} \Phi^{(l)}_{V^l_1 \ldots  V^l_{k_l}} |^2     \cdot   w_{\gamma^{\prime}} (q)  \cdot d^{3}x     + \int_{0}^{t}  \int_{\Sigma^{ext}_{\tau} }      | \rderm_t  \Lie_{Z^I} \Phi^{(l)}_{V^l_1 \ldots V^l_{k_l}}  |^2    \cdot d\tau \cdot   \widehat{w}_{\gamma^{\prime}}^\prime (q) d^{3}x \\
  \notag
  &\les &     \int_{\Sigma^{ext}_{0} }  | \derm  \Lie_{Z^I} \Phi^{(l)}_{V^l_1 \ldots  V^l_{k_l}} |^2     \cdot   w_{\gamma^{\prime}} (q)  \cdot d^{3}x    \\
    \notag
     && + \sum_{|I|\leq k} \big[  \int_{0}^{t}  \int_{\Sigma^{ext}_{\tau} } \Big(  | H_{LL } | \cdot | \derm  \Lie_{Z^I} \Phi^{(l)}_{V^l_1 \ldots V^l_{k_l}} |^2 + |  H | \cdot | \rderm  \Lie_{Z^I} \Phi^{(l)}_{V^l_1 \ldots  V^l_{k_l}} | \cdot | \derm  \Lie_{Z^I} \Phi^{(l)}_{V^l_1 \ldots  V^l_{k_l}}  |  \Big)  \cdot d\tau \cdot   \widehat{w}_{\gamma^{\prime}}^\prime (q) d^{3}x \\
 \notag
&& + \int_{0}^{t}  \int_{\Sigma^{ext}_{\tau} }  \Big( ( | \derm H_{LL} |  + |\rderm H| ) \cdot | \derm  \Lie_{Z^I} \Phi^{(l)}_{V^l_1\ldots  V^l_{k_l}} |^2 \\
 \notag
&& \quad \quad \quad \quad\quad   +  | \derm H | \cdot  | \rderm  \Lie_{Z^I} \Phi^{(l)}_{V^l_1 \ldots  V^l_{k_l}} | \cdot  | \derm  \Lie_{Z^I} \Phi^{(l)}_{V^l_1 \ldots V^l_{k_l}} | \Big)  \cdot d\tau \cdot   w_{\gamma^{\prime}} (q) d^{3}x \\
   &&  +  \int_{0}^{t}  \int_{\Sigma^{ext}_{\tau} }  |  g^{\mu\a} \derm_{\mu } \derm_\a  \Lie_{Z^I} \Phi^{(l)}_{V^l_1 \ldots V^l_{k_l}} | \cdot |  \derm_t  \Lie_{Z^I} \Phi^{(l)}_{V^l_1 \ldots  V^l_{k_l}} |  \cdot d\tau \cdot  w_{\gamma^{\prime}}(q) d^{3}x \; . \big] \;. 
 \eea

 For $ l \in \{1, \ldots, N \}$\,, we estimate the last term in \eqref{decoupledenergyestimatewithoutcommutatorestimate} as follows, 
\bea\label{decoupledestimeontgewaveoperatoroftheLierderivativestimesthederivative}
\notag
 && |  g^{\mu\a} \derm_{\mu } \derm_\a  \Lie_{Z^I} \Phi^{(l)}_{V^l_1 \ldots V^l_{k_l}} | \cdot |  \derm_t  \Lie_{Z^I} \Phi^{(l)}_{V^l_1 \ldots  V^l_{k_l}} |  \\
 \notag
   &\les& (1+\tau) \cdot  |  g^{\mu\a} \derm_{\mu } \derm_\a  \Lie_{Z^I} \Phi^{(l)}_{V^l_1 \ldots  V^l_{k_l}} |^{2} +  \frac{1}{(1+\tau)} |  \derm_t  \Lie_{Z^I} \Phi^{(l)}_{V^l_1 \ldots  V^l_{k_l}} |^{2} \\
   \notag
    &\les&   (1+\tau) \cdot  |  \Lie_{Z^I}  ( g^{\mu\a} \derm_{\mu } \derm_\a  \Phi^{(l)}_{V^l_1 \ldots  V^l_{k_l}} ) |^{2}  \\
    \notag
    && + \frac{ (1+\tau)}{\eps} \cdot  |g^{\a\b} \derm_{\a } \derm_\b  \Lie_{Z^I} \Phi^{(l)}_{V^l_1 \ldots  V^l_{k_l}} - \Lie_{Z^I} ( g^{\a\b} \derm_{\a } \derm_\b  \Phi^{(l)}_{V^l_1 \ldots  V^l_{k_l}}  ) |^{2} \\
    && +  \frac{\eps}{(1+\tau)} |  \derm_t  \Lie_{Z^I} \Phi^{(l)}_{V^l_1 \ldots  V^l_{k_l}} |^{2} \; . 
  \eea
 This how the commutator term enters, and this one reason, among others, on how we get the right factor $\frac{\eps}{(1+\tau)}$ infront of the highest number of Lie derivatives in \eqref{Theinequalityontheenergytobeusedtoapplygronwallrecursively}.

 Another way on how we get the right factor $\frac{1}{(1+\tau)}$ infront of the highest number of Lie derivatives, is that in the energy estimate, there are good components of the metric that decay better for the level of the zero Lie derivative, as pointed out in \eqref{goodbehaviourofthemetricinthebulkofenergyestimates}, and the other components that do not decay fast enough enter as a factor for lower Lie derivatives of the fields.

  However, for $l=0$\,, since the wave equation that we consider is that for $(\Phi^{(0)} -h^0)_{\mu\nu}$\,, see Remark \ref{remarkonthewaveequationforthemetricwherewesubstracttheSchwarzschildianpart}, and therefore the source terms have an additional term, namely $- g^{\alpha\beta}  \derm_{\a } \derm_\b  h^0_{\mu\nu} $\,, we instead decompose for $(\Phi^{(0)} -h^0)_{\mu\nu}$ as follows
 
 \bea\label{decoupledestimeontgewaveoperatoroftheLierderivativesofthemetrictimesthederivativeofmetric}
\notag
 && |  g^{\mu\a} \derm_{\mu } \derm_\a  \Lie_{Z^I} (\Phi^{(0)} -h^0)_{V^0_1 V^0_{2}} | \cdot |  \derm_t  \Lie_{Z^I} (\Phi^{(0)} -h^0)_{V^0_1 V^0_{2}}  |  \\
 \notag
  &\les& |  g^{\mu\a} \derm_{\mu } \derm_\a  \Lie_{Z^I} \Phi^{(0)}_{V^0_1 V^0_{2}} | \cdot |  \derm_t  \Lie_{Z^I} (\Phi^{(0)} -h^0)_{V^0_1 V^0_{2}}  | \\
  \notag
  && + |  g^{\mu\a} \derm_{\mu } \derm_\a  \Lie_{Z^I} h^{0}_{V^0_1 V^0_{2}} | \cdot |  \derm_t  \Lie_{Z^I} (\Phi^{(0)} -h^0)_{V^0_1 V^0_{2}}  | \, .
  \eea
 For the first term, we proceed as describved above in \eqref{decoupledestimeontgewaveoperatoroftheLierderivativestimesthederivative} using the sources for $\Phi^{(0)} $ prescribed from our coupled system of wave equations \eqref{badtensorialcoupledwaveequation}-\eqref{goodtensorialcoupledwaveequation}. For the second term, we use the fact that for $ M \leq  \eps^2 \leq \eps \leq 1 $\;, we have (see \cite{G6}),
 
\bea\label{Thesourceofthesphericallysymmetricpartinthebulktimesthederivative}
\notag
&& \int_{0}^{t} \int_{\Sigma^{ext}_{\tau} } | \Lie_{Z^I} \big( g^{\alpha\beta} \derm_\alpha \derm_\beta h^0 \big) | \cdot |\derm \Lie_{Z^I}  (\Phi^{(0)} -h^0)_{V^0_1 V^0_{2}}  |  w_{\gamma^{\prime}} (q) \, d^{3}x d \tau  \\
\notag
&\leq&   E (|I|) \cdot \eps^2  + E (|I|) \cdot  \eps^2 \cdot \!\! \sum_{|K|\leq |I|} \int_{0}^{t}    \frac{1}{(1+\tau) }    \int_{\Sigma^{ext}_{\tau} }  | \derm  \Lie_{Z^K} (\Phi^{(0)} -h^0)_{V^0_1 V^0_{2}}   |^2     \cdot   w_{\gamma^{\prime}} (q)  \cdot d^{3}x    d\tau  \; , \\
\eea
and this is how we get the right factor of $\frac{1}{(1+\tau) }$  for that term. 

Now, we are left to examine the term $|  g^{\mu\a} \derm_{\mu } \derm_\a  \Lie_{Z^I} \Phi^{(l)}_{V^l_1 \ldots V^l_{k_l}} | \cdot |  \derm_t  \Lie_{Z^I} \Phi^{(l)}_{V^l_1 \ldots  V^l_{k_l}} |$ when $ l \in \{1, \ldots, N \}$\,, or similarly for the metric the term $|  g^{\mu\a} \derm_{\mu } \derm_\a  \Lie_{Z^I} \Phi^{(0)}_{V^0_1 V^0_{2}} | \cdot |  \derm_t  \Lie_{Z^I} (\Phi^{(0)} -h^0)_{V^0_1 V^0_{2}}  |$\,. 
 
We recall that the commutator term can be estimated using our commutator estimate decoupled  for the tangential components \eqref{Theseperatecommutatortermestimateforthetangentialcomponents}, and hence, we get

\bea\label{Theenergyformofthedecoupledcommutatorestimate}
\notag
&&     \frac{ (1+t)}{\eps} \cdot  | g^{\la\mu}    \derm_{\la}   \derm_{\mu} \Lie_{Z^I} \Phi^{(l)}_{V^l_1 \ldots  V^l_{k_l}}   - \Lie_{Z^I}  ( g^{\la\mu} \derm_{\la}   \derm_{\mu}  \Phi^{(l)}_{V^l_1 \ldots  V^l_{k_l}} )  |^2  \\
  \notag
   &\les&  \sum_{|K| \leq |I| -1 }  \frac{ (1+\tau)}{\eps} \cdot  | g^{\la\mu} \cdot \derm_{\la}   \derm_{\mu}  \Lie_{Z^K} \Phi^{(l)}_{V^l_1 \ldots  V^l_{k_l}} |^2 \\
   \notag
&&+  \frac{(1+t)}{\eps \cdot (1+t+|q|)^2}  \cdot \sum_{|K|\leq |I|,}\,\, \sum_{|J|+(|K|-1)_+\le |I|} \,\,\, | \Lie_{Z^{J}} H |^2\, \cdot | \derm \Lie_{Z^K} \Phi^{(l)}_{V^l_1 \ldots  V^l_{k_l}}  |^2 \\
\notag
&& +  \frac{(1+t)}{\eps \cdot  \underbrace{  (1+|q|)^2 }_{\text{bad factor}} } \cdot \sum_{|K|\leq |I|,}\,\, \sum_{|J|+(|K|-1)_+\le |I|} \,\,\, | \Lie_{Z^{J}} H_{L  L} |^2\, \cdot \underbrace{  \big( \sum_{ V^\prime_1 \in {\cal V}^l_{1}, \ldots,  V^\prime_{k_l} \in {\cal V}^l_{k_l} } | \derm  \Lie_{Z^I} \Phi^{(l)}_{V^\prime_1 \ldots  V^\prime_{k_l}}   |^2 \:  \big) }_{\text{decoupled tangential components}}  \; ,\\
\eea
where
\beaa
{\cal V}^{l}_{k} = \begin{cases}  & { \cal{T}} \quad \text{when} \quad  V^l_k \in  {\cal{T}} , \\
& {\cal{U}} \quad  \text{when} \quad V^l_k \in  {\cal{U}} \,,\end{cases}
\eeaa
and where $(|K|-1)_+=|K|-1$ if $|K|\geq 1$ and $(|K|-1)_+=0$ if $|K|=0$\,.

 \begin{lemma}\label{estimatesonthetermsthatcontainbigHandderivativesofBigHintheenergyestimateinawaythatwecouldconcludeforneq3}
We have for $\ga \geq 3 \de $\,, and $0 < \de \leq \frac{1}{4}$\,, that in the exterior region, the following estimates hold for $M \leq \eps$\,, and for all $|I| \leq K$\,, 
 \bea\label{firstterminthebulkoftheenergyestimate}
         \notag
   && | H_{LL } | \cdot | \derm  \Lie_{Z^I} \Phi^{(l)}_{V^l_1 \ldots  V^l_{k_l}}  |^2 + |  H | \cdot | \rderm  \Lie_{Z^I} \Phi^{(l)}_{V^l_1 \ldots  V^l_{k_l}}| \cdot | \derm \Lie_{Z^I} \Phi^{(l)}_{V^l_1 \ldots  V^l_{k_l}}  |  \\
           \notag
&\les&   E ( 3)  \cdot \frac{ \eps }{ (1+t+|q|) } \cdot | \derm  \Lie_{Z^I} \Phi^{(l)}_{V^l_1 \ldots  V^l_{k_l}}|^2    + E ( 3) \cdot \frac{\eps  }{(1+t+|q|)^{1-   \de }}   \cdot | \rderm  \Lie_{Z^I} \Phi^{(l)}_{V^l_1 \ldots  V^l_{k_l}} |^2 \; , \\
 \eea
and
        \bea\label{secondterminthebulkoftheenergyestimate}
        \notag
   && ( | \derm H_{LL} |  + |\rderm H| ) \cdot | \derm  \Lie_{Z^I} \Phi^{(l)}_{V^l_1 \ldots  V^l_{k_l}}  |^2 +  | \derm H | \cdot  | \rderm \Phi_{V} | \cdot  | \derm  \Lie_{Z^I} \Phi^{(l)}_{V^l_1 \ldots  V^l_{k_l}}  | \\
           \notag
&\les&   E ( 3)  \cdot \frac{ \eps }{ (1+t+|q|) \cdot (1+|q|)  } \cdot | \derm  \Lie_{Z^I} \Phi^{(l)}_{V^l_1 \ldots  V^l_{k_l}}  |^2    +  E ( 3) \cdot \frac{\eps  }{(1+t+|q|)^{1-   \de }  \cdot (1+|q|)  }   \cdot | \rderm  \Lie_{Z^I} \Phi^{(l)}_{V^l_1 \ldots  V^l_{k_l}}  |^2  \; .  \\
 \eea
 
 \end{lemma}
 
 \begin{proof}
 
To prove \eqref{firstterminthebulkoftheenergyestimate}, we use the \`a priori decay estimates implied by the bootstrap assumption and given in Lemma \ref{aprioridecayestimates}, and the estimates coming from the condition on the metric in our fixed system of coordinates in Lemma \ref{estimategoodcomponentspotentialandmetric}, to obtain

     \beaa
\notag
&& | H_{LL } | \cdot | \derm \Lie_{Z^I} \Phi^{(l)}_{V^l_1 \ldots  V^l_{k_l}} |^2 + |  H | \cdot | \rderm \Phi_{V} | \cdot | \derm \Lie_{Z^I} \Phi^{(l)}_{V^l_1 \ldots  V^l_{k_l}}  |  \\
&\les&   E ( 3)  \cdot \frac{ \eps \cdot | \derm\Lie_{Z^I} \Phi^{(l)}_{V^l_1 \ldots  V^l_{k_l}}  |^2  }{ (1+t+|q|) }  +  E (2)  \cdot \frac{\eps }{(1+t+|q|)^{1-  \de } }     \cdot | \rderm \Lie_{Z^I} \Phi^{(l)}_{V^l_1 \ldots  V^l_{k_l}}  | \cdot | \derm \Lie_{Z^I} \Phi^{(l)}_{V^l_1 \ldots  V^l_{k_l}}   |  \;.
 \eeaa
 
  We estimate the second term in the right hand side of the inequality as follows,
  \beaa
  &&   E ( 2)  \cdot \frac{\eps }{(1+t+|q|)^{1-   \de } }     \cdot | \rderm \Lie_{Z^I} \Phi^{(l)}_{V^l_1 \ldots  V^l_{k_l}}  | \cdot | \derm \Lie_{Z^I} \Phi^{(l)}_{V^l_1 \ldots  V^l_{k_l}}   |  \\
  &=&E ( 2)  \cdot   \frac{\sqrt{\eps} \cdot \sqrt{(1+t+|q|)} }{(1+t+|q|)^{1-   \de } }   \cdot | \rderm \Lie_{Z^I} \Phi^{(l)}_{V^l_1 \ldots  V^l_{k_l}}  |    \cdot  \frac{\sqrt{\eps} }{\sqrt{(1+t+|q|)} }  \cdot | \derm \Lie_{Z^I} \Phi^{(l)}_{V^l_1 \ldots  V^l_{k_l}}  |   \\
    &\les&  E ( 3)  \cdot \frac{\eps  }{(1+t+|q|)^{1-   \de } }   \cdot | \rderm \Lie_{Z^I} \Phi^{(l)}_{V^l_1 \ldots  V^l_{k_l}}  |^2 + E ( 3)  \cdot \frac{\eps  }{(1+t +|q|) }  \cdot | \derm \Lie_{Z^I} \Phi^{(l)}_{V^l_1 \ldots  V^l_{k_l}}  |^2 \,.
 \eeaa
 
We now prove \eqref{secondterminthebulkoftheenergyestimate}. Using Lemma \ref{aprioridecayestimates} and Lemma \ref{estimategoodcomponentspotentialandmetric}, we have
      \beaa
   && ( | \derm H_{LL} |  + |\rderm H| ) \cdot | \derm \Lie_{Z^I} \Phi^{(l)}_{V^l_1 \ldots  V^l_{k_l}} |^2 +  | \derm H | \cdot  | \rderm\Lie_{Z^I} \Phi^{(l)}_{V^l_1 \ldots  V^l_{k_l}}  | \cdot  | \derm \Lie_{Z^I} \Phi^{(l)}_{V^l_1 \ldots  V^l_{k_l}} | \\
   &\les& E (  3) \cdot \frac{\eps }{(1+t+|q|) \cdot (1+|q|) }  \cdot | \derm \Lie_{Z^I} \Phi^{(l)}_{V^l_1 \ldots  V^l_{k_l}}  |^2  \\
   && +   E ( 3) \cdot \frac{\eps }{(1+t+|q|)^{1-\de} \cdot (1+|q|) } \cdot  | \rderm \Lie_{Z^I} \Phi^{(l)}_{V^l_1 \ldots  V^l_{k_l}}  | \cdot  | \derm\Lie_{Z^I} \Phi^{(l)}_{V^l_1 \ldots  V^l_{k_l}}  |  \;.
   \eeaa
Following exactly the same argument as we did for \eqref{firstterminthebulkoftheenergyestimate} to estimate the second term on the right hand side of the inequality, we obtain the stated result.
\end{proof}

\subsection{The commutator term}\label{estimatethecommutatortermusingbootstraponallderivatives}\

Now, if we would like, we could start by estimating the products in \eqref{Theenergyformofthedecoupledcommutatorestimate} using the fact that when one differentiates a product, one term does not take more than half of the derivatives. This would be crucially important to get a decoupled energy estimate that does not use more derivatives than what it allows to upgrade for. However, at this stage, this is not yet important since we are going anyway to use more derivatives than what we upgrade for, when we estimate the Lie derivatives of the sources of the wave equations. We have no choice other than to estimate the sources of the wave equations by using the bootstrap assumption on all the derivatives involved, in order to derive \`a priori estimates. Thus, at this stage, we can for now use the bootstrap assumption on all the derivatives to estimate the commutator term. Later, we will revisit the commutator term in Subsection \ref{commutatortermrevistedwithbootstraponlyonhaldderivatives} and estimate it by using a bootstrap assumption on only half of the derivatives, although we could have done this directly. We hope that this could also provide clarity about the argument.

\begin{lemma}\label{estimateonthegooddecayingpartofthecommutatorterm}

For $ 0< \delta \leq   \frac{1}{4} $\,, and for $\ga^\prime < 1+ \ga-\de$\,, we have for all $l \in \{1, \ldots, N\}$\,,
  \beaa
       \int_0^t   \int_{\Sigma^{ext}_{\tau} }    \frac {(1+\tau) \cdot  w_{\gamma^\prime} (q)}{\eps \cdot (1+\tau+|q|)^2} \cdot    \sum_{|K|\leq |J|} \Big( \sum_{|J^{\prime}|+(|K|-1)_+\le |J|} \,\,\,
|\Lie_{Z^{J^{\prime}}} H|^2 \cdot {|\derm ( \Lie_{Z^{K} } \Phi^{(l)}) |^2}  \Big)  r^2 dr d\tau &\les& E ( |J| + 2)  \cdot \eps^2 \; ,
\notag
\eeaa
and
  \beaa
      \int_0^t   \int_{\Sigma^{ext}_{\tau} }   \frac {(1+\tau) \cdot  w_{\gamma^\prime} (q)}{(1+\tau+|q|)^2} \cdot   \sum_{|K|\leq |J|} \Big( \sum_{|J^{\prime}|+(|K|-1)_+\le |J|} \,\,\,
|\Lie_{Z^{J^{\prime}}} H|^2 \cdot {|\derm ( \Lie_{Z^{K} } (\Phi^{(0)} - h^0) |^2}  \Big)  r^2 dr d\tau &\les&   E ( |J| + 2)  \cdot \eps^2 \; .
\notag
\eeaa

\end{lemma}
\begin{proof}
From Lemma \ref{aprioridecayestimates}, we get that for all $|J^{\prime}|, |K|\leq |J|$,
    \beaa
   |\Lie_{Z^{J^{\prime}}} H|\cdot {|\derm ( \Lie_{Z^{K} } \Phi^{(l)}) |}   &\leq&   E ( |J| + 2)  \cdot \frac{\eps^2 }{(1+t+|q|)^{2-2\delta} (1+|q|)^{1+\ga + \de}} \;.
       \eeaa
Hence, in the exterior, we have
           \beaa
  && w_{\gamma^\prime} (q)  \cdot  |\Lie_{Z^{J^{\prime}}} H|^2\cdot {|\derm ( \Lie_{Z^{K} } \Phi^{(l)}) |^2} \\
    &\leq& E ( |J| + 2)  \cdot \frac{\eps^4 \cdot  (1+|q|)^{1+2\gamma^\prime} } {(1+t+|q|)^{4-4\delta} (1+|q|)^{2+2\ga+2\de}} \leq E ( |J| + 2)  \cdot   \frac{\eps^4 }{(1+t+|q|)^{4-4\delta} (1+|q|)^{1 +2\de + 2 (\ga-\ga^\prime)} } \; .
       \eeaa
    Thus,
       \beaa
    \frac {(1+ t)\cdot  w_{\gamma^\prime}(q)   }{\eps \cdot (1+t+|q|)^2}   \cdot  |\Lie_{Z^{J^{\prime}}} H|^2 \cdot {|\derm ( \Lie_{Z^{K} } A) |^2}   &\leq&  E ( |J| + 2) \cdot   \frac{\eps^3 }{(1+t+|q|)^{5-4\delta} (1+|q|)^{1 + 2\de+ 2 (\ga-\ga^\prime)} }   \; .
       \eeaa

Since
\bea\label{estimateonradiusr}
r^2 \leq (1+t+|q|)^2 \;,
\eea
     we get
      \beaa
      \notag
    &&   \int_0^t    \int_{\Sigma^{ext}_{\tau} }   \frac {(1+ t)\cdot  w_{\gamma^\prime}(q)  }{\eps \cdot (1+t+|q|)^2}   \cdot  |\Lie_{Z^{J^{\prime}}} H|^2 \cdot {|\derm ( \Lie_{Z^{K} } \Phi^{(l)}) |^2}  \Big)  r^2 dr d\tau \\
&\les&   \int_0^t     \int_{\Sigma^{ext}_{\tau} } E ( |J| + 2) \cdot    \frac{\eps^3 }{(1+t+|q|)^{3-4\delta} (1+|q|)^{1 +2\de+ 2 (\ga-\ga^\prime)  } }  dr  \\
&\les&  E ( |J| + 2)   \int_0^t     \frac{\eps^3 }{(1+t)^{1+\eps^\prime} } \cdot \int_{\Sigma^{ext}_{\tau} }       \frac{\eps^3 }{(1+t+|q|)^{2-4\delta - \eps^\prime} (1+|q|)^{1 +2\de+ 2 (\ga-\ga^\prime)  } }  dr \\
&\les&  E ( |J| + 2)   \int_0^t     \frac{\eps^3 }{(1+t)^{1+\eps^\prime} } \cdot  \int_{\Sigma^{ext}_{\tau} }     \frac{\eps^3 }{ (1+|q|)^{3 -2\de - \eps^\prime + 2 (\ga-\ga^\prime)  } }  dr \, .
\notag
\eeaa
For $\ga^\prime < 1+ \ga-\de$\,, we have $3 -2\de - \eps^\prime + 2 (\ga-\ga^\prime) > 1$ for $\eps^\prime$ small enough. Hence, 
      \beaa
      \notag
    &&   \int_0^t    \int_{\Sigma^{ext}_{\tau} }   \frac {(1+ t)\cdot  w_{\gamma^\prime}(q)  }{\eps \cdot (1+t+|q|)^2}   \cdot  |\Lie_{Z^{J^{\prime}}} H|^2 \cdot {|\derm ( \Lie_{Z^{K} } \Phi^{(l)}) |^2}  \Big)  r^2 dr d\tau \\
    & \les&  E ( |J| + 2) \cdot   \eps^3 \, .
\eeaa
Similarly, we get the result for $\Phi^{(0)} - h^{0}$\,.
     \end{proof}
\begin{remark}
We notice that if we replace in Lemma \ref{estimateonthegooddecayingpartofthecommutatorterm}, the factor in the left hand side of the inequalities of $\frac{1}{(1+\tau+|q|)^2}$ by $\frac{1}{(1+|q|)^2}$\,, then the integrals would not be integrable, which motivates the following Lemma \ref{estimateonthebadpartofthecommutatortermtobtainaGronwallforcomponents}.
\end{remark}

 \begin{lemma}\label{estimateonthebadpartofthecommutatortermtobtainaGronwallforcomponents}

For $M \leq \eps$ and $ 0< \delta \leq  \frac{1}{4} $\,, and for any $\ga^\prime$\,, we have for all $l \in \{1, \ldots, N\}$\,,
\beaa
&&  \int_0^t  \int_{\Sigma^{ext}_{\tau} }   \frac {(1+\tau) \cdot w_{\gamma^\prime}(q) }{\eps \cdot (1+|q|)^2} \cdot   \sum_{|K|\leq |J|} \Big( \sum_{|J^{\prime}|+(|K|-1)_+\leq |J|} \!\!\!\!\!| \Lie_{Z^{J^{\prime}} }H_{LL}|^2   \cdot \sum_{V^{l}_1 \in {\cal S}^{l}_1 \,, \ldots \,, V^{l}_{k_{l}} \in  {\cal S}^{l}_{k_{l}}  } |\derm  \Lie_{Z^{K}} \Phi^{(l)}_{V^l_1 \ldots  V^l_{k_l}} |^2 \Big) r^2 dr d\tau \\
&\les &      \int_0^t   \int_{\Sigma^{ext}_{\tau} }   E (  |J|  +3)    \cdot \frac{\eps }{(1+t+|q|) }     \cdot   \sum_{|K| \leq |J|} \;\; \sum_{V^{l}_1 \in {\cal S}^{l}_1 \,, \ldots \,, V^{l}_{k_{l}} \in  {\cal S}^{l}_{k_{l}}  }    |\derm \Phi^{(l)}_{V^l_1 \ldots  V^l_{k_l}} |^2 \cdot w_{\gamma^\prime}(q) r^2 dr d\tau  \; ,
 \eeaa
 and
\beaa
&&  \int_0^t   \int_{\Sigma^{ext}_{\tau} }   \frac {(1+\tau) \cdot w_{\gamma^\prime}(q) }{\eps \cdot (1+|q|)^2} \cdot   \sum_{|K|\leq |J|} \Big( \sum_{|J^{\prime}|+(|K|-1)_+\leq |J|} \!\!\!\!\!| \Lie_{Z^{J^{\prime}} }H_{LL}|^2   \cdot\sum_{V^{0}_1 \in {\cal S}^{0}_1 \,, V^{0}_{2} \in  {\cal S}^{0}_{2}  }  |\derm  \Lie_{Z^{K}} ( \Phi^{(0)} - h^0)_{V^0_1 V^0_{2}}   |^2   \Big) r^2 dr d\tau \\
&\les &      \int_0^t  \int_{\Sigma^{ext}_{\tau} }    E (  |J|  +3)    \cdot \frac{\eps }{(1+t+|q|) }     \cdot   \sum_{|K| \leq |J|} \;\; \sum_{V^{0}_1 \in {\cal S}^{0}_1 \,, V^{0}_{2} \in  {\cal S}^{0}_{2}  }  |\derm  ( \Phi^{(0)} - h^0)_{V^0_1 V^0_{2}}   |^2  \cdot w_{\gamma^\prime}(q)r^2 dr d\tau \; .
 \eeaa

\end{lemma}

\begin{proof}
For $l \in \{1, \ldots, N\}$\,, we decompose the following sum as  
\beaa
&& \frac{1}{\eps \cdot (1+|q|)^2} \cdot \sum_{|K|\leq |J|}\Big(\sum_{|J^{\prime}|+(|K|-1)_+\leq |J|} \!\!\!\!\!|  \Lie_{Z^{J^{\prime}} }H_{LL}|^2   \cdot  \sum_{V^{l}_1 \in {\cal S}^{l}_1 \,, \ldots \,, V^{l}_{k_{l}} \in  {\cal S}^{l}_{k_{l}}  } |\derm \Phi^{(l)}_{V^l_1 \ldots  V^l_{k_l}} |^2  \Big)    \\
&\les & \sum_{|I| \leq |J|} \!\!\!\!\ |  \Lie_{Z^{I} }H_{LL}|^2  \cdot  \sum_{|K| \leq  |J| } \;\;  \sum_{V^{l}_1 \in {\cal S}^{l}_1 \,, \ldots \,, V^{l}_{k_{l}} \in  {\cal S}^{l}_{k_{l}}  } |\derm \Lie_{Z^{K}} \Phi^{(l)}_{V^l_1 \ldots  V^l_{k_l}}  |^2    \;.
\eeaa
Based on Lemma \ref{estimategoodcomponentspotentialandmetric}, for $M\leq \eps$\,, and for $\de \leq \frac{1}{4}$\,, we have
         \beaa
       \frac{1}{(1+|q|)^2} \sum_{|I|\leq \, |J|}| \Lie_{Z^I} H_{LL}|^2    &\les&   E (  |J|  +3)  \cdot \frac{\eps^2 }{(1+t+|q|)^{2 } }     \; .
      \eeaa
     Consequently,   
\beaa
\notag
&& \frac {(1+t) \cdot w_{\gamma^\prime}(q) }{\eps \cdot (1+|q|)^2} \cdot  \sum_{|K| \leq  |J|}\Big(\sum_{|I| \leq |J|} \!\!\!\!\ | \Lie_{Z^{I} }H_{LL}|^2    \cdot \sum_{V^{l}_1 \in {\cal S}^{l}_1 \,, \ldots \,, V^{l}_{k_{l}} \in  {\cal S}^{l}_{k_{l}}  } |\derm \Phi^{(l)}_{V^l_1 \ldots  V^l_{k_l}} |^2  \Big) \\
\notag
 &\leq&   E (  |J|  +3)  \cdot  \frac{\eps }{(1+t+|q|)}     \cdot \sum_{|K| \leq  |J|}  \; \; \sum_{V^{l}_1 \in {\cal S}^{l}_1 \,, \ldots \,, V^{l}_{k_{l}} \in  {\cal S}^{l}_{k_{l}}  }  |\derm \Phi^{(l)}_{V^l_1 \ldots  V^l_{k_l}}  |^2 \cdot w_{\gamma^\prime}(q)   \; . 
      \eeaa

Similarly, we get the estimate for $\Phi^{(0)} - h^0$\, .

\end{proof}

\subsection{The space-time integral of tangential derivatives}\label{The space-time integral of tangential derivatives}\

We notice that when we use Lemma \ref{estimatesonthetermsthatcontainbigHandderivativesofBigHintheenergyestimateinawaythatwecouldconcludeforneq3} for \eqref{decoupledenergyestimatewithoutcommutatorestimate}, we get derivatives with a seemingly bad factor that prevents us from obtaining the correct factor of the type $\frac{1}{(1+t)}$ in a desired inequality of the form of \eqref{Theinequalityontheenergytobeusedtoapplygronwallrecursively}. However, these seemingly wrong decay rates are not factors for any derivative, but for tangential derivatives for which we have a control on the space-time integral in our decoupled higher order energy estimates. Hence, we will show that for $\de$ small enough depending on $\mu$\,, we can absorb these integrals to the left hand side of our energy estimate \eqref{decoupledenergyestimatewithoutcommutatorestimate} and therefore, they will not appear in a desired inequality of the form of \eqref{Theinequalityontheenergytobeusedtoapplygronwallrecursively}.

We notice that the weight that enters infront of the space-time integral of the tangential derivatives is actually $\widehat{w}_{\gamma^{\prime}}^{\prime}(q)$\,, which in view of definition of the weight, \eqref{derivativeoftheweightintermsoftheweight}, is actually $\frac{\widehat{w} (q)}{(1+|q|)} $\,.

Whereas to the weight that enters infront of \eqref{secondterminthebulkoftheenergyestimate} in the energy estimate \eqref{decoupledenergyestimatewithoutcommutatorestimate}, it is actually $w_{\gamma^{\prime}} (q) $\,.  Thus, for the term \eqref{secondterminthebulkoftheenergyestimate}, we need to absorb to the left hand side a space-time integral prescribed by 
\beaa
\int_{0}^{t}  \int_{\Sigma^{ext}_{\tau} }   E ( 3 ) \cdot \frac{\eps  }{(1+t+|q|)^{1-  \de } \cdot (1+|q|) }   \cdot | \rderm \Phi_{V} |^2   \cdot  w_{\gamma^{\prime}}(q) \cdot d^{n}x\ \cdot d\tau \; . 
\eeaa

 On one hand, for $q < 0$\,, we know that $\frac{\widehat{w}_{\gamma^{\prime}}(q)}{(1+|q|)} = \frac{1}{(1+|q|)^{1-2\mu}}$\,, $\mu < 0$\,, and  $\frac{\widehat{w}_{\gamma^{\prime}} (q)}{(1+|q|)} = \frac{1}{(1+|q|)}$\,. Thus, we look forward to absorb space-time integral on the right hand side 
 \bea\label{welookforwardtoabsorbspace-timeintegral}
\int_{0}^{t}  \int_{\Sigma^{ext}_{\tau} }   E ( 3 )\cdot  \frac{\eps  }{(1+|q|)^{2-  \de } }   \cdot    |\rderm \Phi_{V} |^2 \cdot d^{n}x\ \cdot d\tau 
 \eea
 into a space-time integral on the left hand side as follows, with $\mu < 0$\,,
 \bea\label{space-time integralonthelefthandside}
\int_{0}^{t}  \int_{\Sigma^{ext}_{\tau} }   \frac{\widehat{w}_{\gamma^{\prime}} (q)}{(1+|q|)}  \cdot    |\rderm \Phi_{V} |^2 &=& \frac{1 }{(1+|q|)^{1-  2\mu }  } \cdot    |\rderm \Phi_{V} |^2 \cdot d^{n}x\ \cdot d\tau  \; .
\eea
   We can choose $\mu$ small, and choose $\de$ small enough depending on $\mu$ such that for $\eps$ small enough, the space-time integral \eqref{welookforwardtoabsorbspace-timeintegral} can be absorbed into the space-time integral on the left hand side \eqref{space-time integralonthelefthandside} in our energy estimate \eqref{decoupledenergyestimatewithoutcommutatorestimate}.
  
  On the other hand, for $q > 0$\,, we have $w_{\gamma^{\prime}} (q) = \widehat{w}_{\gamma^{\prime}} (q) $ and therefore it can be absorbed to the left hand side for $\eps $ small and say for $\de \leq 1$\,.

Similarly, we deal with the term \eqref{firstterminthebulkoftheenergyestimate}, where now the weight is the same as the one that controls the tangential derivatives in the left hand side of  \eqref{decoupledenergyestimatewithoutcommutatorestimate} and thanks to the factor $\eps$ coming from our bootstrap assumption, it can be absorbed to the left hand side in \eqref{decoupledenergyestimatewithoutcommutatorestimate} for $\eps $ small enough.

Thus, as discussed above, considering the behaviour of the weights in \eqref{decoupledenergyestimatewithoutcommutatorestimate}, the only terms that are left from \eqref{firstterminthebulkoftheenergyestimate} and \eqref{secondterminthebulkoftheenergyestimate} are 
\beaa
 \int_{0}^{t}  \int_{\Sigma^{ext}_{\tau} }  E (  3)  \cdot \frac{ \eps }{ (1+\tau+|q|)\cdot (1+|q|)  } \cdot | \derm \Phi_{V} |^2     \cdot  w_{\gamma^{\prime}} (q) d^{3}x  \cdot d\tau \, ,
 \eeaa
 which have the good desired factor of $\frac{ 1 }{ (1+\tau)}$\,.

      \subsection{Conclusion}\
      
We revisit our decoupled energy estimate \eqref{decoupledenergyestimatewithoutcommutatorestimate}, taking into account how we handled the sources terms as in \eqref{decoupledestimeontgewaveoperatoroftheLierderivativestimesthederivative}, and taking into account Lemmas \ref{estimatesonthetermsthatcontainbigHandderivativesofBigHintheenergyestimateinawaythatwecouldconcludeforneq3}, \ref{estimateonthegooddecayingpartofthecommutatorterm} and \ref{estimateonthebadpartofthecommutatortermtobtainaGronwallforcomponents}, and considering the discussion in Subsection \ref{The space-time integral of tangential derivatives}, we get Proposition \ref{Theveryfinal }.

   \section{Improved \`a priori bounds on the energy}\

In this section, we can start by using our decoupled energy estimates from Proposition \ref{Theveryfinal }, combined with our decoupled commutator estimate for the tangential components \eqref{Theseperatecommutatortermestimateforthetangentialcomponents}, to upgrade the bootstrap assumption \eqref{aprioriestimate} from a $\de$ power to an $c. \eps$ power, however while requiring the \`a priori bootstrap assumption on more derivatives than what we upgrade for, thus it is \textit{not yet} an actual upgrade. The way one can achieve this, is by induction on $k$ to prove that
\bea
\E_{ k } (t)  \leq   E(k + 5) \cdot \eps \cdot (1 +t)^{c\cdot \eps} \;,
\eea
by requiring the bootstrap assumption \eqref{aprioriestimate} to hold true for $k+5$ derivatives, thus not closing the bootstrap argument yet. This, in a later stage, will be translated into pointwise decay on the fields and applied only to about half of the derivatives, so as to close the bootstrap argument.

    \begin{definition}\label{Definitionofenergywithallthefieldsforgoodcomponentsandenergyforfullcomponents}
We recall Definition \ref{Definitionofgoodenergyandbadenergywithdifferentweightswiththefieldsspeaaretly} and we then define the decoupled higher order energy norm for full system, including the fields, as
\beaa
\notag
\E^{\textit{good}\,,\, \gamma^{\prime}}_{k} (t) &:=&   \sum_{l \in \{1, \ldots, N \}}   \E^{\textit{good}\,,\, \gamma^{\prime}}_{k} (\Phi^{(l)})  (t)    +  \E^{\textit{good}\,,\, \gamma^{\prime}}_{k} (\Phi^{(0)} - h^0)  (t)    \; , \\
\notag
 \E^{full\,,\, \gamma^{\prime}}_{k} (t) &:=&  \sum_{l \in \{1, \ldots, N \}}   \E^{full\,,\, \gamma^{\prime}}_{k} (\Phi^{(l)})  (t)  +   \E^{full\,,\, \gamma^{\prime}}_{k} (\Phi^{(0)} - h^0) \, .
\eeaa
\end{definition}

The goal of this section is to prove the following proposition.

\begin{proposition}\label{Improvesaprioriboundsonenergywithextraweightforgoodcomponentswithboundsongamma}
 Let $M\leq \eps$\,, $\ga \geq 5 \de $\,, and $0<  \de  \leq \frac{1}{4}$\,. Then, for $\ga^\prime < 1+ 4\de $\,, we have 
\bea\label{notupgradedyetbutimprovedboundonenergyofgoodcomponentwithgamprime}
\E^{\textit{good}\,,\, \gamma^{\prime}}_{k} (t)  \les E(k +3) \cdot \eps \cdot (1 +t)^{c\cdot \eps} \,,
\eea
 for $\E^{\textit{good}\,,\, \gamma^{\prime}}_{k} (0) < \eps^2 $\,, that is an estimate of the type \eqref{non-upgradedimprovedgrowthonenergy} as outlined in our strategy of the proof (thus, it is not yet an upgrade since there is a loss on the derivatives). And if we additionally use \eqref{notupgradedyetbutimprovedboundonenergyofgoodcomponentwithgamprime} with a $ \ga^\prime  >  \frac{1}{2} + \ga $\,, then we have 
\bea\label{notupgradedyetbutimprovedboundonenergyoffullcomponent}
\E^{\textit{full}\,,\, \gamma}_{k} (t)  \les E(k +5) \cdot \eps \cdot (1 +t)^{c\cdot \eps} \,.
\eea
for $\E^{\textit{full}\,,\, \gamma}_{k} (0) < \eps^2 $\,. Thus, this imposes that the initial data must satisfy $\E^{\textit{good}\,,\, \gamma^{\prime}}_{k} (0) < \eps^2 $ for a certain $\ga^\prime$ such that $ \frac{1}{2} + \ga < \ga^\prime < 1+ 4\de$\,, from which we get \eqref{notupgradedyetbutimprovedboundonenergyoffullcomponent} under the bootstrap assumption with the weight $w_\ga$\,. Consequently, we get the improved pointwise \`a priori bounds on the fields in Lemma \ref{upgradedaprioridecayestimateswithmorederivatives}, for any $\ga$ such that  $ 0 < 5 \de \leq \ga < \frac{1}{2} +4\de $\,.

\begin{remark}
Recall that the bootstrap assumption was carried out with the weight $\ga$ only\,. Our goal is ultimately improve \eqref{notupgradedyetbutimprovedboundonenergyofgoodcomponentwithgamprime} where we would have $E(k)$ instead of $E(k +5)$, for $k$ large enough. This would be the goal of the next section.
\end{remark}

\end{proposition}

   \subsection{Improved \`a priori bounds on the energy of good components}\

 \begin{lemma}\label{decayestimateonthesourcetermforsourcesofgoodcomponents}
Let $M\leq \eps$\,, $\ga \geq 3 \de $\,, and $0< \de < \frac{1}{2}$\,. For $\ga^\prime < 1+ 4\de $\,, we have in the exterior region the following estimate on the sources of the good components, for all $l \in \{0, 1, \ldots, N\}$\,,   
        
\bea
\notag
\int_{0}^{t}  \int_{\Sigma^{ext}_{\tau} }  \frac{(1+\tau )}{\eps} \cdot  | \Lie_{Z^I} \Big( g^{\alpha\beta} \derm_\alpha \derm_\beta  \Phi^{(l)}_{{\cal S}^l_1 \ldots {\cal S}^l_{k_l} }   \Big) |^2  r^2  \cdot w_{\ga^\prime}(q) dr d\tau &\les&  E ( |I| +3)  \cdot \eps^3 \, .
\eea

\end{lemma}

\begin{proof}

We start by using our \`a priori decay estimates from Lemma \ref{aprioridecayestimates} where we used our bootstrap assumption \eqref{aprioriestimate} with the weight $w_{\gamma}$\,. Using the structure of the good wave equation \eqref{goodtensorialcoupledwaveequation}, we get depending on whether we have $\Phi^{(0)}$ in the sources, or $\Phi^{(l)}$\,, $l \in \{1, 2, \ldots, N \}$\,, that in the exterior region 
                                        \beaa
 |  \Lie_{Z^I} \Big( g^{\la\mu} \derm_{\la}   \derm_{\mu}    \Phi^{(l)}_{{\cal S}^l_1 \ldots {\cal S}^l_{k_l} }  \Big)  |    \leq \begin{cases} &E ( |I| + 3)   \cdot \frac{\eps^2 }{(1+t+|q|)^{3-3\delta} (1+|q|)^{1+\ga}}  \; , \\
 &  E ( |I| + 3)   \cdot \frac{\eps^2 }{(1+t+|q|)^{3-3\delta} (1+|q|)^{1+2\de}} \, . \end{cases} 
       \eeaa
    Hence, for all $l \in \{0, 1, \ldots, N \}$\, we have in the exterior region, 
                     \beaa
 |  \Lie_{Z^I} \Big( g^{\la\mu} \derm_{\la}   \derm_{\mu}    \Phi^{(l)}_{{\cal S}^l_1 \ldots {\cal S}^l_{k_l} }  \Big)  |  & \leq&E ( |I| + 3)   \cdot \frac{\eps^2 }{(1+t+|q|)^{3-3\delta} (1+|q|)^{1+\min\{2\de, \ga\}}} \; ,
    \eeaa
and consequently,         
          \beaa
\notag
\frac{(1+t )}{\eps}  \cdot  | \Lie_{Z^I} \Big( g^{\alpha\beta} \derm_\alpha \derm_\beta   \Phi^{(l)}_{{\cal S}^l_1 \ldots {\cal S}^l_{k_l} }  \Big) |^2  &\les&  E ( |I| +3)  \cdot \frac{\eps^3 }{(1+t + |q| )^{5-6\delta} (1+|q|)^{2+2 \min\{2\de, \ga\} } } \,. 
        \eeaa

        Since $r^2 \les (1+t+|q|)^{2}$\,, we get
          \beaa
\notag
\frac{(1+t )}{\eps}  \cdot  | \Lie_{Z^I} \Big( g^{\alpha\beta} \derm_\alpha \derm_\beta \Phi^{(l)}_{{\cal S}^l_1 \ldots {\cal S}^l_{k_l} }   \Big) |^2  r^2 &\les&  E ( |I| +3)  \cdot \frac{\eps^3 }{(1+t + |q|)^{3-6\delta} (1+|q|)^{2+2 \min\{2\de, \ga\} } } \,.
        \eeaa
Therefore, 
         \beaa
\notag
&& \frac{(1+t )}{\eps} \cdot  | \Lie_{Z^I} \Big( g^{\alpha\beta} \derm_\alpha \derm_\beta  \Phi^{(l)}_{{\cal S}^l_1 \ldots {\cal S}^l_{k_l} }   \Big) |^2  \cdot r^2  \cdot w_{\ga^\prime}(q) \\
&\les&  E ( |I| +3)  \cdot \frac{\eps^3 \cdot (1+|q|)^{1+2\ga^\prime} }{(1+t + |q| )^{3-6\delta} (1+|q|)^{2+2 \min\{2\de, \ga\} } } \\
      &\les&  E ( |I| +3)  \cdot \frac{\eps^3 }{(1+t + |q| )^{3-6\delta} (1+|q|)^{1-2\ga^\prime+ 2\min\{2\de, \ga\} } } \\
&\les&  E ( |I| +3)  \cdot \frac{\eps^3 }{(1+t + |q| )^{1 
+ \eps^\prime } (1+|q|)^{3-6\de-\eps^\prime - 2\ga^\prime+2 \min\{2\de, \ga\} } } 
        \eeaa

          For $\ga^\prime < 1 +2  \min\{2\de, \ga\}  $\,, then $3 - 2\ga^\prime+2 \min\{2\de, \ga\}  > 1 $ and by choosing $\eps^\prime$ and $\de$ small enough, we have $3-6\de-\eps^\prime - 2\ga^\prime+2 \min\{2\de, \ga\}  > 1$ and therefore integrable in $|q|$\,. For $\ga \ge 3\de$\,, we have $\min\{2\de, \ga\} = 2\de$\,, and hence the condition on $\ga$ translates as the condition $\ga^\prime < 1 + 4\de$\,.

\end{proof}

 \begin{corollary}\label{Improvedapriorboundontheenergyofgoodcomponentswithaweightgaprime}
 Let $M\leq \eps$\,, $\ga \geq 5 \de $\,, and $0< \de \leq \frac{1}{4}$\,. For $\ga^\prime < 1+ 4\de $\,, we have 
 \bea
\E^{\textit{good}\,,\, \gamma^{\prime}}_{k} (t)  \les E(k +3) \cdot \eps \cdot (1 +t)^{c\cdot \eps} \,,
\eea
 for initial data such that $\E^{\textit{good}\,,\, \gamma^{\prime}}_{k} (0) < \eps$\,, and under the bootstrap assumption \eqref{aprioriestimate}-\eqref{theboundinthetheoremonEnbyconstantEN} .
 
 \end{corollary}

\begin{proof}

We can insert the estimate on the sources of the good components from Lemma \ref{decayestimateonthesourcetermforsourcesofgoodcomponents}, into the estimates on the good components of fields in Proposition \ref{Theveryfinal }, namely \eqref{decoupledenergyestimateongoodcomponentsforfieldslgretarto1} and \eqref{decoupledenergyestimateongoodcomponentsforthemetricphiominusho}. By induction on $k$\,, we get the desired result.

However, we notice that Lemma \ref{decayestimateonthesourcetermforsourcesofgoodcomponents} imposes that $\ga^\prime < 1+ 4\de$\,, with $\ga \geq 3\de$. Our decoupled energy estimates from Proposition \ref{Theveryfinal } impose that $\ga^\prime < 1+ \ga-\de$\,, with  $\ga \geq 3 \de $. These two conditions are compatible with each other. In fact, by choosing $\ga \geq 5\de$\,, we can sum up these two conditions simply as $\ga^\prime < 1+ 4\de$\,. We also take $\de \leq \frac{1}{4}$\,.

\end{proof}

 Thus, we were able to upgrade the \`a priori bound on the energy of the good components, namely $ \E^{\textit{good}\,,\, \gamma^{\prime}}_{k} (t)$, for a weight $\gamma^{\prime}$ verifying a certain condition. In its turn we would like to use this ``upgrade" (that is not an actual upgrade yet since we are using more derivatives than what we upgrade for) to establish a similar \`a priori ``upgrade" for the full components.

\subsection{Improved \`a priori bounds on the energy of full components}\

We start by estimating the “bad” terms $  \Phi^{(l_i)}_{{\cal S}^{l_i}_1 \ldots {\cal S}^{l_i}_{k_{l_i}} }    \cdot \Phi^{(l_j)}_{{\cal S}^{l_j}_1 \ldots {\cal S}^{l_j}_{k_{l_j}} }  $\,.

\begin{lemma}\label{aprioriestimtaeonthebulkfromphiequared}
 Let $M\leq \eps$\,, $\ga \geq 5 \de $\,, and $0< \de < \frac{1}{2}$\,. Then, for $\ga^{\prime} > \frac{1}{2} + \ga$\, and under the conditions of Corollary \ref{Improvedapriorboundontheenergyofgoodcomponentswithaweightgaprime}, namely $\ga^{\prime} < 1 + 4\de $\,, we have in the exterior region 
 
   \beaa
\notag
 \int_{0}^{t}   \int_{\Sigma^{ext}_{\tau} } \frac{(1+t) }{\eps  } \cdot  |  \Lie_{Z^I} \Phi^{(l_i)}_{{\cal S}^{l_i}_1 \ldots {\cal S}^{l_i}_{k_{l_i}} }  |^2 \cdot  |  \Lie_{Z^I} \Phi^{(l_j)}_{{\cal S}^{l_j}_1 \ldots {\cal S}^{l_j}_{k_{l_j}} }  |^2 \cdot w_{\gamma}(q)  r^2 dr d\tau & \les& E(|I| + 5)  \cdot \eps \cdot  (1+t)^{c\cdot \eps}  \; .
\eeaa
Thus, the estimate holds for $ \frac{1}{2} + \ga < \ga^{\prime} < 1+4\de  $\,.
\end{lemma}

\begin{proof}
Once we have an \`a prior control on $ \E^{\textit{good}\,,\, \gamma^{\prime}}_{k} (t)$ from Corollary \ref{Improvedapriorboundontheenergyofgoodcomponentswithaweightgaprime}, we can then use the weighted Klainerman-Sobolev inequality,
   \bea
\notag
| \derm \Lie_{Z^I} \Phi^{(l_i)}_{{\cal S}^{l_i}_1 \ldots {\cal S}^{l_i}_{k_{l_i}} } | \cdot (1+t+|q|) \cdot \big[ (1+|q|) \cdot w_{\gamma^{\prime}}(q)\big]^{1/2} &\leq&
C \cdot \sum_{|K|\leq |I| + 2 } \|\big(w_{\gamma^{\prime}}(q)\big)^{1/2} \derm \Lie_{Z^K} \Phi^{(l_i)}_{{\cal S}^{l_i}_1 \ldots {\cal S}^{l_i}_{k_{l_i}} } (t,\cdot)\|_{L^2 (\Sigma^{ext}_{t} ) } \\
&\les &  E(k +5) \cdot \sqrt{\eps} \cdot (1 +t)^{c\cdot \eps} \, .
\eea

This would in particular imply a control on $| \derm_{\underline{L}} \Lie_{Z^I}\Phi^{(l_i)}_{{\cal S}^{l_i}_1 \ldots {\cal S}^{l_i}_{k_{l_i}} }  | $ and since for all $U \in {\cal U}$, we have $\derm_{\underline{L}} V = 0$\,, we get in the exterior region,

   \beaa
\notag
| \pa_{\underline{L}} \Lie_{Z^I} \Phi^{(l_i)}_{{\cal S}^{l_i}_1 \ldots {\cal S}^{l_i}_{k_{l_i}} } | \les E( |I| + 5)  \cdot \frac{\sqrt{\eps}}{(1+t+|q|)^{1-c\cdot \eps} \cdot (1+|q|)^{1+\ga^{\prime}}} \; .
\eeaa

By integration along $q$\,, see \cite{G4} for more details, we get 
   \beaa
\notag
|  \Lie_{Z^I} \Phi^{(l_i)}_{{\cal S}^{l_i}_1 \ldots {\cal S}^{l_i}_{k_{l_i}} }  | \les E(|I| + 5)  \cdot \frac{\sqrt{\eps}}{(1+t+|q|)^{1-c\cdot \eps} \cdot (1+|q|)^{\ga^{\prime}}} \; . 
\eeaa
Hence,
   \beaa
\notag
|  \Lie_{Z^I} \Phi^{(l_i)}_{{\cal S}^{l_i}_1 \ldots {\cal S}^{l_i}_{k_{l_i}} }  |^2 \cdot |  \Lie_{Z^I} \Phi^{(l_j)}_{{\cal S}^{l_j}_1 \ldots {\cal S}^{l_j}_{k_{l_j}} }  |^2 \les E(|I| + 5)  \cdot \frac{\eps^2}{(1+t+|q|)^{4-c\cdot \eps} \cdot (1+|q|)^{4\ga^{\prime}}} \; , 
\eeaa
and 
   \beaa
\notag
\frac{(1+t) }{\eps  } \cdot  |  \Lie_{Z^I} \Phi^{(l_i)}_{{\cal S}^{l_i}_1 \ldots {\cal S}^{l_i}_{k_{l_i}} }  |^2 \cdot |  \Lie_{Z^I} \Phi^{(l_j)}_{{\cal S}^{l_j}_1 \ldots {\cal S}^{l_j}_{k_{l_j}} }  |^2 \cdot  w_{\gamma}(q)   \les E(|I| + 5)  \cdot \frac{\eps \cdot w_{\gamma}(q)}{(1+t+|q|)^{3-c\cdot \eps} \cdot (1+|q|)^{4\ga^{\prime}}} \; .
\eeaa

Since $r^2 \les (1+t+|q|)^2$\,, we get
   \beaa
\notag
&&  \int_{\Sigma^{ext}_{\tau} } \frac{(1+t) }{\eps  } \cdot |  \Lie_{Z^I} \Phi^{(l_i)}_{{\cal S}^{l_i}_1 \ldots {\cal S}^{l_i}_{k_{l_i}} }  |^2 \cdot |  \Lie_{Z^I} \Phi^{(l_j)}_{{\cal S}^{l_j}_1 \ldots {\cal S}^{l_j}_{k_{l_j}} }  |^2 w_{\gamma}(q)  r^2 dr  \\
 & \les&   \int_{\Sigma^{ext}_{\tau} } E(|I| + 5)  \cdot \frac{\eps \cdot  (1+ |q|)^{1+2\ga }}{(1+t+|q|)^{1-c\cdot \eps} \cdot (1+|q|)^{4\ga^{\prime}}}  dr \;  \\
 & \les& \frac{1}{ (1+t)^{1-c\cdot \eps} } \cdot \int_{\Sigma^{ext}_{\tau} } E(|I| + 5)  \cdot \frac{\eps }{ (1+|q|)^{-1-2\ga + 4\ga^{\prime}}}  dr \; . 
\eeaa

When $\ga^{\prime} > \frac{1}{2} + \ga$\,, we have $-1-2\ga + 4\ga^{\prime} > 1 $\, and therefore the integration in $r$ is integrable. However, Corollary \ref{Improvedapriorboundontheenergyofgoodcomponentswithaweightgaprime}
that we used, imposed that $\ga^\prime < 1+4\de$\,. Hence, for $ \frac{1}{2} + \ga < \ga^{\prime} < 1+4\de $\,, we have the estimate.

 \end{proof}

  Now, we estimate the other “bad” terms, $  \Phi^{(l_i)}_{{\cal S}^{l_i}_1 \ldots {\cal S}^{l_i}_{k_{l_i}} }    \cdot  \derm \Phi^{(l_j)}_{{\cal S}^{l_j}_1 \ldots {\cal S}^{l_j}_{k_{l_j}} }  $ and $\derm \Phi^{(l_i)}_{{\cal S}^{l_i}_1 \ldots {\cal S}^{l_i}_{k_{l_i}} }    \cdot  \derm \Phi^{(l_j)}_{{\cal S}^{l_j}_1 \ldots {\cal S}^{l_j}_{k_{l_j}} }  \;.$

\begin{lemma}\label{aprioriestimatesontheotherbadtermsphiDphiandDphisquare}
 Let $M\leq \eps$\,, $\ga \geq 5 \de $\,, and $0< \de < \frac{1}{2}$. Then, under the conditions of Corollary \ref{Improvedapriorboundontheenergyofgoodcomponentswithaweightgaprime}, namely $\ga^{\prime} < 1 + 4\de $\,, we have for $\ga^{\prime} > \frac{\ga}{2}$\,, in the exterior region
 
   \beaa
\notag
 \int_{0}^{t}   \int_{\Sigma^{ext}_{\tau} } \frac{(1+t) }{\eps  } \cdot |  \Lie_{Z^I} \Phi^{(l_i)}_{{\cal S}^{l_i}_1 \ldots {\cal S}^{l_i}_{k_{l_i}} }  |^2 \cdot | \derm \Lie_{Z^I} \Phi^{(l_j)}_{{\cal S}^{l_j}_1 \ldots {\cal S}^{l_j}_{k_{l_j}} } |^2 \cdot w_{\gamma}(q)  r^2 dr d\tau & \les& E(|I| + 5)  \cdot \eps \cdot  (1+t)^{c\cdot \eps}  \; , 
\eeaa
and for $\ga^{\prime} > - \frac{1}{2} + \ga $\,, in the exterior region

   \beaa
\notag
 \int_{0}^{t}   \int_{\Sigma^{ext}_{\tau} } \frac{(1+t) }{\eps  } \cdot | \derm \Lie_{Z^I} \Phi^{(l_i)}_{{\cal S}^{l_i}_1 \ldots {\cal S}^{l_i}_{k_{l_i}} }  |^2 \cdot | \derm \Lie_{Z^I} \Phi^{(l_j)}_{{\cal S}^{l_j}_1 \ldots {\cal S}^{l_j}_{k_{l_j}} } |^2 \cdot w_{\gamma}(q)  r^2 dr d\tau & \les& E(|I| + 5)  \cdot \eps \cdot  (1+t)^{c\cdot \eps}  \; . 
\eeaa

\end{lemma}

\begin{proof}

We proceed as in the proof of Lemma \ref{aprioriestimtaeonthebulkfromphiequared}, and we get
   \beaa
\notag
|  \Lie_{Z^I} \Phi^{(l_i)}_{{\cal S}^{l_i}_1 \ldots {\cal S}^{l_i}_{k_{l_i}} }  | \cdot | \derm \Lie_{Z^I} \Phi^{(l_j)}_{{\cal S}^{l_j}_1 \ldots {\cal S}^{l_j}_{k_{l_j}} } | \les E(|I| + 5)  \cdot \frac{\eps}{(1+t+|q|)^{2-c\cdot \eps} \cdot (1+|q|)^{1+2\ga^{\prime}}} \; . 
\eeaa 
Therefore, 

  \beaa
\notag
|  \Lie_{Z^I} \Phi^{(l_i)}_{{\cal S}^{l_i}_1 \ldots {\cal S}^{l_i}_{k_{l_i}} }  |^2 \cdot | \derm \Lie_{Z^I} \Phi^{(l_j)}_{{\cal S}^{l_j}_1 \ldots {\cal S}^{l_j}_{k_{l_j}} } |^2 \les E(|I| + 5)  \cdot \frac{\eps^2}{(1+t+|q|)^{4-c\cdot \eps} \cdot (1+|q|)^{2+4\ga^{\prime}}} \; . 
\eeaa 
Proceeding as before as in the proof of Lemma \ref{aprioriestimtaeonthebulkfromphiequared}, we obtain
   \beaa
\notag
&&  \int_{\Sigma^{ext}_{\tau} } \frac{(1+t) }{\eps  } \cdot |  \Lie_{Z^I} \Phi^{(l_i)}_{{\cal S}^{l_i}_1 \ldots {\cal S}^{l_i}_{k_{l_i}} }  |^2 \cdot | \derm \Lie_{Z^I} \Phi^{(l_j)}_{{\cal S}^{l_j}_1 \ldots {\cal S}^{l_j}_{k_{l_j}} } |^2 \cdot w_{\gamma}(q)  r^2 dr  \\
 & \les&   \int_{\Sigma^{ext}_{\tau} } E(|I| + 5)  \cdot \frac{\eps }{(1+t+|q|)^{1-c\cdot \eps} \cdot (1+|q|)^{1-2\ga + 4\ga^{\prime}}}  dr \; . \\
\eeaa

We have $1-2\ga + 4\ga^{\prime} > 1 $ that is integrable when $\ga^{\prime} > \frac{\ga}{2}$\,. However, Corollary \ref{Improvedapriorboundontheenergyofgoodcomponentswithaweightgaprime} imposes that $\ga^\prime < 1 + 4\de$\,. Hence, for $  \frac{\ga}{2} < \ga^{\prime} < 1 +4\de $\,, we have the estimate.

Whereas for the square of derivatives, we get 
  \beaa
\notag
| \derm  \Lie_{Z^I} \Phi^{(l_i)}_{{\cal S}^{l_i}_1 \ldots {\cal S}^{l_i}_{k_{l_i}} }  |^2 \cdot | \derm \Lie_{Z^I} \Phi^{(l_j)}_{{\cal S}^{l_j}_1 \ldots {\cal S}^{l_j}_{k_{l_j}} } |^2 \les E(|I| + 5)  \cdot \frac{\eps^2}{(1+t+|q|)^{4-c\cdot \eps} \cdot (1+|q|)^{4+4\ga^{\prime}}} \; . 
\eeaa 
Again, proceeding as in the proof of Lemma \ref{aprioriestimtaeonthebulkfromphiequared}, we get
   \beaa
\notag
&&  \int_{\Sigma^{ext}_{\tau} } \frac{(1+t) }{\eps  } \cdot  |  \derm \Lie_{Z^I} \Phi^{(l_i)}_{{\cal S}^{l_i}_1 \ldots {\cal S}^{l_i}_{k_{l_i}} }  |^2 \cdot | \derm \Lie_{Z^I} \Phi^{(l_j)}_{{\cal S}^{l_j}_1 \ldots {\cal S}^{l_j}_{k_{l_j}} } |^2 \cdot w_{\gamma}(q)  r^2 dr  \\
 & \les&   \int_{\Sigma^{ext}_{\tau} } E(|I| + 5)  \cdot \frac{\eps }{(1+t+|q|)^{1-c\cdot \eps} \cdot (1+|q|)^{3-2\ga + 4\ga^{\prime}}}  dr \; . \\
\eeaa

When $\ga^{\prime} > - \frac{1}{2} + \frac{\ga}{2}$\,, we have $3-2\ga + 4\ga^{\prime} > 1 $ that is integrable. 

 \end{proof}

 \begin{lemma}\label{sourcesforthefullcomponnentsofthesystemunderbootstrapassumptiontogetaprioribound}
  Let $M\leq \eps$\,, $\ga \geq 5 \de $\,, and $0< \de < \frac{1}{2}$. Then, applying Corollary \ref{Improvedapriorboundontheenergyofgoodcomponentswithaweightgaprime} with $\ga^\prime$, such that $ \frac{1}{2} + \ga < \ga^\prime < 1+ 4\de$\,, we get in the exterior region the following estimate on the sources of the full components, for all the fields $l \in \{0, 1, \ldots, N \}$\,,

 \bea
\notag
\int_{0}^{t}  \int_{\Sigma^{ext}_{\tau} }  \frac{(1+\tau )}{\eps} \cdot  | \Lie_{Z^I} \Big( g^{\alpha\beta} \derm_\alpha \derm_\beta  \Phi^{(l)}_{{\cal U}_1 \ldots {\cal U}_{k_l} }   \Big) |^2  r^2  \cdot w_{\ga}(q) dr d\tau &\les&   E(|I| + 5)  \cdot \eps \cdot  (1+t)^{c\cdot \eps} \, ,
        \eea
  where the weight is $w_{\ga}$ that is the one used in our bootstrap assumption \eqref{theboundinthetheoremonEnbyconstantEN}-\eqref{aprioriestimate}.

 \end{lemma}
 
 \begin{proof}
Lemmas \ref{aprioriestimtaeonthebulkfromphiequared} and \ref{aprioriestimatesontheotherbadtermsphiDphiandDphisquare} require that $\ga^{\prime} < 1 + 4\de $\,. They also require the lower bounds $\ga^\prime > \frac{1}{2} + \ga > - \frac{1}{2} + \ga$ and $\ga^\prime > \frac{\ga}{2}$\,. However, for $\ga^{\prime} < 1 + 4\de $\,, the $\max \{\frac{\ga}{2}, \frac{1}{2} + \ga \} = \frac{1}{2} + \ga $\,. Thus, the condition on $\ga^\prime$ is  $ \frac{1}{2} + \ga < \ga^\prime < 1+ 4\de$ in order to estimate the new non-linearities for the bad wave equation \eqref{badtensorialcoupledwaveequation}. The other non-linearities in \eqref{badtensorialcoupledwaveequation} were already estimated in Lemma \ref{decayestimateonthesourcetermforsourcesofgoodcomponents}, which imposed the higher bound on $\ga^\prime$, namely $\ga^\prime < 1+ 4\de$\,.

\end{proof}

  \begin{corollary}\label{updateontheboundofthefulllenergyusingbootstrapandanaxilarlyweightgaprime}
  
        Let $M\leq \eps$\,, $\ga \geq 5 \de $\,, and $0< \de  \leq \frac{1}{4}$\,. Let $\ga^\prime$ be $ \frac{1}{2} + \ga < \ga^\prime < 1+ 4\de$, such that $\E^{\textit{good}\,,\, \gamma^{\prime}}_{k} (0) < \eps^2 $\,. Then, we have
        
   \bea
\E^{\textit{full}\,,\, \gamma}_{k} (t)  \les E(k +3) \cdot \eps \cdot (1 +t)^{c\cdot \eps} \,,
\eea
for $\E^{\textit{full}\,,\, \gamma}_{k}  (0) < \eps^2$\,.
  \end{corollary}
  
\begin{proof}
We insert the estimate from Lemma \ref{sourcesforthefullcomponnentsofthesystemunderbootstrapassumptiontogetaprioribound}, into our decoupled energy estimate from Proposition \ref{Theveryfinal }. We then get the stated bound on the energy by induction on $k$ and by using Gr\"onwall lemma.
\end{proof}

\begin{remark}
The upper bound on $\ga^\prime$\,, namely $\ga^\prime <  1 + 4\de$ is due to Lemma \ref{decayestimateonthesourcetermforsourcesofgoodcomponents}, which with the lower bound on $\ga^\prime$ to estimate the bad terms in Lemma \ref{aprioriestimtaeonthebulkfromphiequared}, altogether imposes that $\ga < \frac{1}{2} + 4\de$\,. However, this upper bound on $\ga$ comes only from the fact that the sources in Lemma \ref{decayestimateonthesourcetermforsourcesofgoodcomponents} involve the metric $\Phi^{(0)}$ where the decay rate is limited, since the actual good estimates are on $\Phi^{(0)} - h^0$, thus $h^0$ hinders the decay rate of $\Phi^{(0)}$\,. The lower bound on $\ga^\prime$ comes the fact that these new non-linearities $ \Phi^{(l_i)}_{{\cal S}^{l_i}_1 \ldots {\cal S}^{l_i}_{k_{l_i}} }    \cdot \Phi^{(l_j)}_{{\cal S}^{l_j}_1 \ldots {\cal S}^{l_j}_{k_{l_j}} }  $ do not decay \`a priori fast enough and therefore they require an additional weight for $ \E^{\textit{good}\,,\, \gamma^\prime}_{k} (0) $ with $\ga^\prime  > \frac{1}{2} + \ga $\,. If these new non-linearities are not present for the metric $\Phi^{(0)} - h^0$\,, there would be for $\ga^\prime$ no lower bound depending on $\ga$\,, and therefore no upper bound on $\ga$ and thus $\ga$ could have be chosen freely. 
\end{remark}

\subsection{Improved \`a priori pointwise decay on the fields}\

      \begin{lemma}\label{upgradedaprioridecayestimateswithmorederivatives}
      
     Let $M\leq \eps$\,, $\ga \geq 5 \de $\,, and $0<\de  \leq \frac{1}{4}$\,. Let $\ga^\prime$ be such that $ \frac{1}{2} + \ga < \ga^\prime < 1+ 4\de$\,, and such that $\E^{\textit{good}\,,\, \gamma^{\prime}}_{k} (0) < \infty $\,. Then, under the bootstrap assumption \eqref{theboundinthetheoremonEnbyconstantEN}-\eqref{aprioriestimate} with the weight $w_\ga$\,, by choosing $ 0 < 5 \de \leq \ga < \frac{1}{2} +4\de $\,, we have for all $|I| \leq K $, in the entire exterior region, the following estimates for all $l \in \{1, \ldots, N \} $\,,
     
   \beaa
\notag
 |\derm   \Lie_{Z^I} \Phi^{(l)} (t,x)  |     + |\derm (  \Lie_{Z^I} \Phi^{(0)} - h^0 ) (t,x)  |    &\les& E ( |I| + 7)  \cdot \frac{\eps }{(1+t+|q|)^{1-c \cdot \eps} \cdot (1+|q|)^{1+\ga}} \,, \\
\notag
   |  \Lie_{Z^I} \Phi^{(l)} (t,x)  |  + |  \Lie_{Z^I} (\Phi^{(0)} - h^0 ) (t,x)  |  &\les&   E ( |I| +7)  \cdot \frac{\eps }{(1+t+|q|)^{1-c \cdot \eps} \cdot (1+|q|)^{ \ga}} \,, \\
 \notag
  |  \rderm  \Lie_{Z^I} \Phi^{(l)}  (t,x)  |  + |\rderm (  \Lie_{Z^I} \Phi^{(0)} - h^0 ) (t,x)  |    &\les& E ( |I| + 8)  \cdot \frac{\eps }{(1+t+|q|)^{2-c \cdot \eps} \cdot (1+|q|)^{\ga} } \,. \\
      \eeaa
  For $l =0 $\,, we have the following estimates in the exterior region,  
  
      \beaa
            \notag
       |\derm \Lie_{Z^I} H   (t,x)  | +  |\derm \Lie_{Z^I} \Phi^{(0)}   (t,x)  |    &\les&  E ( |I| + 7)  \cdot \frac{\eps }{(1+t+|q|)^{1- c \cdot \eps } \cdot (1+|q|)^{1+c\cdot \eps}} \,,  \\
    \notag
      |   \Lie_{Z^I} H (t,x)  |  + |   \Lie_{Z^I} \Phi^{(0)}  (t,x)  | &\les&  E ( |I| + 7) \cdot \frac{\eps }{(1+t+|q|)^{1-c \cdot \eps}  \cdot (1+|q|)^{c\cdot \eps} } \,, \\
                \notag
        |  \rderm \Lie_{Z^I} H (t,x)  |   +  |  \rderm \Lie_{Z^I} \Phi^{(0)}  (t,x)  |  &\les& E ( |I| + 8)  \cdot \frac{\eps }{(1+t+|q|)^{2-c \cdot \eps} \cdot (1+|q|)^{c\cdot \eps} } \,. \\
      \eeaa
      
      \end{lemma}

\begin{proof}
From Corollary \ref{updateontheboundofthefulllenergyusingbootstrapandanaxilarlyweightgaprime}, we have updated the estimates on the  $ \E^{\textit{full}\,,\, \gamma}_{k} (t) \les E(|I| + 5)  \cdot \eps \cdot  (1+t)^{c\cdot \eps} $ with a weight $\ga < \frac{1}{2}  + 4\de$\,. Using the weighted Klainerman-Sobolev inequality with the weigh $w_{\ga}$\,, we can translate this as before into pointwise decay on $ \Lie_{Z^I}\Phi^{(l)}$ and $ \Lie_{Z^I}(\Phi^{(0)} - h^0)$\,, with a loss of two more derivatives. Thus, using Lemma \ref{estimateonthesourcetermsforhzerothesphericallsymmtrpart} that gives the known decay rate of $h^0$, this in turn gives the decay for $ \Lie_{Z^I} \Phi^{(0)}$ and by Remark \ref{estimatesonPhioalsoholdforbigH}, this also gives a decay for $ \Lie_{Z^I} H$\,. 
\end{proof}

        \section{Use up-graded \`a priori bounds to close the bootstrap}\
        
Now we want to use the \`a priori bounds on the energy, with $ 0 < 5 \de \leq \ga < \frac{1}{2} +4\de $\,, that we derived in Proposition \ref{Theveryfinal } using a bootstrap assumption on more derivatives than what we upgraded, with $0 \leq \de < \frac{1}{2}$ chosen small enough\,, as well as to use the pointwise \`a priori bounds on the fields from Lemma \ref{upgradedaprioridecayestimateswithmorederivatives}, to actually get a suitable Gr\"onwall inequality on the energy energy estimate that does not require a bootstrap on more derivatives than what we upgrade for, for a certain $k=K$ large enough, that would allow us to close the bootstrap argument. In fact, once we have upgraded the \`a priori bounds on the decoupled energy and the \`a priori pointwise bounds on fields, it is standard to close the close the bootstrap argument using the structure of the sources of the wave equations. We shall give here the outline of the closure of the argument.

\subsection{The commutator term revisited}\label{commutatortermrevistedwithbootstraponlyonhaldderivatives}\

In Subsection \ref{estimatethecommutatortermusingbootstraponallderivatives}, namely in Lemmas \ref{estimateonthegooddecayingpartofthecommutatorterm} and \ref{estimateonthebadpartofthecommutatortermtobtainaGronwallforcomponents}, we estimated the commutator term using more derivatives that the ones involved in the fields, thus this would not allow us to close a bootstrap argument. In Subsection \ref{estimatethecommutatortermusingbootstraponallderivatives}, it was then allowed since we just wanted to derive \`a priori estimates with a loss of derivatives, and not to close the bootstrap argument yet. However, now that we derived our \`a priori estimates in Lemma \ref{upgradedaprioridecayestimateswithmorederivatives}, we need at this point to close our bootstrap argument to prove that these are actual estimates, if the bootstrap assumption is assumed on a certain number of the Lie derivatives large enough.

For this, we want to estimate the commutator term using only the bootstrap assumption on about half of the derivatives of the fields. We proceed as follows.

\begin{lemma}\label{estimateonthecommutatortermusingtheproductsandusingbootstrap}

For $M\leq \eps$ and $\gamma > \delta > 0 $\,, we have

 \beaa
\notag
&&  \frac{(1+ t )}{\eps} \cdot   | g^{\la\mu}    \derm_{\la}   \derm_{\mu} \Lie_{Z^I} \Phi^{(l)}_{V^l_1 \ldots  V^l_{k_l}}   - \Lie_{Z^I}  ( g^{\la\mu} \derm_{\la}   \derm_{\mu}  \Phi^{(l)}_{V^l_1 \ldots  V^l_{k_l}} )  |^2  \\
  \notag
   &\les&  \sum_{|K| \leq |I| - 1}  | g^{\la\mu}    \derm_{\la}   \derm_{\mu} \Lie_{Z^I} \Phi^{(l)}_{V^l_1 \ldots  V^l_{k_l}}   |^2\\
&& +   E (    \lfloor \frac{|I|}{2} \rfloor   +2)  \cdot \Big(  \sum_{ |K| \leq |I|}   \frac{\eps   }{(1+t+|q|)^{3-  2 \de } \cdot (1+|q|)^{2+2\ga}}  \, \cdot  |    \Lie_{Z^K}  H   |^2  \\
&& + \sum_{ |K| \leq |I| }  \frac{\eps }{(1+t+|q|)^{3- 2 \de } }  \, \cdot | \derm \Lie_{Z^K} \Phi^{(l)}_{V^l_1 \ldots  V^l_{k_l}}  |^2    + \sum_{  |K| \leq |I| }     \frac{\eps \cdot  | \Lie_{Z^{K}} H_{L  L} |^2 }{(1+t+|q|)^{1-  \de } \cdot (1+|q|)^{4+2\gamma}} \Big)  \\
    && + \sum_{  |K| \leq |I| } E (  \lfloor \frac{|I|}{2} \rfloor  +3)  \cdot \frac{ \eps   }{ (1+t+|q|)  }  \, \cdot \sum_{ V^\prime_1 \in {\cal V}^l_{1}, \ldots,  V^\prime_{k_l} \in {\cal V}^l_{k_l} } | \derm  \Lie_{Z^I} \Phi^{(l)}_{V^\prime_1 \ldots  V^\prime_{k_l}}   |^2 \, ,
\eeaa
where
\beaa
{\cal V}^l_{k} = \begin{cases}  & { \cal{T}} \quad \text{when} \quad  V^l_k \in  {\cal{T}} , \\
& {\cal{U}} \quad  \text{when} \quad V^l_k \in  {\cal{U}} \,,\end{cases}
\eeaa
and where $(|K|-1)_+=|K|-1$ if $|K|\geq 1$ and $(|K|-1)_+=0$ if $|K|=0$\,. The same estimate holds for $ \Phi^{(0)} - h^0 $ instead of $\Phi^{(l)}$\,.

\end{lemma}

\begin{proof}

We recall that from \eqref{Theenergyformofthedecoupledcommutatorestimate}, we have for all $l \in \{1, \ldots, N\}$\,,
\beaa
\notag
&&     \frac{ (1+t)}{\eps} \cdot  | g^{\la\mu}    \derm_{\la}   \derm_{\mu} \Lie_{Z^I} \Phi^{(l)}_{V^l_1 \ldots  V^l_{k_l}}   - \Lie_{Z^I}  ( g^{\la\mu} \derm_{\la}   \derm_{\mu}  \Phi^{(l)}_{V^l_1 \ldots  V^l_{k_l}} )  |^2  \\
  \notag
   &\les&  \sum_{|K| \leq |I| -1 }  \frac{ (1+\tau)}{\eps} \cdot  | g^{\la\mu} \cdot \derm_{\la}   \derm_{\mu}  \Lie_{Z^K} \Phi^{(l)}_{V^l_1 \ldots  V^l_{k_l}} |^2 \\
   \notag
&&+  \frac{(1+t)}{\eps \cdot (1+t+|q|)^2}  \cdot \sum_{|K|\leq |I|,}\,\, \sum_{|J|+(|K|-1)_+\le |I|} \,\,\, | \Lie_{Z^{J}} H |^2\, \cdot | \derm \Lie_{Z^K} \Phi^{(l)}_{V^l_1 \ldots  V^l_{k_l}}  |^2 \\
\notag
&& +  \frac{(1+t)}{\eps \cdot   (1+|q|)^2 } \cdot \sum_{|K|\leq |I|,}\,\, \sum_{|J|+(|K|-1)_+\le |I|} \,\,\, | \Lie_{Z^{J}} H_{L  L} |^2\, \cdot  \big( \sum_{ V^\prime_1 \in {\cal V}_{1}, \ldots,  V^\prime_{k_l} \in {\cal V}_{k_l} } | \derm  \Lie_{Z^I} \Phi^{(l)}_{V^\prime_1 \ldots  V^\prime_{k_l}}   |^2 \:  \big)   \; .\\
\eeaa
  
We estimate the first term with the good factor $ \frac{1}{(1+t+|q|)}$\,, as follows,
\beaa
&&  \sum_{|K|\leq |I|,}\,\, \sum_{|J|+(|K|-1)_+\le |I|} \,\,\, | \Lie_{Z^{J}} H |\, \cdot  | \derm \Lie_{Z^K} \Phi^{(l)}_{V^l_1 \ldots  V^l_{k_l}}  | \\
&\leq& \sum_{ |K| \leq   \lfloor \frac{|I|}{2} \rfloor   ,\; |J| \leq |I| }   | \Lie_{Z^{J}} H |\, \cdot  | \derm \Lie_{Z^K} \Phi^{(l)}_{V^l_1 \ldots  V^l_{k_l}}  |+ \sum_{ |J| \leq   \lfloor \frac{|I|}{2} \rfloor   ,\; |K| \leq |I| }  | \Lie_{Z^{J}} H |\, \cdot  | \derm \Lie_{Z^K} \Phi^{(l)}_{V^l_1 \ldots  V^l_{k_l}}  | \;.
\eeaa 
Based on our bootstrap assumption, we have from Lemma \ref{aprioridecayestimates}, that for all $|K| \leq   \lfloor \frac{|I|}{2} \rfloor $\,,
            \beaa
 \notag
 |\derm   \Lie_{Z^K} \Phi^{(l)}   |    &\leq&    E (    \lfloor \frac{|I|}{2} \rfloor   +2)  \cdot \frac{\eps   }{(1+t+|q|)^{1-   \de } \cdot (1+|q|)^{1+\ga}}     \; , 
      \eeaa
      and
                 \beaa
 \notag
 |    \Lie_{Z^K}  H   | &\leq&   E (   \lfloor \frac{|I|}{2} \rfloor + 2)  \cdot \frac{\eps }{(1+t+|q|)^{1-  \de } \cdot (1+|q|)^{\de} }     \; . 
      \eeaa
We insert and we obtain
\beaa
&&  \sum_{|K|\leq |I|,}\,\, \sum_{|J|+(|K|-1)_+\le |I|} \,\,\, | \Lie_{Z^{J}} H |\, \cdot  | \derm \Lie_{Z^K} \Phi^{(l)}_{V^l_1 \ldots  V^l_{k_l}}  |\\
&\leq& E (    \lfloor \frac{|I|}{2} \rfloor   +2)   \cdot  \Big(  \sum_{ |K| \leq |I|}   \frac{\eps   }{(1+t+|q|)^{1-  \de } \cdot (1+|q|)^{1+\ga}}  \, \cdot  |    \Lie_{Z^K}  H   |  \\
&& + \sum_{  |K| \leq |I| }  \frac{\eps }{(1+t+|q|)^{1-   \de } \cdot (1+|q|)^{\de} }  \, \cdot  | \derm \Lie_{Z^K} \Phi^{(l)}_{V^l_1 \ldots  V^l_{k_l}}  | \Big)  \; .
\eeaa 

Thus,
\beaa
&&   \frac{1}{(1+t+|q|)} \cdot \sum_{|K|\leq |I|,}\,\, \sum_{|J|+(|K|-1)_+\le |I|} \,\,\, | \Lie_{Z^{J}} H |\, \cdot  | \derm \Lie_{Z^K} \Phi^{(l)}_{V^l_1 \ldots  V^l_{k_l}}  | \\
&\leq& E (    \lfloor \frac{|I|}{2} \rfloor   +2)  \cdot  \Big(  \sum_{ |K| \leq |I| }   \frac{\eps   }{(1+t+|q|)^{2-  \de } \cdot (1+|q|)^{1+\ga}}  \, \cdot  |    \Lie_{Z^K}  H   |  \\
&& + \sum_{ |K| \leq |I| }  \frac{\eps }{(1+t+|q|)^{2-  \de } \cdot (1+|q|)^{\de} }  \, \cdot  | \derm \Lie_{Z^K} \Phi^{(l)}_{V^l_1 \ldots  V^l_{k_l}}  | \Big)   \; .
\eeaa

We now examine the second term with the bad factor (non-decaying in time $t$), that is $\frac{1}{(1+|q|)}$\,. We have

\beaa
&&  \frac{1}{(1+|q|)}  \cdot  \sum_{|K|\leq |I|,}\,\, \sum_{|J|+(|K|-1)_+\le |I|} \,\,\, | \Lie_{Z^{J}} H_{L  L} |\, \cdot \sum_{ V^\prime_1 \in {\cal V}_{1}, \ldots,  V^\prime_{k_l} \in {\cal V}_{k_l} } | \derm  \Lie_{Z^I} \Phi^{(l)}_{V^\prime_1 \ldots  V^\prime_{k_l}}   | \\
 &\leq& \sum_{ |J| \leq   \lfloor \frac{|I|}{2} \rfloor    ,\; |K| \leq |I| } \frac{1}{(1+|q|)}  \cdot   | \Lie_{Z^{J}} H_{L  L} |\, \cdot  \sum_{ V^\prime_1 \in {\cal V}_{1}, \ldots,  V^\prime_{k_l} \in {\cal V}_{k_l} } | \derm  \Lie_{Z^I} \Phi^{(l)}_{V^\prime_1 \ldots  V^\prime_{k_l}}   |\\
  && + \sum_{|K|\leq \lfloor \frac{|I|}{2} \rfloor +1  ,}\,\, \sum_{|J| \leq |I|}    \frac{1}{(1+|q|)}  \cdot  | \Lie_{Z^{J}} H_{L  L} |\, \cdot \sum_{ V^\prime_1 \in {\cal V}_{1}, \ldots,  V^\prime_{k_l} \in {\cal V}_{k_l} } | \derm  \Lie_{Z^I} \Phi^{(l)}_{V^\prime_1 \ldots  V^\prime_{k_l}}   | \;.
  \eeaa

From Lemma \ref{estimategoodcomponentspotentialandmetric}, we get for all $|J| \leq   \lfloor \frac{|I|}{2} \rfloor  $\,,
          \beaa
        \notag
\frac{1}{(1+|q|)} \cdot |  \Lie_{Z^J} H_{L L} |                    &\les&    E (   \lfloor \frac{|I|}{2} \rfloor  + 3)  \cdot \frac{ \eps   }{ (1+t+|q|) } \; .
       \eeaa

Based on the \`a priori estimates of Lemma \ref{aprioridecayestimates}, we have that for all $|K| \leq   \lfloor \frac{|I|}{2} \rfloor  +1$\;,

                  \beaa
 \notag
  |\derm  ( \Lie_{Z^K} \Phi )   |  &\leq&     E (  \lfloor \frac{|I|}{2} \rfloor  + 2)  \cdot \frac{\eps }{(1+t+|q|)^{1- \de} \cdot  (1+|q|)^{1+\gamma}}    \; .
      \eeaa

Therefore,

          \beaa
&& \frac{1}{(1+|q|)} \cdot   \sum_{|K|\leq |I|,}\,\, \sum_{|J|+(|K|-1)_+\le |I|} \,\,\, | \Lie_{Z^{J}} H_{L  L} |\, \cdot  \sum_{ V^\prime_1 \in {\cal V}_{1}, \ldots,  V^\prime_{k_l} \in {\cal V}_{k_l} } | \derm  \Lie_{Z^I} \Phi^{(l)}_{V^\prime_1 \ldots  V^\prime_{k_l}}   |\\
     &\les& \sum_{  |K| \leq |I| }  E (  \lfloor \frac{|I|}{2} \rfloor  + 3)  \cdot \frac{ \eps   }{ (1+t+|q|)   }  \, \cdot  \sum_{ V^\prime_1 \in {\cal V}_{1}, \ldots,  V^\prime_{k_l} \in {\cal V}_{k_l} } | \derm  \Lie_{Z^I} \Phi^{(l)}_{V^\prime_1 \ldots  V^\prime_{k_l}}   | \\
 && + \sum_{  |K| \leq |I| }  E (  \lfloor \frac{|I|}{2} \rfloor  +2)  \cdot \frac{\eps \cdot  | \Lie_{Z^{K}} H_{L  L} | }{(1+t+|q|)^{1-  \de } \cdot (1+|q|)^{2+\gamma}} \;.
\eeaa

Finally, we obtain

 \beaa
\notag
&& | g^{\la\mu}    \derm_{\la}   \derm_{\mu} \Lie_{Z^I} \Phi^{(l)}_{V^l_1 \ldots  V^l_{k_l}}   - \Lie_{Z^I}  ( g^{\la\mu} \derm_{\la}   \derm_{\mu}  \Phi^{(l)}_{V^l_1 \ldots  V^l_{k_l}} )  | \\
  \notag
   &\les&  \sum_{|K| < |I| }  | g^{\la\mu} \cdot \derm_{\la}   \derm_{\mu}   \Lie_{Z^{K}}  \Phi^{(l)}_{V^l_1 \ldots  V^l_{k_l}}   | \\
&& +   E (    \lfloor \frac{|I|}{2} \rfloor   +2)  \cdot \Big(  \sum_{ |K| \leq |I|}   \frac{\eps   }{(1+t+|q|)^{2-   \de } \cdot (1+|q|)^{1+\ga}}  \, \cdot  |    \Lie_{Z^K}  H   |  \\
&& + \sum_{ |K| \leq |I| }  \frac{\eps }{(1+t+|q|)^{2-  \de } }  \, \cdot | \derm \Lie_{Z^K} \Phi^{(l)}_{V^l_1 \ldots  V^l_{k_l}}  |\Big)   \\
   && + \sum_{  |K| \leq |I| } E (  \lfloor \frac{|I|}{2} \rfloor  +3)  \cdot \frac{ \eps   }{ (1+t+|q|)  }  \, \cdot  \sum_{ V^\prime_1 \in {\cal V}_{1}, \ldots,  V^\prime_{k_l} \in {\cal V}_{k_l} } | \derm  \Lie_{Z^I} \Phi^{(l)}_{V^\prime_1 \ldots  V^\prime_{k_l}}   | \\
 && + \sum_{  |K| \leq |I| }      E (  \lfloor \frac{|I|}{2} \rfloor  +2)  \cdot \frac{\eps \cdot  | \Lie_{Z^{K}} H_{L  L} | }{(1+t+|q|)^{1-  \de } \cdot (1+|q|)^{2+\gamma}} \; .
\eeaa

Similarly for $\Phi^{(l)} - h^0$. Hence, we get the result.

\end{proof}

\subsection{The decoupled energy estimate re-visited}\

With our new estimate on the commutator term in Lemma \ref{estimateonthecommutatortermusingtheproductsandusingbootstrap}, we can proceed as we did earlier to derive Proposition \ref{Theveryfinal }, except that now, since we estimated the commutator term with a bootstrap assumption on only $ \lfloor \frac{|I|}{2} \rfloor  +3$ derivatives\,, namely what is implied by the notation $E(\lfloor \frac{|I|}{2} \rfloor  +3)$, and there are only $E(3)$ involved in Lemma \ref{estimatesonthetermsthatcontainbigHandderivativesofBigHintheenergyestimateinawaythatwecouldconcludeforneq3} to estimate the rest of the terms in the decoupled energy estimate \eqref{decoupledenergyestimatewithoutcommutatorestimate}, we can therefore replace in Proposition \ref{Theveryfinal } the constant $E(k  +3)$ by $E(\lfloor \frac{k}{2} \rfloor  +3)$, which would allow us to close a bootstrap, however, we have the additional terms in the commutator estimate in Lemma \ref{estimateonthecommutatortermusingtheproductsandusingbootstrap} that are 

 \bea\label{termsfromthecomutatorestimatethatinvolveHandthatneedHardytypeinequality}
\notag
&&   \int_{0}^{t}   \int_{\Sigma^{ext}_{\tau} }     \sum_{  |K| \leq |I| }   \Big(  \frac{\eps  \cdot  |    \Lie_{Z^K}  H   |^2 }{(1+t+|q|)^{3-  2 \de } \cdot (1+|q|)^{2+2\ga}}  \,    +      \frac{\eps \cdot  | \Lie_{Z^{K}} H_{L  L} |^2 }{(1+t+|q|)^{1-  \de } \cdot (1+|q|)^{4+2\gamma}} \big) \cdot  w_{\ga^\prime} (q) \cdot r^2 dr d\tau \; . \\
\eea

However, these can be estimated using the following Hardy-type inequality, that we recapitulate from \cite{G6}.
\begin{lemma}\label{Hardytypeinequality}
Assume  $\Phi_{V_1 \ldots  V_{k}} $ is decaying sufficiently fast at spatial infinity for all time $t$\,. Let $\ga^\prime \neq 0$  and $0 \leq a \leq 2$\,, then we have the following estimate in the exterior,

 \bea
\notag
  \int_{\Sigma^{ext}_{\tau} }  \frac{1}{(1+t+r)^{a}} \cdot   \frac{w_{\ga^\prime} (q)}{(1+|q|)^2} \cdot  | \Phi_{V_1 \ldots  V_{k}} |^2 r^2 dr d\si^{2}&\leq& c(\ga^\prime) \cdot  \int_{\Sigma^{ext}_{\tau} }   \frac{ w_{\ga^\prime}(q) }{(1+t+r)^{a}}  \cdot  | \pa_r\Phi_{V_1 \ldots  V_{k}}  |^2    r^2 dr d\si^{2} \; . \\
 \eea

 \end{lemma}

We notice that the term $|    \Lie_{Z^K}  H   |^2 $ in \eqref{termsfromthecomutatorestimatethatinvolveHandthatneedHardytypeinequality} enters with the right decay in $t$\,, so that we can estimate it by $|    \derm \Lie_{Z^K}  H   |^2 $ using our Hardy-type inequality of Lemma \ref{Hardytypeinequality}, and thereafter estimate it by  $ |    \derm \Lie_{Z^K}  \Phi^{(0)}   |^2$  using \eqref{linkbetweenbigHandsamllh}, and finally estimate it by $ | \derm   \Lie_{Z^K}  (\Phi^{(0)} - h^0  ) |^2$ using the decay rate of $h^0$ from Lemma \ref{estimateonthesourcetermsforhzerothesphericallsymmtrpart}. This would lead to terms of the form of a derivative with the correct decay in $t$ in a way that is suitable for a good Gr\"onwall inequality.

Whereas to the term $|    \Lie_{Z^K}  H_{L  L}   |^2$ in \eqref{termsfromthecomutatorestimatethatinvolveHandthatneedHardytypeinequality}, we notice that this does not enter with the right decay in $t$\,, namely $\frac{1}{(1+t+|q|)^{1-  \de } }$\,. However, using the Hardy-type inequality of Lemma \ref{Hardytypeinequality}, it can by estimated by  $|    \derm \Lie_{Z^K}  H_{LL}   |^2 $\,. Now, using our assumption on the metric, namely \eqref{wavecoordinatesestimateonLiederivativesZonmetric}, we can estimate $|    \derm \Lie_{Z^K}  H_{LL}   |^2 $ in its turn by $|    \rderm \Lie_{Z^K}  H   | $\,, and other better decaying terms. These in their turn can be estimated by  $|    \rderm \Lie_{Z^K} \Phi^{(0)}   | $ and therefore by  $|    \rderm (\Lie_{Z^K}  ( \Phi^{(0)} -h^0 ) | $ and by good decaying terms and by $|    \derm \Lie_{Z^K} h^0  | $\,. At the end, we are left to deal with $|    \rderm (\Lie_{Z^K}  ( \Phi^{(0)} -h^0 ) | $ which has a good decay in $|q|$ so that it can be absorbed into the left hand side of the energy estimate \eqref{decoupledenergyestimatewithoutcommutatorestimate} as we did in Subsection \ref{The space-time integral of tangential derivatives}, i.e. to absorb it in the term on the left hand side of the energy estimate, \eqref{firstterminthebulkoftheenergyestimate} and \eqref{secondterminthebulkoftheenergyestimate}, that is of the form 
  \bea\label{Thetangentialcomponentsonthelefthandsideofenergyestimate}
\int_{0}^{t}  \int_{\Sigma^{ext}_{\tau} }   E ( 3 ) \cdot \frac{\eps  }{(1+t+|q|)^{1-  \de } \cdot (1+|q|) }   \cdot | \rderm  ( \Phi^{(0)} -h^0 ) |^2   \cdot  w_{\gamma^{\prime}}(q) \cdot d^{n}x\ \cdot d\tau \; .
\eea
However, when we apply our decoupled energy estimates for the good components, the control in the energy estimate for the tangential derivatives of the type of \eqref{Thetangentialcomponentsonthelefthandsideofenergyestimate} is only for the good components at that stage, whereas the tangential derivatives generated from \eqref{termsfromthecomutatorestimatethatinvolveHandthatneedHardytypeinequality} combined with the Hardy-type inequlity involve the full components, so they cannot be absorbed immediately to the left hand side by applying the decoupled energy estimate for the good components. Instead, they need to be carried with all the time, so that they can be absored to the left hand side at the end when we apply our energy estimates for the full components. Similarly for the term that has a good decaying factor, namely $|    \Lie_{Z^K}  H   |^2 $\,, that does not need to be absorbed, however after applying the Hardy-type inequality and got it estimated by $ | \derm   \Lie_{Z^K}  (\Phi^{(0)} - h^0  ) |^2$\,, it contains the full components and therefore it should be carried with all the time so that it would appear in a Gr\"onwall type inequality for the full components. One should think of these space-time integral as terms that one has to carry with to deal with them at the end when one applies the energy estimate for the full components.

In summary, we get an energy estimate of the type of Proposition \ref{Theveryfinal }, using only $E(\lfloor \frac{|I|}{2} \rfloor  +3)$, i.e. a bootstrap assumption with only $\lfloor \frac{|I|}{2} \rfloor  +3$ derivatives, and with space-time integrals on the right hand side that involve $\rderm H$ with a good factor in $|q|$ so that they will ultimately be absorbed to the left hand side when applying the energy estimate for the full components, as well as a space-time integrals of $\derm H$ with a good decay in $t$\, that will ultimately give a good Gr\"onwall inequality on $ \derm (\Phi^{(0)} - h^0) $. These integrals are terms that we need to carry with till we apply the energy estimates for the full components so that they can be dealt with then to establish a suitable Gr\"onwall type inequality. 

\begin{remark}
This procedure also means that when we apply the energy estimate for the full components, we need to deal with terms that have $\Lie_{Z^K}  h^0 $ alone, because when we had terms of the type  $|    \derm (\Lie_{Z^K} H  ) | $\, and therefore, $|    \derm \Lie_{Z^K}   \Phi^{(0)}  | $\,, on the right hand side of our energy estimate, we expanded $ |\derm\Lie_{Z^K}  \Phi^{(0)} | \leq | \derm \Lie_{Z^K}  (\Phi^{(0)} - h^0 ) |+ | \derm \Lie_{Z^K}  h^0|$\,, so as to control $\derm \Lie_{Z^K}  (\Phi^{(0)} - h^0 )$ and thus, we left ourselves with $ \derm \Lie_{Z^K}  h^0$ alone (we recall that $ \Lie_{Z^K}  \Phi^{(0)}$ has infinite energy). Thus, we would be left with dealing with $\Lie_{Z^K}  h^0 $ separately and this has a limited decay rate as prescribed by Lemma \ref{estimateonthesourcetermsforhzerothesphericallsymmtrpart}. Whereas to the term that is the wave equation on $h^0 $, this was already dealt with in \eqref{Thesourceofthesphericallysymmetricpartinthebulktimesthederivative}.
\end{remark}

\subsection{The source terms as products}\

To control the source terms of the wave equation of the good components, namely the right hand side of \eqref{goodtensorialcoupledwaveequation}, so that we could insert that control in our decoupled energy estimate \eqref{decoupledenergyestimatewithoutcommutatorestimate}, we need to use the improved \`a prior decay estimates from Lemma \ref{upgradedaprioridecayestimateswithmorederivatives}, where we traded the $\de$ in the bootstrap assumption \eqref{theboundinthetheoremonEnbyconstantEN}-\eqref{aprioriestimate}, with an $c\cdot \eps$ instead, yet while requiring the bootstrap assumption on more derivatives than what we upgraded (so we did not close the bootstrap argument). Now we would like to use these improved \`a priori estimates on only half of the derivatives to upgrade the higher order derivatives, so that we can close the argument for a certain number of derivatives high enough. 

For this, when we estimate the derivatives of the sources of the ``good" wave equation \eqref{goodtensorialcoupledwaveequation}, in order to control
\bea\label{sourcetermsofthewaveequationinourenergyestimateafterhavingappliedcommutatorsoLiederivativesofsources}
(1+\tau) \cdot  |  \Lie_{Z^I}  ( g^{\mu\a} \derm_{\mu } \derm_\a  \Phi^{(l)}_{V^l_1 \ldots  V^l_{k_l}} ) |^{2}
\eea
in \eqref{decoupledestimeontgewaveoperatoroftheLierderivativestimesthederivative}, we consider that when differentiating, the other terms cannot get more than half the derivatives. In other words, for a product of terms, say $S_1 \cdot S_2 \cdot S_3 $\,, the product of Lie derivatives $\Lie_{Z^I}$ can be decomposed as follows

  \bea\label{generalinequalityaboutsplittingasumthatisconstrainedtoadduptolenth}
  \notag
 |\Lie_{Z^I} (S_1 \cdot S_2 \cdot S_3 ) | &\leq&  \sum_{|K| + |J| +|M|  \leq |I| } |  (\Lie_{Z^K} S_1) \cdot ( \Lie_{Z^J}  S_2 ) \cdot ( \Lie_{Z^M}  S_3 ) | \\
   \notag
 &\leq& \sum_{|J|\,, |M| \leq  \lfloor \frac{|I|}{2} \rfloor\;,\;  |K| \leq|I|  } |  \Lie_{Z^K} S_1 | \cdot |  \Lie_{Z^J}  S_2  | \cdot  | \Lie_{Z^M}  S_3  |  \\
   \notag
 && + \sum_{|K|\,, |M| \leq  \lfloor \frac{|I|}{2} \rfloor\;,\;  |J| \leq |I| }  | \Lie_{Z^K} S_1 | \cdot  | \Lie_{Z^J}  S_2  | \cdot  | \Lie_{Z^M}  S_3  | \\
 &&   + \sum_{|K|\,, |J| \leq  \lfloor \frac{|I|}{2} \rfloor\;,\;  |M| \leq |I| } |  \Lie_{Z^K} S_1  | \cdot  | \Lie_{Z^J}  S_2  | \cdot  | \Lie_{Z^M}  S_3  |  \; . 
  \eea

This would give in our energy estimate, \eqref{decoupledenergyestimatewithoutcommutatorestimate}-\eqref{decoupledestimeontgewaveoperatoroftheLierderivativestimesthederivative}, for $ \frac{1}{2} + \ga < \ga^\prime < 1+ 4\de$\,, a control on the $w_{\ga^\prime}$-weighted $L^2$-norm of the good components with bulk terms that are space-time integrals that need all to be carried with to the energy estimate on the full components, so that they could either be absorbed to the left hand side of the energy estimate when there are terms with tangential derivatives with good factors in terms of $|q|$\,, or to lead to a Gr\"onwall inequality on the full components when they are terms with good factors in terms of $t$\,.

Now, it's turn, the energy estimate on the full components should be applied with the weight $w_\ga$\,, instead of $w_\ga^\prime$ for the good components, and we do exactly the same procedure for the same terms as before in the sources of the ``bad' equation \eqref{badtensorialcoupledwaveequation}, except for the new additional terms 
 \bea\label{Thenewtermsinthebadwaveequations}
  O(   \Phi^{(l_i)}_{{\cal S}^{l_i}_1 \ldots {\cal S}^{l_i}_{k_{l_i}} } \cdot \Phi^{(l_j)}_{{\cal S}^{l_j}_1 \ldots {\cal S}^{l_j}_{k_{l_j}} }  ) +  O( \Phi^{(l_i)}_{{\cal S}^{l_i}_1 \ldots {\cal S}^{l_i}_{k_{l_i}} } \cdot \derm \Phi^{(l_j)}_{{\cal S}^{l_j}_1 \ldots {\cal S}^{l_j}_{k_{l_j}} } )  + O( \derm  \Phi^{(l_i)}_{{\cal S}^{l_i}_1 \ldots {\cal S}^{l_i}_{k_{l_i}} } \cdot \derm \Phi^{(l_j)}_{{\cal S}^{l_j}_1 \ldots {\cal S}^{l_j}_{k_{l_j}} } ) \,,
  \eea
   where we need to use our decoupled energy estimates on the ``good" components in order to control them. We are going to use the fact that the terms involved in \eqref{Thenewtermsinthebadwaveequations} are good components in the sense that they are all terms that satisfy a good wave equation for which we already applied the prodecure above in \eqref{generalinequalityaboutsplittingasumthatisconstrainedtoadduptolenth} to control their good decoupled energy. Whereas to the other terms in the sources of the full wave equations \eqref{badtensorialcoupledwaveequation}, other than \eqref{Thenewtermsinthebadwaveequations}, they can be estimated using \eqref{generalinequalityaboutsplittingasumthatisconstrainedtoadduptolenth} since this would lead to terms in the correct form.

  In fact, to deal with the terms \eqref{Thenewtermsinthebadwaveequations} in the sources in \eqref{sourcetermsofthewaveequationinourenergyestimateafterhavingappliedcommutatorsoLiederivativesofsources} in our energy estimate \eqref{decoupledestimeontgewaveoperatoroftheLierderivativestimesthederivative}, we proceed as follows. For all $l \in \{0, 1, \ldots, N\}$,
    \bea\label{decompositionofsumforthebadtermssoastousedthedecoupledenergyononeoftheterms}
  \notag
&&  \Lie_{Z^I}  ( \Phi^{(l_i)}_{{\cal S}^{l_i}_1 \ldots {\cal S}^{l_i}_{k_{l_i}} } \cdot  \Phi^{(l_j)}_{{\cal S}^{l_j}_1 \ldots {\cal S}^{l_j}_{k_{l_j}} } ) \\
  \notag
&\les& \sum_{|K| +|J| \leq |I|}   | \Lie_{Z^K}  \Phi^{(l_i)}_{{\cal S}^{l_i}_1 \ldots {\cal S}^{l_i}_{k_{l_i}} }   | \cdot | \Lie_{Z^J}  \Phi^{(l_j)}_{{\cal S}^{l_j}_1 \ldots {\cal S}^{l_j}_{k_{l_j}} }    | \\
\notag
  &\leq&  \sum_{ |J| \leq   \lfloor \frac{|I|}{2} \rfloor   ,\; |K| \leq |I| } | \Lie_{Z^K}   \Phi^{(l_i)}_{{\cal S}^{l_i}_1 \ldots {\cal S}^{l_i}_{k_{l_i}} }    | \cdot |  \Lie_{Z^J}   \Phi^{(l_j)}_{{\cal S}^{l_j}_1 \ldots {\cal S}^{l_j}_{k_{l_j}} }   |  + \sum_{ |J| \leq   \lfloor \frac{|I|}{2} \rfloor   ,\; |K| \leq |I| } | \Lie_{Z^J}   \Phi^{(l_i)}_{{\cal S}^{l_i}_1 \ldots {\cal S}^{l_i}_{k_{l_i}} }     | \cdot | \Lie_{Z^K}   \Phi^{(l_j)}_{{\cal S}^{l_j}_1 \ldots {\cal S}^{l_j}_{k_{l_j}} }   |  \; . \\
\eea

When we want to use energy estimates, we need to distinguish the case when $l = 0$\,, since the energy should then by applied for $\Phi^{(0)} - h^0$\,, unlike $\Phi^{(l)}$\,. However, for all $|J| \leq  \lfloor \frac{|I|}{2} \rfloor $, the weighted Klainerman-Sobolev inequality, for $ \frac{1}{2} + \ga < \ga^\prime < 1+ 4\de$\,, gives a control for $l_i \in \{1, \ldots, N \}$\,,
     \beaa
\notag
&& | \derm \Lie_{Z^J} \Phi^{(l_i)}_{{\cal S}^{l_i}_1 \ldots {\cal S}^{l_i}_{k_{l_i}} } | \cdot (1+t+|q|) \cdot \big[ (1+|q|) \cdot w_{\gamma^{\prime}}(q)\big]^{1/2} \\
&\les&\sum_{|K|\leq  \lfloor \frac{|I|}{2} \rfloor  + 2 } \|\big(w_{\gamma^{\prime}}(q)\big)^{1/2} \derm \Lie_{Z^K} \Phi^{(l_i)}_{{\cal S}^{l_i}_1 \ldots {\cal S}^{l_i}_{k_{l_i}} } (t,\cdot)\|_{L^2 (\Sigma^{ext}_{t} ) } \les \sqrt{ \E^{\textit{good}\,,\, \gamma^{\prime}}_{  \lfloor \frac{|I|}{2} \rfloor  + 2 } (t)}
\eeaa
Thus, for all $|J| \leq  \lfloor \frac{|I|}{2} \rfloor $\,,
   \beaa
\notag
| \pa_{\underline{L}} \Lie_{Z^I} \Phi^{(l_i)}_{{\cal S}^{l_i}_1 \ldots {\cal S}^{l_i}_{k_{l_i}} } | \les \frac{\sqrt{\eps} \cdot \sqrt{ \E^{\textit{good}\,,\, \gamma^{\prime}}_{k} (t)}}{(1+t+|q|) \cdot (1+|q|)^{1+\ga^{\prime}}} \; .\\
\eeaa
By integration along $q$\,, see \cite{G4} for details, we get that for all $|J| \leq  \lfloor \frac{|I|}{2} \rfloor $\,.
   \beaa
\notag
|  \Lie_{Z^J} \Phi^{(l_i)}_{{\cal S}^{l_i}_1 \ldots {\cal S}^{l_i}_{k_{l_i}} }  | \les \frac{\sqrt{\eps} \cdot \sqrt{ \E^{\textit{good}\,,\, \gamma^{\prime}}_{  \lfloor \frac{|I|}{2} \rfloor + 2 } (t)}}{(1+t+|q|) \cdot (1+|q|)^{\ga^{\prime}}} \; . 
\eeaa
Hence,
   \beaa
\notag
&& \sum_{ |J| \leq   \lfloor \frac{|I|}{2} \rfloor   ,\; |K| \leq |I| }   \frac{(1+t) }{\eps  }  \cdot  |  \Lie_{Z^J} \Phi^{(l_i)}_{{\cal S}^{l_i}_1 \ldots {\cal S}^{l_i}_{k_{l_i}} }  |^2 \cdot | \Lie_{Z^K}   \Phi^{(l_j)}_{{\cal S}^{l_j}_1 \ldots {\cal S}^{l_j}_{k_{l_j}} }   |^2 \cdot w_{\gamma}(q)   \\
&\les& \frac{\eps  \cdot  \E^{\textit{good}\,,\, \gamma^{\prime}}_{  \lfloor \frac{|I|}{2} \rfloor  + 2 } (t) }{(1+t+|q|) \cdot (1+|q|)^{2\ga^{\prime}}} \cdot | \Lie_{Z^K}   \Phi^{(l_j)}_{{\cal S}^{l_j}_1 \ldots {\cal S}^{l_j}_{k_{l_j}} }   |^2 \cdot w_{\gamma}(q)\; , 
\eeaa
and therefore,
   \bea\label{badnon-lineartermreadytoapplyHardytypeinequlityusingtheauxilarlyweightgamaprime}
\notag
&& \int_{\Sigma^{ext}_{\tau} }  \frac{(1+\tau ) }{\eps  } \cdot  |  \Lie_{Z^J} \Phi^{(l_i)}_{{\cal S}^{l_i}_1 \ldots {\cal S}^{l_i}_{k_{l_i}} }  |^2 \cdot | \Lie_{Z^K}   \Phi^{(l_j)}_{{\cal S}^{l_j}_1 \ldots {\cal S}^{l_j}_{k_{l_j}} }   |^2 \cdot w_{\gamma}(q)  \cdot   r^2 dr  \\
&\les&  \E^{\textit{good}\,,\, \gamma^{\prime}}_{  \lfloor \frac{|I|}{2} \rfloor  + 2 } (\tau) \cdot  \int_{\Sigma^{ext}_{\tau} }  \frac{\eps  }{(1+\tau+|q|) \cdot (1+|q|)^{2\ga^{\prime}}} \cdot | \Lie_{Z^K}   \Phi^{(l_j)}_{{\cal S}^{l_j}_1 \ldots {\cal S}^{l_j}_{k_{l_j}} }   |^2 \cdot w_{\gamma}(q)  \cdot   r^2 dr\; .
\eea
Now, since $ \frac{1}{2} + \ga < \ga^\prime < 1+ 4\de$\,, we can choose $\ga^\prime \geq 1$\, say by taking $\ga^\prime= 1 +2\de$\,. Therefore, we can apply the Hardy-type inequality of Lemma \ref{Hardytypeinequality} in \eqref{badnon-lineartermreadytoapplyHardytypeinequlityusingtheauxilarlyweightgamaprime} using the fact that $2\ga^\prime \geq 2$\,, and we get 

\beaa
\notag
&& \int_{\Sigma^{ext}_{\tau} }  \frac{(1+\tau ) }{\eps  }  \cdot |  \Lie_{Z^J} \Phi^{(l_i)}_{{\cal S}^{l_i}_1 \ldots {\cal S}^{l_i}_{k_{l_i}} }  |^2 \cdot | \Lie_{Z^K}   \Phi^{(l_j)}_{{\cal S}^{l_j}_1 \ldots {\cal S}^{l_j}_{k_{l_j}} }   |^2 \cdot w_{\gamma}(q)  \cdot   r^2 dr  \\
&\les& \int_{\Sigma^{ext}_{\tau} }  \frac{\eps  \cdot  \E^{\textit{good}\,,\, \gamma^{\prime}}_{  \lfloor \frac{|I|}{2} \rfloor  + 2 } (\tau) }{(1+\tau+|q|) } \cdot | \derm \Lie_{Z^K}   \Phi^{(l_j)}_{{\cal S}^{l_j}_1 \ldots {\cal S}^{l_j}_{k_{l_j}} }   |^2 \cdot w_{\gamma}(q)  \cdot   r^2 dr\; \\
&\les&  \int_{0}^{t} \frac{\eps \cdot   \E^{\textit{good}\,,\, \gamma^{\prime}}_{  \lfloor \frac{|I|}{2} \rfloor  + 2 } (\tau) }{(1+\tau) }   \cdot  \int_{\Sigma^{ext}_{\tau} }  \cdot | \derm \Lie_{Z^K}   \Phi^{(l_j)}_{{\cal S}^{l_j}_1 \ldots {\cal S}^{l_j}_{k_{l_j}} }   |^2 \cdot w_{\gamma}(q)  \cdot   r^2 dr d\tau\; .
\eeaa
Therefore, we get for all $l_i\,, l_j \in \{1, \ldots, N \}$\,,
\bea\label{estimateonthebadtermsusingtheenergyofgoodcomponents}
\notag
&& \int_{\Sigma^{ext}_{\tau} }  \frac{(1+\tau ) }{\eps  }  \cdot  |  \Lie_{Z^J} \Phi^{(l_i)}_{{\cal S}^{l_i}_1 \ldots {\cal S}^{l_i}_{k_{l_i}} }  |^2 \cdot | \Lie_{Z^K}   \Phi^{(l_j)}_{{\cal S}^{l_j}_1 \ldots {\cal S}^{l_j}_{k_{l_j}} }   |^2 \cdot w_{\gamma}(q)  \cdot   r^2 dr  \\
&\les&  \int_{0}^{t} \frac{\eps   }{(1+\tau) } \cdot  \E^{\textit{good}\,,\, \gamma^{\prime}}_{  \lfloor \frac{|I|}{2} \rfloor  + 2 } (\tau) \cdot  \E^{\textit{good}\,,\, \gamma}_{  |I| } (\tau)  \cdot  d\tau\; .
\eea
Similarly, we get the same estimate for the other terms in \eqref{Thenewtermsinthebadwaveequations}, which are better since a derivative is already there in the product. Whereas for the case of $l  = 0$\,, we write $\Phi^{(0)} = \Phi^{(0)} - h^0 + h^0  $ and we carry the energy on $\Phi^{(0)} - h^0$ and we treat the terms with $h^0$ using the already known decay rate from Lemma \ref{estimateonthesourcetermsforhzerothesphericallsymmtrpart}\,.

Consequently, we estimated all the ``bad" terms using only $\E^{\textit{good}\,,\, \gamma^{\prime}}_{  \lfloor \frac{|I|}{2} \rfloor  + 2 } (t)$ and  $\E^{\textit{good}\,,\, \gamma}_{  |I| } (t)$\,. However, since these are energies on good components that satisfy a good wave equation with good sources \eqref{sourcetermsofthewaveequationinourenergyestimateafterhavingappliedcommutatorsoLiederivativesofsources}, then as discussed above, using \eqref{generalinequalityaboutsplittingasumthatisconstrainedtoadduptolenth}, our decoupled energy estimate \eqref{decoupledenergyestimatewithoutcommutatorestimate}-\eqref{decoupledestimeontgewaveoperatoroftheLierderivativestimesthederivative} imply that these energies can be controlled by space-time integrals of the fields with factors that would allow us to get a suitable Gr\"onwall inequality \eqref{Theinequalityontheenergytobeusedtoapplygronwallrecursively} on the energy of the whole system for a certain number of derivatives large enough.

      \subsection{The closure of the bootstrap argument}\
      
Let us recapitulate what we have done. We ran a bootstrap argument \eqref{theboundinthetheoremonEnbyconstantEN}-\eqref{aprioriestimate} on the whole coupled system of tensorial non-linear wave equations \eqref{badtensorialcoupledwaveequation}-\eqref{goodtensorialcoupledwaveequation}, with an energy with a weight $w_{\ga}$\,, with $\ga$ such that $ 0 < 5 \de \leq \ga < \frac{1}{2} +2\de $\,, which ensures that $\frac{1}{2} + \ga <  \frac{1}{2} +2\de <    \frac{1}{2} +4\de $\,, so that we could apply Proposition \ref{Improvesaprioriboundsonenergywithextraweightforgoodcomponentswithboundsongamma}. We then choose $\ga^\prime = 1 + 2 \de$ and therefore, we have $  \frac{1}{2} + \ga < \ga^\prime = 1 + 2\de(\ga) <  \frac{1}{2} +4\de$\,, and thanks to our decoupled energy estimates, we can via Proposition \ref{Improvesaprioriboundsonenergywithextraweightforgoodcomponentswithboundsongamma} improve the bound on $\E^{\textit{good}\,,\, \gamma^{\prime}}_{k} (t) $ in \eqref{notupgradedyetbutimprovedboundonenergyofgoodcomponentwithgamprime} and on $\E^{\textit{full}\,,\, \gamma}_{k} (t)$ in \eqref{notupgradedyetbutimprovedboundonenergyoffullcomponent}, where we traded the $\de$ with an $\eps$ small enough, and we obtained improved pointwise \`a priori bounds in Lemma \ref{upgradedaprioridecayestimateswithmorederivatives}. However, this was carried out with a loss of derivatives, i.e. we had to use the bootstrap assumption \eqref{theboundinthetheoremonEnbyconstantEN}-\eqref{aprioriestimate} on more derivatives than what we upgraded for. In order to close the bootstrap, we need to use these \`a priori bounds on no more than half of the derivatives to establish a suitable Gr\"onwall inequality that allows us to upgrade the bound on the energy of the whole system with the weight $w_\ga$\,. To achieve this, we use the fact that this can be done to estimate $\E^{\textit{good}\,,\, \gamma^{\prime}}_{k} (t)$\,, since the structure of the sources of the good wave equation allows us to get a space-time integral on the fields that is suitable for the desired Gr\"onwall-type inequality while using the improved decay on only half of the derivatives as in \eqref{generalinequalityaboutsplittingasumthatisconstrainedtoadduptolenth}. We then use this control on the energy of the good components to control the new bad non-linear terms for the full system \eqref{Thenewtermsinthebadwaveequations}, using the fact that these non-linearities involve only good components for which we have already controlled suitably the energy. Thus, this allows us to control the energy for the system $\E^{\textit{full}\,,\, \gamma}_{k} (t)$\,, with the weight $w_\ga$\,, in a way that gives us a suitable Gr\"onwall inequality of the type of \eqref{Theinequalityontheenergytobeusedtoapplygronwallrecursively} that allows us to upgrade the bound on the energy for $\eps $ small enough and for a certain number of derivatives high enough.
                                
Finally, for any $0 \leq \de < \frac{1}{4}$\,, choose $\ga$ such that $ 0 < 5 \de \leq \ga < \frac{1}{2} +2\de $ and $\ga^\prime = 1 + 2\de$\,. These conditions on $\ga$ and $\ga ^\prime$ satisfy the conditions of Proposition \ref{Improvesaprioriboundsonenergywithextraweightforgoodcomponentswithboundsongamma}, and allow us to upgrade the \`a priori bounds on the energy as well as the pointwise bounds on the fields, and to use these to get a suitable control on the energy of the good components and thereafter on the energy of the full system, using the fact that the bad non-linearities involve only good components.\\

\textbf{Conflict of Interest and Funding Statement:} The work in this manuscript was carried out by the author while receiving funding from the Beijing Institute of Mathematical Sciences and Applications (BIMSA) in China.

\textbf{Data Availability Statement:} No datasets were generated or analysed during the current study in this manuscript.

\end{document}